\PassOptionsToPackage{capitalize}{cleveref}
\documentclass[DIV=15, headings=small,10pt,british,fleqn]{scrartcl}
\pdfoutput=1
\title{\large Clifford-symmetric polynomials}
\author{Fabian Lenzen\thanks{fabian.lenzen@tum.de}}
\date{\today}

\usepackage{babel}
\usepackage[utf8]{inputenc}
\usepackage[T1]{fontenc}
\usepackage{csquotes}
\usepackage{chars}
\usepackage{microtype}

\usepackage{booktabs,array,multirow,blkarray,tabularx}

\usepackage[shortcuts]{extdash}					

\usepackage{afterpage}
\usepackage[all, defaultlines=3]{nowidow}
\usepackage{trimclip}
\usepackage[backend=biber,style=alphabetic,citetracker=true, giveninits=true]{biblatex}

\usepackage[hidelinks,breaklinks=true,unicode]{hyperref}
\usepackage[noabbrev]{cleveref}

\renewcommand{\bfseries}{\fontseries{b}\selectfont}
\setkomafont{disposition}{\fontseries{b}\selectfont\boldmath}
\setkomafont{part}{\Large}
\addtokomafont{partentry}{\normalsize}
\setkomafont{sectionentry}{}
\setkomafont{partnumber}{\Large}
\setkomafont{section}{\large}
\setkomafont{subsection}{\normalsize}
\setkomafont{subsubsection}{\normalsize}
\setkomafont{caption}{\small}
\setkomafont{captionlabel}{\itshape}
\setkomafont{title}{\large\bfseries}
\setkomafont{author}{\normalsize}
\setkomafont{date}{\normalsize}

\setcapindent{0em}            

\SetSymbolFont{operators}{bold}{OT1}{lmr}{b}{n}
\SetMathAlphabet{\mathbf}{normal}{OT1}{lmr}{b}{n}
\SetMathAlphabet{\mathbf}{bold}{OT1}{lmr}{bx}{n}
\SetMathAlphabet{\mathsf}{bold}{OT1}{lmss}{b}{n}
\SetMathAlphabet{\mathit}{bold}{OT1}{lmr}{b}{it}

\DeclareFontFamily{U}{mathc}{}
\DeclareFontShape{U}{mathc}{m}{it}{<->s*[1.03] mathc10}{}
\DeclareMathAlphabet{\mathcalx}{U}{mathc}{m}{it}

\let\include\input

\RedeclareSectionCommand[
  beforeskip=-8.5ex plus -1ex minus -.5ex,
  afterskip=2.3ex plus .2ex
]{section}

\renewcommand*{\bibfont}{\small}                             
\DeclareBibliographyAlias{preprint}{misc}                    
\DeclareFieldFormat[preprint]{title}{\mkbibquote{#1\isdot}}  
\DeclareFieldFormat{eprint:HAL}{\textsc{hal}:~\href{http://hal.archives-ouvertes.fr/#1}{#1}}
\DeclareFieldFormat{eprint:EuDML}{\textsc{EuDML}:~\href{http://eudml.org/doc/#1}{#1}}
\newcolumntype{Q}[1]{>{\centering\arraybackslash$}p{#1}<{$}} 
\newcolumntype{C}{>{$}Sc<{$}}                                

\usepackage[inline, shortlabels]{enumitem}
\setlist[itemize]{itemsep=0pt}

\usepackage{amsthm,thmtools}
\newtheoremstyle{theorem}{}{}{}{}{\itshape}{.}{ }{}
\theoremstyle{theorem}

\newtheorem*{theorem*}{Theorem}
\newtheorem{theorem}{Theorem}[section]
\newtheorem{lemma}[theorem]{Lemma}
\newtheorem*{proposition*}{Proposition}
\newtheorem{proposition}[theorem]{Proposition}
\newtheorem{corollary}[theorem]{Corollary}

\newtheorem*{definition*}{Definition}
\newtheorem{definition}[theorem]{Definition}
\newtheorem{definition-lemma}[theorem]{Definition and Lemma}

\newtheorem{example}[theorem]{Example}

\newtheorem{remark}[theorem]{Remark}

\newtheorem*{claim}{Claim}

\patchcmd{\thmhead}{(#3)}{#3}{}{}

\crefname{lemma}{Lemma}{Lemmas}
\crefname{observation}{Observation}{Observations}
\crefname{notation}{Notation}{Notations}
\crefname{caveat}{Caveat}{Caveats}
\crefformat{footnote}{#2\footnotemark[#1]#3}
\crefformat{equation}{(#2#1#3)}
\crefrangeformat{equation}{(#3#1#4--#5#2#6)}
\crefmultiformat{equation}{(#2#1#3)}{ and~(#2#1#3)}{, (#2#1#3)}{ and~(#2#1#3)}

\newlist{thmlist}{enumerate}{1}
\setlist[thmlist]{label=\textup{(\roman*)}, ref=\thetheorem.\textup{(\roman*)}, noitemsep}

\newlist{prooflist}{enumerate}{1}
\setlist[prooflist]{label=\textup{(\roman{prooflisti})}, ref=(\roman{prooflisti}),noitemsep}
\crefname{prooflisti}{part}{parts}

\newlist{thmlist*}{enumerate*}{1}
\setlist[thmlist*]{label=\textup{(\roman{thmlist*i})}, ref=\thetheorem.(\roman{thmlist*i}),noitemsep}

\makeatletter
\newcommand\strongnopagebreak{\nopagebreak\@afterheading} 
\makeatother

\addtotheorempostheadhook[theorem]{\crefalias{thmlisti}{theorem}}
\addtotheorempostheadhook[remark]{\crefalias{thmlisti}{remark}}
\addtotheorempostheadhook[lemma]{\crefalias{thmlisti}{lemma}}
\addtotheorempostheadhook[recall]{\crefalias{thmlisti}{recall}}
\addtotheorempostheadhook[definition]{\crefalias{thmlisti}{definition}}
\addtotheorempostheadhook[example]{\crefalias{thmlisti}{example}}
\addtotheorempostheadhook[fact]{\crefalias{thmlisti}{fact}}
\addtotheorempostheadhook[proposition]{\crefalias{thmlisti}{proposition}}
\addtotheorempostheadhook[notation]{\crefalias{thmlisti}{notation}}

\numberwithin{equation}{section}
\numberwithin{table}{section}
\numberwithin{figure}{section}

\usepackage{atbegshi}
\newcommand\showtimer{%
	\pdfresettimer}
\AtBeginDocument{\showtimer}
\AtBeginShipout {\showtimer}

\usepackage{chars}
\usepackage{amsmath,amsthm,amssymb,amsfonts,mathtools,thmtools,needspace}
\usepackage[only,mapsfrom,longmapsfrom,mapsfromchar]{stmaryrd}
\allowdisplaybreaks

\usepackage{bm,oubraces}
\usepackage[math]{cellspace}
\usepackage{ifthen,xifthen,xargs,calc}

\makeatletter
\newcommand\ie{i.\,e\@ifnextchar.{}{.\@}}
\newcommand\cf{cf\@ifnextchar.{}{.\@}}
\newcommand\eg{e.\,g\@ifnextchar.{}{.\@}}
\newcommand\wrt{w.\,r.\,t\@ifnextchar.{}{.\@}}
\newcommand{\resp}{resp\@ifnextchar.{}{.\@}}
\newcommand{\Wlog}{W.\,l.\,o.\,g\@ifnextchar.{}{.\@}}
\newcommand{\wlofg}{w.\,l.\,o.\,g\@ifnextchar.{}{.\@}}
\newcommand{\ses}{s.\,e.\,s\@ifnextchar.{}{.\@}}
\newcommand{\st}{s.\,t\@ifnextchar.{}{.\@}}

\DeclareMathOperator{\Hom}{Hom}

\DeclareMathOperator{\End}{End}
\newcommand{\id}{\mathrm{id}}%
\DeclareMathOperator{\im}{im}

\DeclareMathOperator{\rk}{rk}
\DeclareMathOperator{\ord}{ord}

\newcommand{\asteq}{\mathbin{\smash{\stackrel{*}{=}}}}
\newcommand{\isom}{\cong}
\newcommand{\ldot}{\mathbin{.}}

\newcommand{\conn}{\mathrel{-}}
\newcommand{\disc}{\not\conn}

\newcommand{\NH}{\operatorname{\mathsf{NH}}}
\newcommand{\NC}{\operatorname{\mathsf{NC}}}
\newcommand{\Pol}{\operatorname{\mathrm{Pol}}}
\newcommand{\SPol}{\mathop{\Lambda\operator@font Pol}\nolimits}

\newcommand{\ONH}{\operatorname{\mathsf{ONH}}}

\newcommand\pty[1]{\lvert#1\rvert}
\newcommand{\cl}{\mathfrak{c}}
\newcommand{\Cl}{\mathfrak{C}}

\newcommand{\sH}{\operatorname{\tilde{\mathsf{H}}}}
\newcommand{\sPol}{\operatorname{\widetilde{\mathrm{Pol}}}}
\newcommand{\delC}{\mathfrak d}
\newcommand{\epsC}{\mathfrak e}
\newcommand{\NHC}{\operatorname{\mathsf{NH}\mathfrak C}}
\newcommand{\NCC}{\operatorname{\mathsf{NC}\mathfrak C}}
\newcommand{\HC}{\operatorname{\mathsf{H}\mathfrak C}}
\newcommand{\PolC}{\operatorname{\mathrm{Pol}\mathfrak C}}
\newcommand{\SPolC}{\mathop{\Lambda{\operator@font Pol}\mathfrak C}\nolimits}

\newcommand{\HGrC}{\mathop{\mathit{H}\mathfrak C}\nolimits}
\newcommand{\Gr}{\operatorname{Gr}}
\newcommand{\Fl}{\operatorname{Fl}}

\newcommand{\sZ}{\mathbf{Z}/2\mathbf{Z}}
\newcommand{\ksVS}{k\mhyphen\mathrm{sVect}}

\newcommand{\whitelozenge}{%
	\ooalign{$\textcolor{white}{\blacklozenge}$\cr${\lozenge}$}%
}
\newcommand{\whitebullet}{%
	\ooalign{$\textcolor{white}{\bullet}$\cr${\circ}$}%
}

\newcommand{\subalign}[1]{%
	\vcenter{%
		\Let@ \restore@math@cr \default@tag
		\baselineskip\fontdimen10 \scriptfont\tw@
		\advance\baselineskip\fontdimen12 \scriptfont\tw@
		\lineskip\thr@@\fontdimen8 \scriptfont\thr@@
		\lineskiplimit\lineskip
		\ialign{\hfil$\m@th\scriptstyle##$&&$\m@th\scriptstyle##$\hfil\crcr
			#1\crcr
		}%
	}
}
\newcommand\restrict[2]{{%
	\left.\kern-\nulldelimiterspace %
	#1 %
	\right|_{#2} %
}}

\newcommand{\xmapsfrom}[2][]{%
	\ext@arrow3095\leftarrowfill@{#1}{#2}\mapsfromchar
}
\newcommand{\xtofrom}[2][]{%
	\mathrel{%
		\raise.4ex\hbox{%
			$\ext@arrow 0395\rightarrowfill@{\phantom{#1}}{#2}$}%
		\setbox0=\hbox{%
			$\ext@arrow 0395\leftarrowfill@{#1}{\phantom{#2}}$}%
		\kern-\wd0 \lower.4ex\box0%
	}%
}
\newcommand{\mapstofrom}{%
	\mathrel{\ooalign{%
		\raise.4ex\hbox{$\mapsto$}\cr%
		\raise-.4ex\hbox{$\mapsfrom$}\cr}%
	}%
}

\DeclareRobustCommand{\longinto}{%
  \mathrel{\mathpalette\shook@rightarrow\relax}%
}

\newcommand{\shook@rightarrow}[2]{%
  \sbox\z@{$\m@th#1\lhook$}%
  {\lhook}%
  \kern-.25\wd\z@
  {\smash{\clipbox{{.75\wd\z@} 0pt 0pt {-\width}}{$\m@th#1\longrightarrow$}}}%
}

\newcommand*{\relrelbarsep}{.386ex}
\newcommand*{\relrelbar}{%
  \mathrel{%
    \mathpalette\@relrelbar\relrelbarsep
  }%
}
\newcommand*{\@relrelbar}[2]{%
  \raise#2\hbox to 0pt{$\m@th#1\relbar$\hss}%
  \lower#2\hbox{$\m@th#1\relbar$}%
}
\providecommand*{\rightrightarrowsfill@}{%
  \arrowfill@\relrelbar\relrelbar\rightrightarrows
}
\providecommand*{\xrightrightarrows}[2][]{%
  \ext@arrow 0359\rightrightarrowsfill@{#1}{#2}%
}

\let\into\hookrightarrow

\let\xto\xrightarrow

\let\longto\longrightarrow

\DeclareRobustCommand\vdots{%
  \mathpalette\@vdots{}%
}
\newcommand*{\@vdots}[2]{%
	\sbox0{$#1\cdotp\cdotp\cdotp\m@th$}%
	\sbox2{$#1.\m@th$}%
	\vbox{%
	\dimen@=\wd0 %
	\advance\dimen@ -3\ht2 %
	\kern.5\dimen@
	\dimen@=\wd2 %
	\advance\dimen@ -\ht2 %
	\dimen2=\wd0 %
	\advance\dimen2 -\dimen@
	\vbox to \dimen2{%
		\offinterlineskip
		\copy2 \vfill\copy2 \vfill\copy2 %
	}%
	}%
}
\DeclareRobustCommand\ddots{%
	\mathinner{%
		\mathpalette\@ddots{}%
		\mkern\thinmuskip
	}%
}
\newcommand*{\@ddots}[2]{%
	\sbox0{$#1\cdotp\cdotp\cdotp\m@th$}%
	\sbox2{$#1.\m@th$}%
	\vbox{%
		\dimen@=\wd0 %
		\advance\dimen@ -3\ht2 %
		\kern.5\dimen@
		\dimen@=\wd2 %
		\advance\dimen@ -\ht2 %
		\dimen2=\wd0 %
		\advance\dimen2 -\dimen@
		\vbox to \dimen2{%
			\offinterlineskip
			\hbox{$#1\mathpunct{.}\m@th$}%
			\vfill
			\hbox{$#1\mathpunct{\kern\wd2}\mathpunct{.}\m@th$}%
			\vfill
			\hbox{$#1\mathpunct{\kern\wd2}\mathpunct{\kern\wd2}\mathpunct{.}\m@th$}%
		}%
	}%
}

\let\term\emph
\newcommand{\proofsection}[1]{\par\vspace\topsep\noindent{}---\textit{#1}:}
\newcommand{\within}[2]{\mathrel{\mathmakebox[\widthof{$#1$}][r]{#2}}}
\newcommand{\eqsp}{\mathrel{\phantom{=}}}
\newcommand{\quot}[2]{\left.\raisebox{.2em}{$#1$}\middle/\raisebox{-.2em}{$#2$}\right.}

\newcommand\casesGap{\hphantom{\left\{\rule{0pt}{1cm}\right. \kern-\nulldelimiterspace}}
\newlength{\algnRef}
\newlength{\algnRefB}

\mathchardef\mhyphen="2D
\newcommandx{\Mod}[3][1={},3={}]{%
	\ifthenelse{\isempty{#2}}{}{#2\mhyphen}%
	{\operator@font#1Mod}%
	\ifthenelse{\isempty{#3}}{}{\mhyphen#3}%
}

\newcommand{\Proj}[1]{%
	#1\mhyphen{\operator@font Proj}
}
\newcommand{\VS}[1]{%
	#1\mhyphen{\operator@font Vect}
}

\newcommand{\p}{\mathfrak{p}}

\let\epsilon\varepsilon

\makeatother

\usepackage{tikz}[2015/08/07]
\usepackage{tikz-cd}
\usetikzlibrary{
	calc,
	fit,
	shapes.geometric,
	decorations.pathreplacing,
	angles,quotes,
	backgrounds,
	positioning,
	calligraphy,
	intersections,
	patterns,
	quotes
}

\tikzset{
	|/.tip={Bar[width=.8ex,round]},
	back line/.style={densely dotted},
	cross line/.style={%
		preaction={%
			draw=white,
			-, 
			line width=6pt
		}
	},
	brace tip/.style = {
		sloped, allow upside down, yshift=+1ex
	},
	brace tip'/.style = {
		sloped, allow upside down, yshift=-1ex
	},
	brace/.style={
		decoration={calligraphic brace, amplitude=.8ex},
		decorate,
		line width=.3ex
	},
	brace'/.style={
		decoration={calligraphic brace, amplitude=.8ex, mirror},
		decorate,
		line width=.3ex
	},
	braced box/.style = {
		rectangle,
		inner sep=0mm,
		append after command = {
			\pgfextra{
				\draw[brace] (\tikzlastnode.north east) to node[xshift=1ex] (#1){} (\tikzlastnode.south east);
			}
		}
	},
	braced box'/.style = {
		rectangle,
		inner sep=0mm,
		append after command = {
			\pgfextra{
				\draw[brace] (\tikzlastnode.south west) to node[xshift=-1ex] (#1){} (\tikzlastnode.north west);
			}
		}
	},
	braced box''/.style = {
		rectangle,
		inner sep=0mm,
		append after command = {
			\pgfextra{
				\draw[brace] (\tikzlastnode.north west) to node[yshift=1ex] (#1){} (\tikzlastnode.north east);
			}
		}
	},
	braced box'''/.style = {
		rectangle,
		inner sep=0mm,
		append after command = {
			\pgfextra{
				\draw[brace] (\tikzlastnode.south east) to node[yshift=-1ex] (#1){} (\tikzlastnode.south west);
			}
		}
	},
	mth/.style = {
		commutative diagrams/every diagram,
		every path/.style={
			commutative diagrams/.cd, every arrow, every label, arrows=dash, tips=on proper draw
		},
		matrix of nodes/.append style={
			nodes={font=\normalsize}
		}
	}
}
\tikzset{
	diag/.style={
		execute at begin picture={%
			\let\oldblacklozenge\blacklozenge%
			\let\oldlozenge\lozenge%
			\newcommand{\blacklozengeX}{\scalebox{.65}{$\oldblacklozenge$}}%
			\newcommand{\lozengeX}{\scalebox{.65}{$\oldlozenge$}}%
			\newcommand{\decofactor}{1}
			\let\oldbullet\bullet%
			\let\oldcirc\circ
			\renewcommand{\blacklozenge}{\scalebox{\decofactor}{$\blacklozengeX$}}%
			\renewcommand{\lozenge}{\scalebox{\decofactor}{$\lozengeX$}}%
			\renewcommand{\bullet}{\scalebox{\decofactor}{$\oldbullet$}}%
			\renewcommand{\circ}{\scalebox{\decofactor}{$\oldcirc$}}%
		},
		x=5mm,
		y=5mm,
		baseline=-.5ex,
		curved/.style={looseness=1.7},
		rot/.style={sloped},
		deco/.style={
			font=\normalsize,
			every deco,
			inner sep=0pt,
			label distance=0pt
		},
		every node/.append style={
			inner sep=0pt,
			font=\scriptsize
		},
		below/.default={
			1.2mm
		},
		above/.default={
			1.2mm
		},
		every label/.append style={
			font=\scriptsize,
			inner sep=0,
			label distance=1mm
		},
		smaller/.style={
			scale=.8,
			every deco/.append style={
				execute at begin node={
					\renewcommand{\decofactor}{.9}
				}
			}	
		},
		inline/.style={
			x=1ex, y=1ex,
			every deco/.append style={
				execute at begin node={
				}
			},
			left/.default={.5ex},
			right/.default={.5ex}
		},
		every diag,
		every inline diag,
	},
	every deco/.style={},
	every diag/.style={},
	every inline diag/.style={}
}
\tikzset{
  bezier/controls/.code args={(#1) and (#2)}{
    \def\mystartcontrol{#1}
    \def\mytargetcontrol{#2}
  },
  bezier/limit/.store in=\mylimit,
  bezier/limit=1cm,
  bezier/.code={
    \tikzset{bezier/.cd,#1}
    \tikzset{
      to path={
        let
        \p0=(\tikztostart),    \p1=(\mystartcontrol),
        \p2=(\mytargetcontrol), \p3=(\tikztotarget),
        \n0={veclen(\x1-\x0,\y1-\y0)},
        \n1={veclen(\x3-\x2,\y3-\y2)},
        \n2={\mylimit}
        in  \pgfextra{
          \pgfmathtruncatemacro\ok{max((\n0>\n2),(\n1>\n2))}
        }
        \ifnum\ok=1 %
        let
        \p{01}=($(\p0)!.5!(\p1)$), \p{12}=($(\p1)!.5!(\p2)$), \p{23}=($(\p2)!.5!(\p3)$),
        \p{0112}=($(\p{01})!.5!(\p{12})$), \p{1223}=($(\p{12})!.5!(\p{23})$),
        \p{01121223}=($(\p{0112})!.5!(\p{1223})$)
        in
        to[bezier={controls={(\p{01}) and (\p{0112})}}]
        (\p{01121223})
        to[bezier={controls={(\p{1223}) and (\p{23})}}]
        (\p3)
        \else
        [overlay=false] .. controls (\p1) and (\p2) ..  (\p3) [overlay=true]
        \fi
      },
    },
  },
  limit bb/.style n args={2}{
    overlay,
    decorate,
    decoration={
      show path construction,
      moveto code={},
      lineto code={\path[#2] (\tikzinputsegmentfirst) -- (\tikzinputsegmentlast);},
      curveto code={
        \path[#2]
        (\tikzinputsegmentfirst)
        to[bezier={limit=#1,controls={(\tikzinputsegmentsupporta) and (\tikzinputsegmentsupportb)}}]
        (\tikzinputsegmentlast);
      },
      closepath code={\path[#2] (\tikzinputsegmentfirst) -- (\tikzinputsegmentlast);},
    },
  },
  limit bb/.default={1mm}{draw},
}

\usepackage[]{ifdraft}
\tikzcdset{
	slightly cramped/.style = {
		every matrix/.append style={inner sep=+-.2em},
		every cell/.append style={inner sep=+.3em}
	}
}
\ifdraft{
	
	\tikzset{
		limit bb/.style n args={2}{},
		out/.style={},
		in/.style={},
		in looseness/.style={},
		out looseness/.style={},
		looseness/.style={},
		brace/.style={},
		brace'/.style={}
	}
}{}

\RedeclareSectionCommand[
  beforeskip=-2.3ex plus -1ex minus -.5ex,
  afterskip=-.5em,
  font={\normalfont\bfseries\scshape\boldmath}
]{section}
\RedeclareSectionCommand[
  beforeskip=-2.3ex plus -1ex minus -.5ex,
  afterskip=-.5em,
  font={\normalfont\bfseries\boldmath}
]{subsection}

\setkomafont{title}{\large\bfseries}
\setkomafont{author}{\normalsize}
\setkomafont{date}{\normalsize}

\tikzset{
	to start/.style={pos=0.375},
	to end/.style={pos=0.625}
}
\begin{document}
\maketitle

\begin{abstract}
	\noindent
	{\bfseries Abstract}\enspace
	Based on the NilHecke algebra $\NH_n$,
	the odd NilHecke algebra developed by Ellis, Khovanov and Lauda,
	and on Kang, Kashiwara and Tsuchioka's quiver Hecke superalgebra,
	we develop the Clifford Hecke superalgebra $\NHC_n$
	as another super-algebraic analogue of $\NH_n$.
	We show that there is a notion of symmetric polynomials fitting in this picture,
	and we prove that these are generated by an appropriate analogue of elementary symmetric polynomials,
	whose properties we shall discuss in this text.
\end{abstract}

\section{Introduction}
For the entire paper, let $k$ be  a field with $\operatorname{char} k \neq 2$.
Consider the polynomial ring $\Pol_n≔k[x_1,\dotsc,x_n]$.
The symmetric group $S_n$ acts on $\Pol_n$ by interchanging indeterminates.
Let $\alpha_i=x_{i+1}-x_i$; then $\Pol_n$ has a decomposition $\Pol_n=\Pol_n^{s_i}\oplus \alpha_i\Pol_n^{s_i}$
into direct summands on which the simple transposition $s_i$ acts by $\id$ and $-\id$, respectively.

The \term{NilHecke algebra $\NH_n$ of type $\mathrm{A}_n$} is the $k$-subalgebra of $\End_k(\Pol_n)$
generated by multiplication with the indeterminates $x_i$, and by the \term{Demazure operators}
$\partial_i\colon \Pol_n\to\Pol_n^{s_i}, f\mapsto \frac{f-s_i(f)}{x_{i+1}-x_i}$, for $i = 1,\dotsc, n-1$.
The latter assigns to a polynomial its $s_i$-invariant part
and satisfy $\ker\partial_i=\alpha_i\Pol_n^{s_i}$
\autocites[§2]{Demazure:Invariants}.
They are $s_i$-derivations
(\ie, $\partial_i(fg)=\partial_i(f)g+s_i(f)\partial(g)$)
and satisfy $\partial_i^2=0$.
Thus, $\NH_n$ is the $k$-algebra generated by $x_1,\dotsc,\dotsc, x_n$ and $\partial_1,\dotsc,\partial_{n-1}$ and subject to the relations
\begin{alignat*}{4}
		x_ix_j &= x_jx_i &~& \text{for all $1 \leq i,j \leq n$,} & \qquad\qquad
			\partial_i\partial_{i+1}\partial_i &= \partial_{i+1}\partial_i\partial_{i+1} &~& \text{for all $1 \leq i < n-1$,}\quad\\
		\partial_i x_i - x_{i+1}\partial_i &= -1 && \text{for all $1 \leq i < n$,} & \qquad\qquad
			\partial_i\partial_j &= \partial_j\partial_i&& \text{if $j\neq i, i+1$,}\\
		\partial_i x_{i+1} - x_i\partial_i &= 1 && \text{for all $1 \leq i < n$,} &
			\partial_i^2 &= 0 && \text{for all $i$,}\\
		\partial_ix_j - x_j\partial_i &= 0 &~& \text{if $j\neq i, i+1$.}
\end{alignat*}
which acts faithfully on the polynomial ring $\Pol_n$.
The subalgebra $\NC_n$ of $\NH_n$ generated by the symbols $\partial_1,\dotsc,\partial_{n-1}$ is called \term{NilCoxeter algebra}.

The invariant algebra $\SPol_n≔\Pol_n^{S_n}$,
whose elements are called \term{symmetric polynomials},
is the largest subalgebra of $\Pol_n$ annihilated by $\NC_n$.
As a $k$-algebra, $\SPol_n$ is isomorphic to the polynomial ring
$k[\varepsilon_1^{(n)},\dotsc,\varepsilon_n^{(n)}]$ in the \term{elementary symmetric polynomials}
\begin{equation}
	\varepsilon^{(n)}_k = \sum_{\mathclap{1\leq i_1<\cdots<i_k\leq n}} x_{i_1}\dotsm x_{i_k},
\end{equation}
and to the polynomial ring $k[h_1^{(n)},\dotsc,h_n^{(n)}]$ in the \term{complete symmetric polynomials}
\begin{equation}
	h^{(n)}_k           = \sum_{\mathclap{1\leq i_1\leq \cdots\leq i_k\leq n}} x_{i_1}\dotsm x_{i_k}.
\end{equation}
As a $\SPol_n$-algebra, $\Pol_n$ is free of rank $n!$.

\Citeauthor{EKL:odd-NilHecke} have developed a theory of
\emph{odd} symmetric polynomials and have constructed an odd analogue $\ONH_n$ of the NilHecke algebra \autocite{EKL:odd-NilHecke}.
By \emph{odd}, it is meant that odd polynomials are generated by indeterminates $\xi_1,\dotsc,\xi_n$ that, instead of commutativity, satisfy the relation $\xi_i \xi_j = -\xi_j \xi_i$ whenever $i\neq j$.
Our goal is to mimick the theory of symmetric and odd symmetric polynomials
for a \term{super polymial algebra} $\sPol_I$ associated to a $\sZ$-graded index set $I=I_0\sqcup I_1$.
For $i \in I$, let $\pty{i}$ be the index such that $i \in I_{\pty{i}}$, called the \term{parity} of $i$.
Then $\sPol_I$ is the $\sZ$-graded algebra generated by indeterminates $\xi_i$ for $i\in I$, that are subject to the relations $\xi_i\xi_j=(-1)^{\pty{i}\pty{j}}\xi_j\xi_i$.

In \cref{sec:superalgebra}, we start with a brief review on super algebras
and the construction of a quiver Hecke Clifford superalgebra $\NHC(C)$
from \autocite{KKT:Quiver-Hecke-Superalgebras}
in order to construct a faithful polynomial for $\NHC(C)$
along the lines of \autocite{KL:quantum-group-I} in \cref{sec:super-klr}.

In \cref{sec:clifford-symmetric-polynomials},
we develop our theory of Clifford polynomials $\PolC(I)$ and Clifford symmetric polynomials $\SPolC(I)$.
We show that the latter are generated by a generalisation of elementary symmetric polynomials
and that $\PolC(I)$ is a free $\SPolC(I)$-module.
Our main results, which are stated in \cref{prop:Clifford-Polynomials-free},
lead us to a super-analogue for the construction of cyclotomic quotients of the NilHecke algebra
and for the cohomology rings of Grassmannians and flag varieties,
which generalise from the purely odd constructions of the respective notions defined in \cite{EKL:odd-NilHecke},

\section{Superalgebras, supermodules and supercategories}
\label{sec:superalgebra}

A \term{$k$-super vector space} $V$ is a $\sZ$-graded $k$-vector space $V=V_0⊕V_1$.
The $\sZ$-degree of a homogeneous element $m$ is called its \term{parity}, which is denoted by $\pty{m}$.
A homogeneous element is called \term{even} and \term{odd} according to its parity.
The dimension of $V$ is written $\dim_k V=\dim_k V_0|\dim_k V_1$,
and we write $k^{m|n}$ for the canonical $m|n$-dimensional $k$-super vector space.
A morphism $f\colon V \to W$ of super vector spaces is a morphism $f$ of vector spaces that satisfies $f(V_p) \subseteq W_p$ for all $p \in \sZ$.
We denote the thus defined category of super vector spaces by $\ksVS$.

For $V, W \in \ksVS$, we denoted by $\hom_k(V, W)$ the super vector space with $\hom_k(V, W)_p = \{f \in \Hom_k{V, W} \mid f(V_q) \subseteq W_{p+q} \ \forall q \in \sZ \}$;
\ie, the super vector space whose even and odd parts consist of parity-preserving and -exchanging morphisms, respectively.
In particular, $\Hom_{\ksVS}(-,-) = \hom_k(-, -)_0$.

The tensor product $V⊗_k W$ with components $(V⊗_k W)_p ≔ ⨁_{q\in\mathbf Z/2\mathbf Z}V_q⊗W_{p-q}$ fits into the usual adjunction with $\hom_k(-,-)$.
Furthermore, with the braiding $V⊗W\to W⊗V, v⊗w ↦ (-1)^{\pty{v}\pty{w}}w⊗v$,
$\ksVS$ becomes a closed symmetric monoidal category;
\ie, the tensor product of morphisms satisfies
\begin{align}
	\label{eq:super-interchange}
	(f⊗g)(v⊗w)&=(-1)^{\pty{g}\pty{v}}f(v)⊗g(w) &
	(f⊗g)\circ(f'⊗g') &= (-1)^{\pty{f'}\pty{g}}f\!f'⊗gg'
\end{align}
for all homogeneous elements $v\in V, w\in W$
and homogeneous morphisms $f \in \hom_k(V', V)$, $f' \in \hom_k(V'', V')$, $g \in \hom_k(W', W)$ and $g' \in \hom_k(W'', W')$.

\begin{remark}
	Given finite dimensional super vector spaces $V=k^{n|m}$ and $W=k^{p|q}$ with homogeneous bases,
	a homomorphism in $\hom_k(V, W)$ can be written as an $(m|n)\times (p|q)$-block matrix
	\[
		\bordermatrix{
			& m	& n\cr
			p & T_{0,0} & T_{0,1}\cr
			q & T_{1,0} & T_{1,1}
		}
	\]
	where $\hom_k(V,W)_0$ and $\hom_k(V,W)_1$ comprise the diagonal and the antidiagonal blocks, respectively.
\end{remark}

\subsection{Superalgebras}
\begin{definition}
	A \term{superalgebra} $A$ over a field $k$ is a super vector space $A$ with a graded $k$-algebra structure.
	It is called \term{supercommutative} or simply \term{commutative}
	if the multiplication commutes with the braiding;
	\ie, $ab = (-1)^{\pty{a}\pty{b}}ba$ for all homogeneous $a,b∈A$.
	In particular, this means that in a commutative superalgebra,
	every odd element squares to zero.
	The tensor product of two superalgebras $A$ and $B$ (over $k$)
	carries a superalgebra structure by $(a⊗b)⋅(a'⊗b')=(-1)^{\pty{a'}\pty{b}}aa'⊗bb'$.
\end{definition}
\begin{example}
	For every super vector space $V$,
	the space $\hom_k(V,V)$ is a superalgebra.
	In general, it is not commutative, unless $V=k^{1\mid 0}$ or $V=k^{0\mid 1}$.
	For super vector spaces $V$ and $W$, the space $\hom_k(V,V)\otimes\hom_k(W,W)\subseteq\hom_k(V\otimes W, V\otimes W)$
	is a super subalgebra, according to \eqref{eq:super-interchange}.
\end{example}
\begin{example}
	The exterior algebra on $n$ generators
	is the $k$-algebra
	\[\Lambda[\omega_1, \dotsc, \omega_n] ≔ k⟨\omega_1, \dotsc, \omega_n⟩/(\omega_i\omega_j + \omega_j\omega_i)_{1≤i,j≤n}.\]
	It is a commutative superalgebra.
	The \term{polynomial superalgebra} on $m|n$ indeterminates is the commutative superalgebra
	$k[y_1,\dotsc, y_m|\omega_1,\dotsc, \omega_n] ≔ k[y_1,\dotsc, y_m] ⊗_k \Lambda[\omega_1,\dotsc, \omega_n]$.
\end{example}
\begin{example}
	\label{def:super-algebra:clifford-algebra}
	The \term{Clifford algebra} on $n$ generators is the $k$-superalgebra
	\begin{equation}
	\Cl_n≔⟨\cl_1,\ldots, \cl_n\mid \cl_i^2 = 1\ \text{and}\ \cl_i \cl_j = -\cl_j \cl_i\ \text{if $i\neq j$}⟩
	\end{equation}
	all of whose generators $\cl_i$ are odd.
	It is not supercommutative.
\end{example}
\begin{definition}
	Given two $k$-superalgebras $A$ and $B$, an \term{$A$-$B$-super bimodule} $M$
	is a super vector space endowed with \emph{even} superalgebra morphisms
	$A\to\hom_k(M,M)$ and $B^{\mathrm{op}}\to\Hom_k(M,M)$.
	According to \eqref{eq:super-interchange},
	a morphism $f$ of bimodules satisfies $f(amb) = (-1)^{\pty{f}\pty{a}} af(m)b$ for $a \in A$, $b \in B$.
	We denote the category of $A$-$B$-super bimodules by $\Mod[s]{A}[B]$.

	One-sided modules are defined in the obvious way.
	Over a supercommutative superalgebra,
	a left module obtains right module structure over the opposite superalgebra by imposing
	$m a = (-1)^{\pty{a}\pty{m}} am$.
	Tensor products of bimodules over superalgebras are defined in the usual way.
\end{definition}

\subsection{Supercategories}

\begin{definition}
	A $k$-linear \term{supercategory}, \term{superfunctor} and \term{supernatural transformation}
	respectively are a category, a functor and a natural transformation enriched in the monoidal category $\ksVS$.
	Explicitly:
	\begin{itemize}
	\item A \emph{supercategory} $\mathcal C$ is a category such that $\Hom_{\mathcal{C}}(X, Y)$ is a $k$-super vector space for all $X, Y \in \mathcal{C}$,
	and composition is an even map
	\[
		\circ_{X, Y, Z}\colon \Hom_\mathcal C(X, Y)\otimes_k\Hom_\mathcal C(Y, Z) \to \Hom_\mathcal C(X, Z)
	\]
	of $k$-super vector spaces for all $X, Y, Z \in \ksVS$.
	\item A \emph{superfunctor} $F\colon \mathcal{C} \to \mathcal{D}$ of supercategories is a functor such that the map $F_{XY}\colon \Hom_{\mathcal{C}}(X, Y) \to \Hom_{\mathcal{D}}(FX, FY)$ is an even map of super vector spaces for all $X, Y \in \mathcal{C}$.
	\item Let $F, G\colon \mathcal{C} \to \mathcal{G}$ be two superfunctors, then the space $\hom(F, G)$ of \emph{supernatural transformations} from $F$ to $G$ is the super vector space such that $\hom(F, G)_p$ is the $k$-vector space of all (ordinary) natural transformations
	that satisfy $\eta_{Y}\circ F(f)=(-1)^{p\pty{f}}\circ G(f)\circ \eta_{X}$ for all $X, Y \in \mathcal{C}$ and homogeneous $f \in \Hom_{\mathcal{C}}(X, Y)$.
	\end{itemize}
	A supercategory $\mathcal{C}$ that is also a monoidal category is a \term{monoidal supercategory}
	if its associators and unitors are even supernatural transformations and satisfy \eqref{eq:super-interchange};
	It is \term{closed} and \term{symmetric} if it is as an ordinary monoidal category;
	see \autocite{Bru:Supercategories} for a comprehensive treatment.
\end{definition}
\begin{example}
	The super vector spaces $\hom_k(-,-)$ turn $\ksVS$
	into a closed symmetric monoidal supercategory.
\end{example}
\begin{example}
	The category $\Mod[s]{A}$ of modules over a super $k$-algebra $A$
	is a closed symmetric monoidal supercategory, where $\Hom_A(M, N)$ is defined as the equalizer
	\[
		\Hom_A(M, N) ≔ \operatorname{eq}\Bigl(
			\hom_k(M, N)\xrightrightarrows[a_*]{a^*}\prod_{a\in A}\hom_k(M,N)
		\Bigr).
	\]
	The forgetful functor $\Mod[s]{A}\to\ksVS$ is a superfunctor.
\end{example}

\subsection{Super-Diagrams}
\label{sec:super-diagrams}
We draw vertical lines for morphisms; these are
placed next to each other to represent their tensor product.
The super interchange law \eqref{eq:super-interchange} is thus depicted as
\begingroup
\tikzset{every diag/.append style={smaller}}
\begin{equation*}
	\begin{array}{crccc}
	f⊗g &{}={}&
	(f⊗1)\circ(1⊗g) &{}={}&
	(-1)^{\pty{f}\pty{g}} (1⊗g)\circ(f⊗1)\\
	\tikz[diag]\draw	(0,-1) node[below]{$V$}  --node[fill=white, draw, inner sep=1pt]{$f$} (0,1) node[above]{$W$}
						(2,-1) node[below]{$V'$} --node[fill=white, draw, inner sep=1pt]{$\mathstrut g$} (2,1) node[above]{$W'$};
	&{}\coloneqq{}&
	\tikz[diag]\draw
		(0,-1) node[below]{$V$}
			--node[near end, fill=white, draw, inner sep=1pt]{$f$}
			(0,1) node[above]{$W$}
		(2,-1)
			node[below]{$V'$}
			--node[near start, fill=white, draw, inner sep=1pt]{$\mathstrut g$}
			 (2,1) node[above]{$W'$};
	&{}={}&
	(-1)^{\pty{f}\pty{g}}
	\tikz[diag]\draw
		(0,-1) node[below]{$V$}
			--node[near start, fill=white, draw, inner sep=1pt]{$f$}
			 (0,1) node[above]{$W$}
		(2,-1) node[below]{$V'.$}
			--node[near end, fill=white, draw, inner sep=1pt]{$\mathstrut g$}
			 (2,1) node[above]{$W'$};
	\\
	\end{array}
\end{equation*}
\endgroup
Such diagrams are to be understood as \emph{first vertical}
and \emph{then horizontal composition} of morphisms \autocite[cf.][]{Bru:Supercategories}.

\section{Super-KLR algebras}
\label{sec:super-klr}
We now recall the notion of KLR-superalgebras
following \autocite{KKT:Quiver-Hecke-Superalgebras}.
This is a generalisation of the ordinary KLR-algebra
from \autocites{KL:quantum-group-I}{Rou:2-Kac-Moody}.

\subsection{NilHecke Superalgebra}
Recall the Clifford algebra $\Cl_n$ from \cref{def:super-algebra:clifford-algebra}.
The following definitions are motivated by the definition of the quiver Hecke Clifford superalgebra
from \autocite[§3.3]{KKT:Quiver-Hecke-Superalgebras}; see \cref{quiver Hecke superalgebra:graphically}.
\begin{definition}
	\label{def:super-klr-algebra:nil-hecke-clifford-algebra}
	For any field $k$, the \term{NilHecke Clifford superalgebra} $\NHC_n$  of type $\mathrm{A}_n$ is the $k$-superalgebra
	with even generators $y_1,\dotsc,y_n$ and $\delC_1,\dotsc,\delC_{n-1}$, and odd generators $\cl_1,\dotsc,\cl_n$,
	subject to the relations
	\settowidth\algnRef{$\delC_i \delC_{i+1} \delC_i$}
	\begin{equation}
		\label{eq:Definition-NilHecke-Clifford-algebra}
		\begin{aligned}
			\mathmakebox[\algnRef][r]{y_iy_j} &=
				y_jy_i
				&&\forall i,j,\\
			\mathmakebox[\algnRef][r]{\cl_i\cl_j} &=
				-\cl_j\cl_i
				&&\forall i\neq j,\\
			\mathmakebox[\algnRef][r]{\cl_i^2} &=
				1&&\forall i,\\
			\mathmakebox[\algnRef][r]{y_i \cl_j} &=
				(-1)^{\delta_{i,j}} \cl_j y_i
				&&\forall i,j,
		\end{aligned}
		\hfil
		\settowidth\algnRef{$-1-c_ic_{i+1}$}
		\begin{aligned}
			\delC_i y_j - y_{s_i(j)}\delC_i &=
			\begin{cases}
				-1-\cl_i \cl_{i+1}					& \text{if $j=i$}\\
				\phantom{-}1-\cl_i \cl_{i+1}		& \text{if $j=i+1$}\\
				0									& \text{otherwise}
			\end{cases}\\
			\delC_i\cl_j &= \casesGap\mathmakebox[\algnRef][l]{\cl_{s_i(j)}\delC_i}
			\quad\forall i,j,
		\end{aligned}
		\hfil
	\end{equation}
	and the NilCoxeter-relations
	\settowidth{\algnRef}{$(-1)^{\delta_{i,j}} \cl_j y_i$}
	\begin{equation}
		\label{eq:NilCoxeter--Clifford-relations}
		\begin{aligned}
			\delC_i^2 &= 0\\
			\delC_i \delC_j &= \mathmakebox[\algnRef][l]{\delC_j \delC_i} &&\text{for $|i-j|>1$}\\
			\delC_i \delC_{i+1} \delC_i &= \delC_{i+1} \delC_i \delC_{i+1}.
		\end{aligned}
	\end{equation}
	The \term{NilCoxeter Clifford superalgebra} $\NCC_n\subseteq \NHC_n$ is the subalgebra
	 generated by the symbols $\delC_i$ and $\cl_j$.
\end{definition}
\begin{definition}
	\label{def:noncommutative-algebra}
	Let $A$ be a superalgebra over a field $k$.
	An \term{$A$-superalgebra} $B$ is an $A$-$A$-bimodule equipped with an associative multiplication map
	$B⊗_A B\to B$
	that is an even $A$-$A$-bimodule homomorphism.
\end{definition}

We can consider $\NHC_n$ as a $\Cl_n$-superalgebra
generated by the even symbols $y_i$ and $\delC_i$.
We introduce the following analogues to the polynomial representation
of the classical NilHecke algebra:

\begin{definition}
	\label{def:Polynomial-Clifford-Algebra}
	The \term{polynomial Clifford superalgebra} $\PolC_n$ is the $\Cl_n$-superalgebra
	with generators and relations
	\begin{equation*}
		\PolC_n ≔ \bigl\langle y_1,\dotsc, y_n \bigm|
		y_i y_j = y_j y_i,
		y_i \cl_j = (-1)^{\delta_{ij}} \cl_j y_i ∀i,j\bigr\rangle_{\Cl_n}.
	\end{equation*}
\end{definition}
\begin{lemma}
	\label{lem:PolCn-is-free}
	$\PolC_n$ is a free as left and right $\Cl_n$-module
	with basis $\{y_1^{\alpha_1}\dotsm y_n^{\alpha_n} \mid \alpha_i ∈ \mathbf{N}\}$.
	In particular, $\PolC_n$ has the same graded rank
	as left and as right $\Cl_n$-module.
\end{lemma}
\begin{proof}
	Take any monomial from $\PolC_n$.
	According to the defining relation $y_i\cl_j=(-1)^{\delta_{ij}}\cl_jy_i$ of $\PolC_n$,
	we may move all symbols $\cl_i$ to the left (or right),
	possibly at the cost of introducing a sign when sliding them past a symbol $y_i$ with the same index.
	Since the $y_i$'s commute,
	we may sort powers of $y_i$'s by their index.
	As a left (or right) $\Cl_n$-module,
	$\PolC_n$ is thus isomorphic to $\Cl_n⊗_k \Pol_n$ (or $\Pol_n⊗_k\Cl_n$).
\end{proof}
\begin{definition}
	\label{def:clifford-demazure}
	As a Clifford-analogue for the Demazure operator $\partial_i$ on $\Pol_n$,
	we define an even endomorphism $\delC_i$ of $\PolC_n$ by
		\begin{align*}
			\delC_i(y_i) 		&≔ -1-\cl_i\cl_{i+1}\ \forall i, 				\qquad	&	\delC_i(y_j) &≔ 0\ ∀j\neq i,i+1,\\
			\delC_i(y_{i+1}) 	&≔ \phantom{-}1-\cl_i\cl_{i+1}\ \forall i,				&	\delC_i(\cl_j) &≔ 0\ ∀j,\\
		\intertext{such that $\delC_i$ is a $s_i$-derivation, \ie,}
			\delC_i(fg)			&≔ \mathrlap{\delC_i(f)g+s_i(f)\delC_i(g) \ ∀j\neq i,i+1}.
		\end{align*}
	Here, the symmetric group $S_n$ acts on $\PolC_n$
	by permuting the symbols $y_i$ and $\cl_j$ independently.
\end{definition}
\begin{lemma}
	\label{lemma:clifford-demazure well-defined}
	The Clifford Demazure operator $\delC_i$ is a well-defined morphism of right $\Cl_n$-modules.
\end{lemma}
\begin{proof}
	To show well-definedness inductively, take an $f∈\PolC_n$.
	Without loss of generality, we assume that $i=1$.
	We show that $\delC_1$ is compatible with the relations of $\PolC_n$.
	\begin{itemize}
		\item Relations involving only symbols $y_i$: since
		\begin{align*}
			\delC_1 (y_1 y_2) &= (-1-\cl_1\cl_2)y_2 + y_2 (1-\cl_1\cl_2) = 0,\\
			\delC_1 (y_2 y_1) &= (1-\cl_1\cl_2)y_1 + y_1 (-1-\cl_1\cl_2) = 0,
		\end{align*}
		we get for all $f$ that \[\delC_1(y_1 y_2 f) = y_1y_2 \delC_1 (f) = y_2y_1 \delC_1 (f) = \delC_1(y_2 y_1 f).\]
		\item Relations involving symbols $y_i$ and $\cl_i$:
		since
		$\delC_1(\cl_1 y_1) = \cl_2(-1-\cl_1 \cl_2) = \cl_1-\cl_2 = -\delC_1(y_1 \cl_1),$
		we get
		\begin{equation*}
			\delC_1(\cl_1 y_1 f) = \delC_1(\cl_1 y_1) f + \cl_2 y_2 \delC_1 (f) = -\delC_1(y_1 \cl_1) f -  y_2 \cl_2 \delC_1 (f) = -\delC_1( y_1 \cl_1 f).
		\end{equation*}
		The other relations $\delC_1(\cl_2 y_1 f) = - \delC_1( y_1 \cl_2 f)$ and $\delC_1(y_2\cl_1 f) = \delC_1(\cl_2y_2 f)$ and $\delC_1(\cl_2y_2 f)=-\delC_1(y_2\cl_2 f)$ are shown similarly.
	\end{itemize}
	It follows by induction on the length of $f$ (in terms of the generators) that $\delC_i$ is well-defined.
	The operator $\delC_i$ is a right $\Cl_n$-module homomorphism
	because $\delC_i(f\cl_j) = \delC_i(f)\cl_j + s_i(f)\delC_i(\cl_j) = \delC_i(f)\cl_j$
	for all $f∈\PolC_n$ and $1\leq j\leq n$.
\end{proof}
\begin{remark}
	\label{lem:Kernel-of-Clifford-Demazure-is-Algebra}
	The kernel $\ker\delC_i$ is a $\Cl_n$-super subalgebra of $\PolC_n$
	because $\delC_i(fg) = \delC(f)g + s_i(f)\delC_i(g) = 0$ for $f,g\in\ker\delC_i$.
	Trivially, we have $\Cl_n \subseteq \ker\delC_i$.
\end{remark}
\begin{lemma}
	\label{lem:super-klr-algebra:clifford-demazure-relations}
	\begin{thmlist}
		\item The Clifford Demazure operators $\delC_i$ and multiplications by $y_i$ and $\cl_i$
			render $\PolC_n$ a representation of the NilHecke Clifford superalgebra.
		\item $\ker\delC_i=\im\delC_i$.
	\end{thmlist}
\end{lemma}
\begin{proof}
	\begin{thmlist}
	\item
	We have to make sure that the Clifford Demazure operators
	satisfy the relations from \cref{def:super-klr-algebra:nil-hecke-clifford-algebra}.
	From the definition of $\delC_i$ one can see
	that the relations \cref{eq:Definition-NilHecke-Clifford-algebra} are satisfied.
	We show that the operators $\delC_i$ satisfy \cref{eq:NilCoxeter--Clifford-relations}:
	\begin{itemize}
		\item To show $\delC_i^2 (f)=0$:
		Assume that the statement holds for $f,g∈\PolC_n$.
		It follows from
		\settowidth{\algnRef}{$\delC_1 \delC_2 \delC_1 (y_3 f)$}
		\begin{align*}
			\mathmakebox[\algnRef][r]{\delC_1^2(fg)} &= \delC_i \big( \delC_i(f)g + s_i(f) \delC_i (g) \big)\\
			&= s_i(\delC_i(f))\, \delC_i(g) + \delC_i(s_i f)\, \delC_i (g)
		\end{align*}
		that it suffices to show that $s_i\delC_i=-\delC_i s_i$.
		To show this inductively, assume that $s_i\delC_i (f)=-\delC_i(s_i f)$.
		We need to show that $s_i\delC_i (y_j f)=-\delC_i \bigl(s_i(y_j f)\bigr)$,
		which is obvious for $j\neq i$.
		Applying the induction hypothesis in $*$ below, one calculates
		\begin{align*}
			\mathmakebox[\algnRef][r]{s_i\delC_i(y_i f)} &= s_i\bigl[(-1-\cl_i\cl_{i+1}) f
				+ y_{i+1}\,\delC_i f\bigr]\\
			&\asteq (-1 +\cl_i\cl_{i+1}) \, s_i f
				- y_{i+1}\, \delC_i s_i f\\
			&= -\delC_i(y_{i+1}\, s_i f),
		\end{align*}
		The calculation for $j=i+1$ is similar.
		\item To show $\delC_i\delC_j (f)=\delC_j \delC_i (f)$ for $|i-j|>1$:
		Clear from the definition of $\delC_i$.
		\item To show $\delC_i\delC_{i+1}\delC_i(f)= \delC_{i+1}\delC_{i}\delC_{i+1}(f)$:
		Assume the statement holds for $f∈\PolC_n$.
		It is clear that the statement then also holds true for $\cl_j f$ and $y_jf$
		for all $j\neq i,i+1,i+2$.
		For these cases, we assume \wlofg\ that $i=1$.
		Applying the induction hypothesis in $*$, one computes
		\begin{align*}
			\delC_1 \delC_2 \delC_1 (y_1 f)
			&= \delC_1 \delC_2 (-1-\cl_1\cl_2)f + y_2\, \delC_2f)\\*
			&= \delC_1 \big((-1-\cl_1\cl_3)\delC_2 f + (-1-\cl_2\cl_3)\delC_1 f + y_3\, \delC_2 \delC_1 f\big)\\
			&= (-1-\cl_2\cl_3)\delC_1 \delC_2 f + y_3 \,\delC_1 \delC_2 \delC_1 f\\
			&\asteq
			   \delC_2 \big((-1-\cl_1\cl_2)\delC_2 f + y_2\, \delC_1 \delC_2 f\big)\\
			&= \delC_2 \delC_1 (y_1\, \delC_2 f)\\*
			&= \delC_2 \delC_1 \delC_2 (y_1 f);\\
			\delC_1 \delC_2 \delC_1 (y_2 f)
			&= \delC_1 \delC_2 \big((-1-\cl_1\cl_2)f + y_1\,\delC_1 f\big)\\*
			&= \delC_1 \big((1-\cl_1\cl_3) \delC_2 f + y_1 \,\delC_2 \delC_1 f\big)\\
			&= (-1-\cl_2\cl_3)\delC_1 \delC_2 f + (-1-\cl_1\cl_2)\delC_2 \delC_1 f + y_2\, \delC_1 \delC_2 \delC_1 f\\
			&\asteq
			   \delC_2 \big((-1-\cl_1\cl_3)\delC_1 + y_3\,\delC_1 \delC_2 f\big)\\
			&= \delC_2 \delC_1 \big((-1-\cl_2\cl_3)f + y_3\, \delC_2 f)\\*
			&= \delC_2 \delC_1 \delC_2 (y_2 f);\\
			\delC_1 \delC_2 \delC_1(y_3 f)
			&= \delC_1 \delC_2 (y_3 \,\delC_1 f)\\*
			&= \delC_1 \big( (1-\cl_2\cl_3) \delC_1 f + y_2 \,\delC_2 \delC_1 f)\\
			&= (1-\cl_1\cl_2) \delC_2 \delC_1 f + y_1\, \delC_1 \delC_2 \delC_1 f\\
			&\stackrel{*}{=}
			   \delC_2 \big((1-\cl_1\cl_3)f + (1-\cl_1\cl_2)\delC_2 f + y_1\, \delC_1 \delC_2 f\big)\\
			&= \delC_2 \delC_1 \big((1-\cl_2\cl_3)f + y_2 \,\delC_2 f\big)\\*
			&=\delC_2 \delC_1 \delC_2 (y_3\, f).
		\end{align*}
	\end{itemize}
	Therefore, $\PolC_n$ is a representation of $\NHC_n$.
	We show in \cref{lem:Clifford-NilHecke-acts-Faithfully} that it is faithful.

	\item
	Because $\delC_i^2=0$, we may regard $(\PolC_n, \delC_i)$ as a chain complex of right $\Cl_n$-modules,
	using its polynomial grading in the $y_i$'s.
	We claim that this chain complex $\PolC_n$ is contractible by the homotopy
		\[
			h_k\colon  (\PolC_n)_k \to (\PolC_n)_{k+1},
			f\mapsto
			\begin{cases}
				\frac{1}{2}(-1-\cl_i\cl_{i+1})y_i f & \text{if $k$ is even,}\\
				\frac{1}{2}(+1-\cl_i\cl_{i+1})y_{i+1}f & \text{if $k$ is odd.}
			\end{cases}
		\]
	Assume $f$ is even. We obtain
	\begin{align*}
		&\eqsp h_{k-1}(\delC_i f) + \delC_i h_k (f)\\
		&= \tfrac 12(1-\cl_i\cl_{i+1})y_{i+1} (\delC_i f)
			+\tfrac{1}{2}\delC_i\bigl( (-1-\cl_i\cl_{i+1}) y_if\bigr)\\
		&= \tfrac 12(1-\cl_i\cl_{i+1})y_{i+1} (\delC_i f)
			+\tfrac{1}{2}(-1+\cl_i\cl_{i+1}) \bigl[(-1-\cl_i\cl_{i+1}) f + y_{i+1}\delC_i(f)\bigr]\\
		&= \tfrac{1}{2}\bigl[ (1-\cl_i\cl_{i+1})+(-1+\cl_i\cl_{i+1}) \bigr] y_{i+1} \delC_i (f)
			+ \tfrac{1}{2}(-1+\cl_1\cl_{i+1})(-1-\cl_i\cl_{i+1}) f\\
		&=-f.
 	\end{align*}
 	The computation is similar if $f$ is odd.
	Therefore, $-h$ indeed is a homotopy from the identity to the zero morphism,
 	and the chain complex $(\PolC_n, \delC_i)$ is acyclic.
 	This proves the statement.\qedhere
	\end{thmlist}
\end{proof}

\subsection{Quiver Hecke (Clifford) superalgebra}
\label{sec:quiver-hecke-superalgebra}
Let $I=I_0\sqcup I_1$ be a finite, $\sZ$-graded set as in the introduction.

\begin{definition}[{\autocite[§3]{KKT:Quiver-Hecke-Superalgebras}}]
	A \term{generalised Cartan matrix} on $I$ is a matrix $C=(d_{ij})∈𝐙^{I×I}$ such that
	\begin{thmlist}
		\item $d_{ii}=2$ for all $i∈I$,
		\item $d_{ij}≤0$ for distinct $i,j∈I$,
		\item $d_{ij}=0$ if and only if $d_{ji}=0$, and
		\item $d_{ij}$ is even if $i∈I_\text{1}$.
	\end{thmlist}
	We define a $𝐙$-valued bilinear form on $𝐍^I$ by $⟨i, j⟩≔-d_{ij}$.
	Equivalently, we may describe the same datum by a graph $\Gamma_C$
	with vertex set $I$ and $d_{ij}$ directed edges from $i$ to $j$ for $i\neq j$.
	For distinct indices $i,j$, we write $i\conn j$ if $d_{ij}\neq 0$
	and $i\disc j$ if $d_{ij}=0$.
\end{definition}
For the remainder of this section, let $C∈\mathbf{Z}^{I\times I}$ be a symmetrisable generalised Cartan matrix
and fix an orientation for each edge $i\conn j$ of $\Gamma_C$.
We write $i\to j$ if the edge is oriented from $i$ to $j$.

The following diagrammatic calculus for the quiver Hecke super algebra from \autocite[§3]{KKT:Quiver-Hecke-Superalgebras}
mimics the constructions from \autocites{KL:quantum-group-I}{KL:quantum-group-II}.

\begin{definition}
	A \term{string diagram} with $n$ strings consists of $n$ continuous paths
	$𝜙_k:[0,1]↪𝐑×[0,1]$ for $1≤k≤n$, called strings,
	such that
	\begin{thmlist}
		\item each string starts in $𝐍×\{0\}$ and ends in $𝐍×\{1\}$,
		\item the projection of any string to the second coordinate is strictly monotonically increasing,
		\item each endpoint has precisely one string attached to it, and
		\item at most two strings may intersect in a single point.
	\end{thmlist}
	A string diagram is defined up to isotopy.
	Strings in a diagram may be endowed with certain point-like decorations,
	whose positions are also specified up to isotopy.
	Crossings and decorations are endowed with parities.
	Unless decorations commute, we make sure they appear at distinct heights in one diagram.
	Strings are numbered from the left at their bottom endpoints.
\end{definition}

\begin{definition}[($\sH_n(C)$, diagrammatically)]
	\label{quiver Hecke superalgebra:graphically}
	The \term{quiver Hecke superalgebra} $\sH_n(C)$ is the $k$-linear supercategory
	consisting of the following data.
	\begin{itemize}
		\item Its objects are sequences $\nu∈I^n$.
		\item If $\nu'$ is not a permutation of $\nu$, then $\Hom(\nu, \nu') \coloneqq 0$.
			Otherwise, $\Hom(\nu, \nu')$ is generated by string diagrams on $n$ strings
			which connect identical entries of $\nu$ and $\nu'$.
			We say that the $k$-th string is labelled by $\nu_k$
			and usually write its label below the string.
			Strings may be decorated by an arbitrary non-negative number of dots,
			distant from the crossings.
		\item
			The dots
			$x_{k,\nu} ≔
				\tikz[diag, inline] \draw (0,-1) node[right]{$\nu_k$} -- node[rot, deco]{$\bullet$} (0,1);$
			and crossings
			$\tau_{k,\nu} ≔
				\tikz[diag, inline] \draw
						(-1,-1) node[left]{$\nu_k$} -- (1,1)
						(1,-1) node[right]{$\nu_{k+1}$} --(-1,1);
			$
			have parities
			$|\tikz[diag, inline] \draw (0,-1) node[right]{$\nu_k$} -- node[rot, deco]{$\bullet$} (0,1);| ≔ |\nu_k|$
			and
			$|
				\tikz[diag, inline] \draw
						(-1,-1) node[left]{$\nu_k$} -- (1,1)
						(1,-1) node[right]{$\nu_{k+1}$} --(-1,1);
			 |≔|\nu_k||\nu_{k+1}|$.
		\item Composition is given by vertically stacking diagrams
		and subject to the following local relations:
	\settowidth{\algnRef}{$%
		\displaystyle
		\tikz[diag] \draw (-1,-1) node[below]{$i$} to node[label=10:$d_{ij}$]{$\bullet$} (-1,1) (1,-1)node[below]{$i$} -- (1,1);
		+\tikz[diag] \draw (-1,-1) node[below]{$i$} to (-1,1) (1,-1)node[below]{$i$} to node[label=10:$d_{ji}$]{$\bullet$} (1,1);
	$}
	\begin{align}
		\tikz[diag] \draw
			(-1,-1) node[below]{$i$} -- node[deco] {$\bullet$} (-1,1)
			(1,-1) node[below]{$j$} -- node[deco] {$\bullet$} (1,1);
		=
		\tikz[diag] \draw
			(-1,-1) node[below]{$i$} -- node[deco, near end] {$\bullet$} (-1,1)
			(1,-1) node[below]{$j$} -- node[deco, near start] {$\bullet$} (1,1);
		&=
		\casesGap
		\mathmakebox[\algnRef]{
			\mathllap{(-1)^{\pty{i}\pty{j}}}
			\tikz[diag] \draw
				(-1,-1) node[below]{$i$} -- node[deco, near start] {$\bullet$} (-1,1)
				(1,-1) node[below]{$j$} -- node[deco, near end] {$\bullet$} (1,1);
		}
		\quad
		\forall i,j\\
		%
		\label{eq:bullets crossings 1}
		  \tikz[diag] \draw (-1,-1) node[below]{$i$} -- (1,1) (1,-1) node[below]{$j$} to[near start] node[deco]{$\bullet$} (-1,1);
		- (-1)^{|i||j|}
		  \tikz[diag] \draw (-1,-1) node[below]{$i$} -- (1,1) (1,-1) node[below]{$j$} to[near end] node[deco]{$\bullet$} (-1,1);
		&=\begin{cases}
			\mathmakebox[\algnRef]{0}
			& \text{if $i\neq j$}\\
			\mathmakebox[\algnRef]{%
				\tikz[diag] \draw (-1,-1) node[below]{$i$} -- (-1,1) (1,-1)node[below]{$i$} -- (1,1);
			}
			& \text{if $i=j$}
		\end{cases}
		\\
		\label{eq:bullets crossings 2}
		  \tikz[diag] \draw (-1,-1) node[below]{$i$} to[near start] node[deco]{$\bullet$} (1,1) (1,-1) node[below]{$j$} -- (-1,1);
		- (-1)^{|i||j|}
		  \tikz[diag] \draw (-1,-1) node[below]{$i$} to[near end] node[deco]{$\bullet$} (1,1) (1,-1) node[below]{$j$} -- (-1,1);
		&=\begin{cases}
			\mathmakebox[\algnRef]{0}
			& \text{if $i\neq j$}\\
			\mathmakebox[\algnRef]{%
				\mathllap{-{}}\tikz[diag] \draw (-1,-1) node[below]{$i$} -- (-1,1) (1,-1)node[below]{$i$} -- (1,1);
			}
			& \text{if $i=j$}
		\end{cases}
		\\
		%
		%
		\label{eq:double crossing}
		\tikz[diag, curved] \draw (-1,-2) node[below]{$i$} to[out=+45, in=-45] (-1,2)
			(1,-2) node[below]{$j$} to[out=135,in=-135] (1,2);
		&= \begin{cases}
			\mathmakebox[\algnRef]{%
				0
			}
			& \text{if $i=j$}\\
			\mathmakebox[\algnRef]{%
				\tikz[diag] \draw (-1,-1) node[below]{$i$} -- (-1,1) (1,-1)node[below]{$j$} -- (1,1);
			}
			& \text{if $i\disc j$};\\
			\mathmakebox[\algnRef]{%
				\tikz[diag] \draw (-1,-1) node[below]{$i$} to node[deco, label=10:$d_{ij}$]{$\bullet$} (-1,1) (1,-1)node[below]{$j$} -- (1,1);
				+\tikz[diag] \draw (-1,-1) node[below]{$i$} to (-1,1) (1,-1)node[below]{$j$} to node[deco, label=10:$d_{ji}$]{$\bullet$} (1,1);
			}
			& \text{if $i\conn j$}
		\end{cases}
		\\
		%
		%
		\label{eq:braid relation}
		\tikz[diag, scale=.6] \draw (-2,-2) node[below]{$i$} -- (2,2)
			(2,-2) node[below]{$k$} -- (-2,2)
		(0,-2) node[below]{$j$} to[out=135,in=-135, curved] (0,2);
		-(-1)^{|i||j|}
		\tikz[diag, scale=.6] \draw (-2,-2) node[below]{$i$} -- (2,2)
			(2,-2) node[below]{$k$} -- (-2,2)
			(0,-2) node[below]{$j$} to[out=45,in=-45, curved] (0,2);
		&=
		\begin{cases}
			\mathmakebox[\algnRef]{%
				\tikz[diag] \draw (-2,-1) node[below]{$i$} --(-2,1) (0,-1) node[below]{$j$} --(0,1) (2,-1) node[below]{$i=k$} --(2,1);
			}
			& \text{if $i=k$ and $i\conn j$}\\
			\mathmakebox[\algnRef]{0}
			& \text{otherwise.}
		\end{cases}
	\end{align}
	The exponents in \eqref{eq:double crossing} are meant as powers; \ie, as repetition of the symbol.
	\end{itemize}
\end{definition}
\begin{remark}
	There is an embedding of $\sH_n(C)$
	into Brundan's diagrammatic super Kac-Moody 2-category
	\autocite{Bru:Super-Kac-Moody}
	by adding an arrow tip pointing upwards to every string.
\end{remark}

For calculations it is advantageous
to \emph{adjoin} odd generators for all odd indices $i∈I_1$
instead of endowing the $x_i$'s with a super grading themselves.
This is captured by the superalgebra defined in the following diagramatically
along the lines of \autocite{KKT:Quiver-Hecke-Superalgebras}.
\begin{definition}[($\HC_n(C)$, diagrammatically)]
	\label{def:super-klr-algebra:HC(C):diagramatically}
	With the same data as above, the \term{quiver Hecke Clifford superalgebra} $\HC_n(C)$
	is the supercategory with the following data:
	\begin{itemize}
	\item It has the same objects $\nu∈I^n$ as $\sH_n(C)$.
	\item It has the generating morphisms
	$y_{k,\nu}=\tikz[diag, inline] \draw (0,-1) node[right]{$\nu_k$} -- node[rot, deco]{$\blacklozenge$} (0,1);$
	and $\sigma_{k,\nu}=%
		\tikz[diag, inline] \draw
			(-1,-1) node[left]{$\nu_k$} -- (1,1)
			(1,-1) node[right]{$\nu_{k+1}$} -- node[deco]{$\whitebullet$} (-1,1);%
	$ of even parity,
	and the generating morphism $\cl_{k,\nu}=\tikz[diag, inline] \draw (0,-1) node[right]{$\nu_k$} -- node[rot, deco]{$\whitelozenge$} (0,1);$ of odd parity.
	We refer to the symbols $y_{k,\nu}$ and $\cl_{k,\nu}$ as \emph{diamonds}, and to symbols $\sigma_{k,\nu}$ as \emph{crossings}.
	\item The morphisms are subject to the relations
	$\cl_{k,\nu}=0$ if $|\nu_k|=0$,
	$\tikz[diag, inline] \draw (0,-1.2) node[right]{$i$} --
		node[pos=0.275, rot, deco]{$ \whitelozenge$}
		node[pos=0.725, rot, deco]{$\whitelozenge$}
		(0,1.2);
	=\tikz[diag, inline] \draw (0,-1) node[right]{$i$} -- (0,1); $,
	the commutativity relations%
	$\tikz[diag, inline] \draw (0,-1.2) node[left]{$i$} --
		node[pos=0.275, rot, deco]{$\whitelozenge$}
		node[pos=0.725, rot, deco]{$\blacklozenge$}
		(0,1.2);
	=-
	\tikz[diag, inline] \draw (0,-1.2) node[right]{$i$} --
		node[pos=0.275, rot, deco]{$\blacklozenge$}
		node[pos=0.725, rot, deco]{$\whitelozenge$}
		(0,1.2);
	$
	,
	$\tikz[diag, inline]
		\draw
		(-2,-1.2) node[left]{$i$} -- node[to start, rot, deco]{$\whitelozenge$}(-2,1)
		(2,-1.2) node[right]{$j$} -- node[to end, rot, deco]{$\whitelozenge$} (2,1)
		(0,0) node{$\cdots$};
	=-
	\tikz[diag, inline]
		\draw
		(-2,-1.2) node[left]{$i$} -- node[to end, rot, deco]{$\whitelozenge$}(-2,1)
		(2,-1.2) node[right]{$j$} -- node[to start, rot, deco]{$\whitelozenge$} (2,1)
		(0,0) node{$\cdots$};$
	,
	$\tikz[diag, inline]
		\draw
		(-2,-1.2) node[left]{$i$} -- node[to start, rot, deco]{$\whitelozenge$}(-2,1)
		(2,-1.2) node[right]{$j$} -- node[to end, rot, deco]{$\blacklozenge$} (2,1)
		(0,0) node{$\cdots$};
	=
	\tikz[diag, inline]
		\draw
		(-2,-1.2) node[left]{$i$} -- node[to end, rot, deco]{$\whitelozenge$}(-2,1)
		(2,-1.2) node[right]{$j$} -- node[to start, rot, deco]{$\blacklozenge$} (2,1)
		(0,0) node{$\cdots$};$,
	and the following relations for the interaction with crossings:
	\settowidth{\algnRef}{$%
		(-1)^{d_{ij}/2}
		\tikz[diag] \draw
		(-1,-1) node[below]{$i$} --
		node[rot,label={[label distance=0pt, absolute]10:$d_{ij}$}, deco]{$\blacklozenge$}
		(-1,1)
		(1,-1) node[below]{$j$} -- (1,1);
		+
		(-1)^{d_{ji}/2}
		\tikz[diag] \draw
		(-1,-1) node[below]{$i$} -- (-1,1)
		(1,-1) node[below]{$j$} --
		node[label={[label distance=0pt, absolute]10:$d_{ji}$},deco,rot]{$\blacklozenge$}
		(1,1);
	$}
	\settowidth{\algnRefB}{$\displaystyle
	\phantom{-}
	\tikz[diag] \draw
	(-1,-1) node[below]{$i$} -- (-1,1)
	(1,-1) node[below]{$i$} -- (1,1);
	-
	\tikz[diag] \draw
	(-1,-1) node[below]{$i$} --
	node[rot, deco]{$\whitelozenge$}
	(-1,1)
	(1,-1) node[below]{$i$} --
	node[rot, deco]{$\whitelozenge$}
	(1,1);
	$
	}
	\begin{align}
		%
		\label{eq:qHC:sliding-black-diamonds}
		\tikz[diag] \draw
		  	(-1,-1) node[below]{$i$} -- (1,1)
		  		(1,-1) node[below]{$j$} --
		  			node[very near start, rot, deco]{$\blacklozenge$}
		  			node[deco]{$\whitebullet$}
		  			(-1,1);
		-
		\tikz[diag] \draw
		  	(-1,-1) node[below]{$i$} -- (1,1)
		  	(1,-1) node[below]{$j$} --
		  		node[very near end, rot, deco]{$\blacklozenge$}
		  		node[deco]{$\whitebullet$}
		  		(-1,1);
		&=\begin{cases}
			\mathmakebox[\algnRef][l]{\mathmakebox[\algnRefB]{0}}
			& \text{if $i\neq j$,}\\
			\mathmakebox[\algnRef][l]{%
				\phantom{-}
				\tikz[diag] \draw
					(-1,-1) node[below]{$i$} -- (-1,1)
						(1,-1) node[below]{$i$} -- (1,1);
				-
				\tikz[diag] \draw
					(-1,-1) node[below]{$i$} --
						node[rot, deco]{$\whitelozenge$}
						(-1,1)
					(1,-1) node[below]{$i$} --
						node[rot, deco]{$\whitelozenge$}
						(1,1);
			}
			& \text{if $i=j$,}
		\end{cases}
		\\
		\tikz[diag] \draw
		  	(-1,-1) node[below]{$i$} --
		  		node[very near start, rot, deco]{$\blacklozenge$}
		  		(1,1)
		  	(1,-1) node[below]{$j$} --
		  		node[deco]{$\whitebullet$}
		  		(-1,1);
		-
		\tikz[diag] \draw
		  	(-1,-1) node[below]{$i$} --
		  		node[very near end, rot, deco]{$\blacklozenge$}
		  		(1,1)
		  	(1,-1) node[below]{$j$} --
		  			node[deco]{$\whitebullet$}
		  			(-1,1);
		&=\begin{cases}
			\mathmakebox[\algnRef][l]{\mathmakebox[\algnRefB]{0}}
			& \text{if $i\neq j$,}\\
			\mathmakebox[\algnRef][l]{%
				-
				\tikz[diag] \draw
					(-1,-1) node[below]{$i$} -- (-1,1)
					(1,-1)node[below]{$i$} -- (1,1);
				-
				\tikz[diag] \draw
					(-1,-1) node[below]{$i$} --
						node[rot, deco]{$\whitelozenge$}
						(-1,1)
					(1,-1) node[below]{$i$} --
						node[rot, deco]{$\whitelozenge$}
						(1,1);
			}
			& \text{if $i=j$,}
		\end{cases}
		\\
		\label{eq:qHC:sliding-white-diamonds}
		\tikz[diag] \draw
		  	(-1,-1) node[below]{$i$} -- (1,1)
		  		(1,-1) node[below]{$j$} --
		  			node[very near start, rot, deco]{$\whitelozenge$}
		  			node[deco]{$\whitebullet$}
		  			(-1,1);
		-
		\tikz[diag] \draw
		  	(-1,-1) node[below]{$i$} -- (1,1)
		  		(1,-1) node[below]{$j$} --
		  			node[very near end, rot, deco]{$\whitelozenge$}
		  			node[deco]{$\whitebullet$}
		  			(-1,1);
		&= \mathmakebox[\algnRef][l]{
		\tikz[diag] \draw
		  	(-1,-1) node[below]{$i$} --
		  		node[very near start, rot, deco]{$\whitelozenge$}
		  		(1,1)
		  	(1,-1) node[below]{$j$} --
		  			node[deco]{$\whitebullet$}
		  			(-1,1);
		-
		\tikz[diag] \draw
		  	(-1,-1) node[below]{$i$} --
		  		node[very near end, rot, deco]{$\whitelozenge$}
		  		(1,1)
		  	(1,-1) node[below]{$j$} --
		  			node[deco]{$\whitebullet$}
		  			(-1,1);
			= 0
		}
		\casesGap
		\quad \text{for all $i$, $j$,}
		\\
		%
		%
		\label{eq:qHC:double-crossing}
		\begin{tikzpicture}[diag]
			\draw[curved, name path=stringA]
				(-1,-2) node[below]{$i$} 	to[out=+45, in=-45] (-1,2);
			\draw[curved, name path=stringB]
				(1,-2) node[below]{$j$} 	to[out=135,in=-135] (1,2);
			\draw[name intersections={of=stringA and stringB}]
				(intersection-1) node[deco]{$\whitebullet$}
				(intersection-2) node[deco]{$\whitebullet$};
		\end{tikzpicture}
		&= \begin{cases}
			\mathmakebox[\algnRef][l]{\mathmakebox[\algnRefB][c]{0}}%
			& \text{if $i=j$,}\\
			\mathmakebox[\algnRef][l]{\mathmakebox[\algnRefB][c]{%
				\tikz[diag] \draw (-1,-1) node[below]{$i$} -- (-1,1) (1,-1)node[below]{$j$} -- (1,1);
			}}
			& \text{if $i\disc j$,}\\
			\mathmakebox[\algnRef][l]{%
				(-1)^{d_{ij}/2}
				\tikz[diag] \draw
					(-1,-1) node[below]{$i$} --
						node[rot,label={[label distance=0pt, absolute]10:$d_{ij}$}, deco]{$\blacklozenge$}
						(-1,1)
					(1,-1) node[below]{$j$} -- (1,1);
				+
				(-1)^{d_{ji}/2}
				\tikz[diag] \draw
					(-1,-1) node[below]{$i$} -- (-1,1)
					(1,-1) node[below]{$j$} --
						node[label={[label distance=0pt, absolute]10:$d_{ji}$},deco,rot]{$\blacklozenge$}
						(1,1);
			}
			& \text{if $i\conn j$,}
		\end{cases}
		\\
		%
		%
		\begin{tikzpicture}[diag, scale=.6]
			\draw[name path=stringA]
				(-2,-2) 	node[below]{$i$} -- (2,2)
				(2,-2) 		node[below]{$k$} -- (-2,2)
				(0,0)		node[deco]{$\whitebullet$};
			\draw[name path=stringB]
				(0,-2) 		node[below]{$j$} to[out=135, in=-135, curved] (0,2);
			\draw[name intersections={of=stringA and stringB}]
				(intersection-1) node[deco]{$\whitebullet$}
				(intersection-2) node[deco]{$\whitebullet$};
		\end{tikzpicture}
		-
		\begin{tikzpicture}[diag, scale=.6]
			\draw[name path=stringA]
				(-2,-2) 	node[below]{$i$} -- (2,2)
				(2,-2) 		node[below]{$k$} -- (-2,2)
				(0,0)		node[deco]{$\whitebullet$};
			\draw[name path=stringB]
				(0,-2) 		node[below]{$j$} to[out=45, in=-45, curved] (0,2);
			\draw[name intersections={of=stringA and stringB}]
				(intersection-1) node[deco]{$\whitebullet$}
				(intersection-2) node[deco]{$\whitebullet$};
		\end{tikzpicture}
		&=
		\begin{cases}
			\mathmakebox[\algnRef][l]{\mathmakebox[\algnRefB][c]{%
				\tikz[diag] \draw
					(-2,-1) node[below]{$i$} --(-2,1)
					(0,-1) node[below]{$j$} --(0,1)
					(2,-1) node[below]{$i=k$} --(2,1);
			}}
			& \text{if $i=k$ and $i\conn j$,}\\
			\mathmakebox[\algnRef][l]{\mathmakebox[\algnRefB][c]{0}}
			& \text{otherwise.}
		\end{cases}
	\end{align}
	\end{itemize}
	For every odd index, $\HC_n(C)$ contains a copy of $\NHC_n$.
\end{definition}
\begin{lemma}%
	\label{lemma:quiver-Hecke-super-embedding}%
	Choose numbers $\gamma_{ij} \in k$ such that $\gamma_{ij}=1$ if at least one of $i$, $j$ has even parity,
		$\gamma_{ii}=\frac{1}{2}$ if $i$ has odd parity and $\gamma_{ij}\gamma_{ji} = - \frac{1}{2}$ otherwise.
	There is a morphism
	\begin{align*}
			\iota:\sH_n(C) &\longto \HC_n(C),\\*
			x_{k,\nu}
				& \longmapsto (\cl_{k,\nu})^{\pty{\nu_k}} y_{k,\nu}\\*
			\tau_{k,\nu}
				&\longmapsto \gamma_{\nu_k, \nu_{k+1}} (\cl_{k,\nu}-\cl_{k+1,\nu})^{\pty{\nu_k}\pty{\nu_{k+1}}} \sigma_{k,\nu};
	\end{align*}
	see \autocite[thm.\ 3.13]{KKT:Quiver-Hecke-Superalgebras}.
\end{lemma}
\begin{proof}
	We verify the well-definedness of this map.
	We have to check compatibility with the following relations:
	\begin{itemize}
	\item Compatibility with \eqref{eq:bullets crossings 1}:
		If one of the indices $i$ and $j$ has even parity,
		domain and codomain locally reduce to the ordinary KLR-algebra,
		so nothing remains to prove.
		Let us thus assume that both $i$ and $j$ have odd parity.
		\begingroup
		\tikzset{every diag/.append style={smaller}}
		\delimiterfactor=600
		\delimitershortfall=20ex
		\begin{multline}
				\tfrac 1{\gamma_{ij}}\iota(\tau_{i,\nu}x_{k+1,\nu}  + \tau_{i,\nu}x_{k+1,\nu})
				= \tfrac 1{\gamma_{ij}} \iota\left(
				\tikz[diag] \draw
					(-1,-1) node[below]{$i$} -- (1,1)
					(1,-1) node[below]{$j$} --
						node[very near start, rot, deco]{$\bullet$}
						(-1,1);
				+
				\tikz[diag] \draw
					(-1,-1) node[below]{$i$} -- (1,1)
					(1,-1) node[below]{$j$} --
						node[very near end, rot, deco]{$\bullet$}
						(-1,1);
				\right)
				=
				\left(
				\tikz[diag] \draw
					(-1,-1) node[below]{$i$} -- (1,1)
					(1,-1) node[below]{$j$} --
						node[very near  end, rot, deco]{$\whitelozenge$}
						node[near start, rot, deco]{$\whitelozenge$}
						node[very near start, rot, deco]{$\blacklozenge$}
						node[deco]{$\whitebullet$}
						(-1,1);
				-
				\tikz[diag] \draw
					(-1,-1) node[below]{$i$} --
						node[very near end, rot, deco]{$\whitelozenge$} (1,1)
					(1,-1) node[below]{$j$} --
						node[near start, rot, deco]{$\whitelozenge$}
						node[very near start, rot, deco]{$\blacklozenge$}
						node[deco]{$\whitebullet$}
						(-1,1);
				\right)
				+
				\left(
				\tikz[diag] \draw
					(-1,-1) node[below]{$i$} -- (1,1)
					(1,-1) node[below]{$j$} --
						node[very near end, rot, deco]{$\whitelozenge$}
						node[near end, rot, deco]{$\blacklozenge$}
						node[to end, rot, deco]{$\whitelozenge$}
						node[deco]{$\whitebullet$}
						(-1,1);
				-
				\tikz[diag] \draw
					(-1,-1) node[below]{$i$} -- node[very near end, rot, deco]{$\whitelozenge$} (1,1)
					(1,-1) node[below]{$j$} --
						node[to end, rot, deco]{$\whitelozenge$}
						node[near end, rot, deco]{$\blacklozenge$}
						node[deco]{$\whitebullet$}
						(-1,1);
				\right)
				\\
				=
				\left(
					\tikz[diag] \draw
						(-1,-1) node[below]{$i$} -- (1,1)
						(1,-1) node[below]{$j$} --
							node[very near start, rot, deco]{$\blacklozenge$}
							node[near end, rot, deco]{$\whitelozenge$}
						 	node[very near end, rot, deco]{$\whitelozenge$}
						 	node[deco]{$\whitebullet$}
						 	(-1,1);
					-
					\tikz[diag] \draw
						(-1,-1) node[below]{$i$} --
							node[very near end, rot, deco]{$\whitelozenge$}
							(1,1)
						(1,-1) node[below]{$j$} --
							node[very near start, rot, deco]{$\blacklozenge$}
						 	node[near end, rot, deco]{$\whitelozenge$}
						 	node[deco]{$\whitebullet$}
						 	(-1,1);
				\right)
				+
				\left(
					\tikz[diag] \draw
						(-1,-1) node[below]{$i$} -- (1,1)
						(1,-1) node[below]{$j$} to
							node[very near end, rot, deco]{$\whitelozenge$}
							node[near end, rot, deco]{$\blacklozenge$}
						 	node[to end, rot, deco]{$\whitelozenge$}
						 	node[deco]{$\whitebullet$}
						 	(-1,1);
					-
					\tikz[diag] \draw
						(-1,-1) node[below]{$i$} --
							node[to end, rot, deco]{$\whitelozenge$}
							(1,1)
						(1,-1) node[below]{$j$} to
							node[very near end, rot, deco]{$\whitelozenge$}
							node[near end, rot, deco]{$\blacklozenge$}
							node[deco]{$\whitebullet$}
						 	(-1,1);
				\right)
				\label{eq:compatibility bullet crossing}
			\end{multline}
			by \cref{eq:qHC:sliding-white-diamonds}.
			If $i\neq j$, then \cref{eq:qHC:sliding-black-diamonds} asserts
			that we can slide all diamonds across the crossing.
			According to the commutativity relations for
			$\tikz[diag, inline] \draw (0,-1) -- node[rot, deco]{$\blacklozenge$} (0,1);$'s
			and $\tikz[diag, inline] \draw (0,-1) -- node[rot, deco]{$\whitelozenge$} (0,1);$'s,
			we obtain for $i\neq j$:
			\begin{equation*}
				\tfrac 1{\gamma_{ij}}\iota(\tau_{i,\nu}x_{k+1,\nu}  + \tau_{i,\nu}x_{k+1,\nu})
				=
					\left(
					-
					\tikz[diag] \draw
						(-1,-1) node[below]{$i$} -- (1,1)
						(1,-1) node[below]{$j$} to
							node[very near end, rot, deco]{$\whitelozenge$}
							node[to end, rot, deco]{$\whitelozenge$}
						 	node[near end, rot, deco]{$\blacklozenge$}
						 	node[deco]{$\whitebullet$}
						 	(-1,1);
					+
					\tikz[diag] \draw
						(-1,-1) node[below]{$i$} --
							node[to end, rot, deco]{$\whitelozenge$}
							(1,1)
						(1,-1) node[below]{$j$} --
							node[near end, rot, deco]{$\blacklozenge$}
						 	node[very near end, rot, deco]{$\whitelozenge$}
						 	node[deco]{$\whitebullet$}
						 	(-1,1);
					\right)
					+
					\left(
					\tikz[diag] \draw
						(-1,-1) node[below]{$i$} -- (1,1)
						(1,-1) node[below]{$j$} to
							node[very near end, rot, deco]{$\whitelozenge$}
							node[to end, rot, deco]{$\whitelozenge$}
						 	node[near end, rot, deco]{$\blacklozenge$}
						 	node[deco]{$\whitebullet$}
						 	(-1,1);
					-
					\tikz[diag] \draw
						(-1,-1) node[below]{$i$} --
							node[to end, rot, deco]{$\whitelozenge$}
							(1,1)
						(1,-1) node[below]{$j$} --
							node[near end, rot, deco]{$\blacklozenge$}
						 	node[very near end, rot, deco]{$\whitelozenge$}
						 	node[deco]{$\whitebullet$}
						 	(-1,1);
					\right)
				= 0.
			\end{equation*}
			If $i=j$, we obtain from \cref{eq:qHC:sliding-black-diamonds}:
			\begin{multline*}
				\tfrac 1{\gamma_{ij}}\iota(\tau_{i,\nu}x_{k+1,\nu}  + \tau_{i,\nu}x_{k+1,\nu})
				=
				\left(
					\tikz[diag] \draw
						(-1,-1) node[below]{$i$} -- (1,1)
						(1,-1) node[below]{$i$} to
							node[very near end, rot, deco]{$\whitelozenge$}
							node[near end, rot, deco]{$\whitelozenge$}
						 	node[to end, rot, deco]{$\blacklozenge$}
						 	node[deco]{$\whitebullet$}
						 	(-1,1);
					+
					\tikz[diag] \draw
						(-1,-1) node[below]{$i$} --
							node[pos=.625,rot, deco]{$\whitelozenge$}
							node[pos=.325,rot, deco]{$\whitelozenge$}
							(-1,1)
						 (1,-1) node[below]{$i$} --
						 	(1,1);
					-
					\tikz[diag] \draw
						(-1,-1) node[below]{$i$} --
							node[near start,rot, deco]{$\whitelozenge$}
							node[midway,rot, deco]{$\whitelozenge$}
							node[near end,rot, deco]{$\whitelozenge$}
							(-1,1)
						 (1,-1) node[below]{$i$} --
							node[near start,rot, deco]{$\whitelozenge$}
						 	(1,1);
					-
					\tikz[diag] \draw
						(-1,-1) node[below]{$i$} --
							node[very near end, rot, deco]{$\whitelozenge$}
							(1,1)
						(1,-1) node[below]{$i$} to
							node[near end, rot, deco]{$\blacklozenge$}
						 	node[to end, rot, deco]{$\whitelozenge$}
						 	node[deco]{$\whitebullet$}
						 	(-1,1);
					-
					\tikz[diag] \draw
						(-1,-1) node[below]{$i$} --
							node[pos=.325,rot, deco]{$\whitelozenge$}
							(-1,1)
						 (1,-1) node[below]{$i$} --
							node[pos=.625,rot, deco]{$\whitelozenge$}
						 	(1,1);
					+
					\tikz[diag] \draw
						(-1,-1) node[below]{$i$} --
							node[pos=.325,rot, deco]{$\whitelozenge$}
							node[pos=.625,rot, deco]{$\whitelozenge$}
							(-1,1)
						 (1,-1) node[below]{$i$} --
						 	node[very near start,rot, deco]{$\whitelozenge$}
							node[very near end,rot, deco]{$\whitelozenge$}
						 	(1,1);
				\right) + {} \\{} +
				\left(
					-
					\tikz[diag] \draw
					(-1,-1) node[below]{$i$} -- (1,1)
					(1,-1) node[below]{$i$} to
						node[very near end, rot, deco]{$\whitelozenge$}
						node[near end, rot, deco]{$\whitelozenge$}
					 	node[to end, rot, deco]{$\blacklozenge$}
					 	node[deco]{$\whitebullet$}
					 	(-1,1);
					 +
					 \tikz[diag] \draw
						(-1,-1) node[below]{$i$} --
							node[very near end, rot, deco]{$\whitelozenge$}
							(1,1)
						(1,-1) node[below]{$i$} to
							node[near end, rot, deco]{$\blacklozenge$}
					 		node[to end, rot, deco]{$\whitelozenge$}
					 		node[deco]{$\whitebullet$}
					 		(-1,1);
				\right)
				=2
				\tikz[diag] \draw
					(-1,-1) node[below]{$i$} --
						(-1,1)
					 (1,-1) node[below]{$i$} --
					 	(1,1);.
	 	\end{multline*}
		\endgroup%
	\item Compatibility with \eqref{eq:bullets crossings 2}: Similar.
	\item Compatibility with \eqref{eq:double crossing}:
		If $i=j$, this relation is obviously satisfied.
		For $i\neq j$, nothing remains to be proven if one index is even;
		we thus assume $|i||j|=1$.
		\begingroup
		\tikzset{every diag/.append style={smaller}}
		\delimiterfactor=700
		\delimitershortfall=20ex
		\begin{multline*}
			\underbrace{\frac{1}{\gamma_{ij}\gamma_{ji}}}_{-2}
			\iota\left(
				\begin{tikzpicture}[diag,smaller,smaller,smaller]
					\draw[curved, name path=stringA]
						(-1,-2) node[below]{$i$} 	to[out=+45, in=-45] (-1,2);
					\draw[curved, name path=stringB]
						(1,-2) node[below]{$j$} 	to[out=135,in=-135] (1,2);
				\end{tikzpicture}
			\right)
			=
			\frac{1}{\gamma_{ji}}
			\left(
				\tikz[diag] \draw  (-1,-1) -- (1,1) (1,-1) -- node[deco]{$\whitebullet$} node[deco, near end, rot]{$\whitelozenge$}(-1,1);
				- \tikz[diag] \draw  (-1,-1) -- node[deco, near end, rot]{$\whitelozenge$} (1,1) (1,-1) -- node[deco]{$\whitebullet$} (-1,1);
			\right)
			\iota\left(\tikz[diag] \draw (-1,-1) -- (1,1) (1,-1) -- (-1,1); \right)
			=
			\left(
				\begin{tikzpicture}[diag]
					\draw[curved, name path=stringA]
						(-1,-2) node[below]{$i$} 	to[out=+45, in=-45] node[deco,rot, very near end]{$\whitelozenge$}(-1,2);
					\draw[curved, name path=stringB]
						(1,-2) node[below]{$j$} 	to[out=135,in=-135] node[deco,rot]{$\whitelozenge$}(1,2);
					\draw[name intersections={of=stringA and stringB}]
						(intersection-1) node[deco]{$\whitebullet$}
						(intersection-2) node[deco]{$\whitebullet$};
				\end{tikzpicture}
				-
				\begin{tikzpicture}[diag]
					\draw[curved, name path=stringA]
						(-1,-2) node[below]{$i$} 	to[out=+45, in=-45] (-1,2);
					\draw[curved, name path=stringB]
						(1,-2) node[below]{$j$} 	to[out=135,in=-135]
							node[deco,rot]{$\whitelozenge$}
							node[deco,rot,very near end]{$\whitelozenge$}
							(1,2);
					\draw[name intersections={of=stringA and stringB}]
						(intersection-1) node[deco]{$\whitebullet$}
						(intersection-2) node[deco]{$\whitebullet$};
				\end{tikzpicture}
			\right)+\left(
				\begin{tikzpicture}[diag]
					\draw[curved, name path=stringA]
						(-1,-2) node[below]{$i$} 	to[out=+45, in=-45]
							node[deco,rot]{$\whitelozenge$}
							node[deco,rot,very near end]{$\whitelozenge$}
							(-1,2);
					\draw[curved, name path=stringB]
						(1,-2) node[below]{$j$} 	to[out=135,in=-135] (1,2);
					\draw[name intersections={of=stringA and stringB}]
						(intersection-1) node[deco]{$\whitebullet$}
						(intersection-2) node[deco]{$\whitebullet$};
				\end{tikzpicture}
				+
				\begin{tikzpicture}[diag]
					\draw[curved, name path=stringA]
						(-1,-2) node[below]{$i$} 	to[out=+45, in=-45]
							node[deco,rot]{$\whitelozenge$}(-1,2);
					\draw[curved, name path=stringB]
						(1,-2) node[below]{$j$} 	to[out=135,in=-135]
							node[deco,rot,very near end]{$\whitelozenge$} (1,2);
					\draw[name intersections={of=stringA and stringB}]
						(intersection-1) node[deco]{$\whitebullet$}
						(intersection-2) node[deco]{$\whitebullet$};
				\end{tikzpicture}
			\right)
			\\
			=
				\begin{tikzpicture}[diag]
					\draw[curved, name path=stringA]
						(-1,-2) node[below]{$i$} 	to[out=+45, in=-45] node[deco,rot,pos=0.95]{$\whitelozenge$}(-1,2);
					\draw[curved, name path=stringB]
						(1,-2) node[below]{$j$} 	to[out=135,in=-135] node[deco,rot,very near end]{$\whitelozenge$}(1,2);
					\draw[name intersections={of=stringA and stringB}]
						(intersection-1) node[deco]{$\whitebullet$}
						(intersection-2) node[deco]{$\whitebullet$};
				\end{tikzpicture}
				-2
				\begin{tikzpicture}[diag]
					\draw[curved, name path=stringA]
						(-1,-2) node[below]{$i$} 	to[out=+45, in=-45] (-1,2);
					\draw[curved, name path=stringB]
						(1,-2) node[below]{$j$} 	to[out=135,in=-135]
							(1,2);
					\draw[name intersections={of=stringA and stringB}]
						(intersection-1) node[deco]{$\whitebullet$}
						(intersection-2) node[deco]{$\whitebullet$};
				\end{tikzpicture}
				-
				\begin{tikzpicture}[diag]
					\draw[curved, name path=stringA]
						(-1,-2) node[below]{$i$} 	to[out=+45, in=-45]
							node[deco,rot, pos=0.95]{$\whitelozenge$}
							(-1,2);
					\draw[curved, name path=stringB]
						(1,-2) node[below]{$j$} 	to[out=135,in=-135]
							node[deco,rot,very near end]{$\whitelozenge$}
							 (1,2);
					\draw[name intersections={of=stringA and stringB}]
						(intersection-1) node[deco]{$\whitebullet$}
						(intersection-2) node[deco]{$\whitebullet$};
				\end{tikzpicture}
		\end{multline*}
		because we may drag a $\tikz[diag, inline] \draw (0,-1) -- node[rot, deco]{$\whitelozenge$} (0,1);$ across crossings by \cref{eq:qHC:sliding-white-diamonds}.
		If $i\disc j$, this equals
		$-2\tikz[diag, scale=1.25, inline] \draw (-1,-1) node[left]{$i$} -- (-1,1) (1,-1)node[right]{$j$} -- (1,1);$
		because we may resolve the double crossings by \cref{eq:qHC:double-crossing}.
		If $i\conn j$, we calculate
		\begin{multline*}
				\begin{tikzpicture}[diag]
					\draw[curved, name path=stringA]
						(-1,-2) node[below]{$i$} 	to[out=+45, in=-45] (-1,2);
					\draw[curved, name path=stringB]
						(1,-2) node[below]{$j$} 	to[out=135,in=-135]
							(1,2);
					\draw[name intersections={of=stringA and stringB}]
						(intersection-1) node[deco]{$\whitebullet$}
						(intersection-2) node[deco]{$\whitebullet$};
				\end{tikzpicture}
			=
				(-1)^{\frac{d_{ij}}{2}}
				\cdot
				\tikz[diag] \draw
					(-1,-1) node[below]{$i$} --
						node[rot,deco,label=-80:$d_{ij}$]{$\blacklozenge$}
						(-1,1)
					(1,-1) node[below]{$j$} -- (1,1);
				+
				(-1)^{\frac{d_{ji}}{2}}
				t_{ji}
				\tikz[diag] \draw
					(-1,-1) node[below]{$i$} -- (-1,1)
					(1,-1) node[below]{$j$} --
						node[deco,rot,label=-80:$d_{ji}$]{$\blacklozenge$}
						(1,1);
			=
				\tikz[diag] \draw
					(-1,-1) node[below]{$i$} --
						node[rot, deco,pos=.625]{$\whitelozenge$}
						node[rot, deco,pos=.375]{$\blacklozenge$}
						node{$\biggl(\quad\biggr)^{\mathrlap{d_{ij}}}$}
						(-1,1)
					(1,-1) node[below]{$j$} -- (1,1);
				+
				\tikz[diag] \draw
					(-1,-1) node[below]{$i$} -- (-1,1)
					(1,-1) node[below]{$j$} --
						node[rot, deco,pos=.625]{$\whitelozenge$}
						node[rot, deco,pos=.375]{$\blacklozenge$}
						node{$\biggl(\quad\biggr)^{\mathrlap{d_{ji}}}$}
						(1,1);
			=
				\iota\left(
				\tikz[diag] \draw
					(-1,-1) node[below]{$i$} --
						node[deco]{$\bullet$}
						node[deco]{$\smash{\phantom{{}^{d_{ij}}\bullet}^{d_{ij}}}\vphantom{\bullet}$}
						(-1,1)
					(1,-1) node[below]{$j$} -- (1,1);
				+
				\tikz[diag] \draw
					(-1,-1) node[below]{$i$} -- (-1,1)
					(1,-1) node[below]{$j$} --
						node[deco]{$\bullet$}
						node[deco]{$\smash{\phantom{{}^{d_{ij}}\bullet}^{d_{ij}}}\vphantom{\bullet}$}
						(1,1);
				\right).
		\end{multline*}
		\endgroup
	\item Compatibility with \eqref{eq:braid relation}:
		Assume \wlofg\ that $k=1$.
		The braid relation then follows from the computation
		\begin{align*}
			&\eqsp\tau_{1,s_2s_1(\nu)} \tau_{2,s_1(\nu)} \tau_{1,\nu}\\
			&=	(\cl_{1}-\cl_{2})^{\pty{s_2s_1(\nu)_1}\pty{s_2s_1(\nu)_2}} \sigma_1\,
				(\cl_{2}-\cl_{3})^{\pty{s_1(\nu)_2}\pty{s_1(\nu)_3}} \sigma_2\,
				(\cl_{1}-\cl_{2})^{\pty{\nu_1}\pty{\nu_2}} \sigma_{1,\nu}\\
			&=  (\cl_{1}-\cl_{2})^{\pty{\nu_1}\pty{\nu_2}}
				(\cl_{1}-\cl_{3})^{\pty{\nu_1}\pty{\nu_3}}
				(\cl_{2}-\cl_{3})^{\pty{\nu_2}\pty{\nu_3}}
				\sigma_1 \sigma_2 \sigma_{1,\nu}\\
			&\stackrel{*}{=}
				(\cl_{2}-\cl_{3})^{\pty{\nu_2}\pty{\nu_3}}
				(\cl_{1}-\cl_{3})^{\pty{\nu_1}\pty{\nu_3}}
				(\cl_{1}-\cl_{2})^{\pty{\nu_1}\pty{\nu_2}}
				\sigma_2 \sigma_1 \sigma_{2,\nu}\\
			&=	(\cl_{2}-\cl_{3})^{\pty{s_1s_2(\nu)_2}\pty{s_1s_2(\nu)_3}} \sigma_2\,
				(\cl_{1}-\cl_{2})^{\pty{s_2(\nu)_1}\pty{s_2(\nu)_2}} \sigma_1\,
				(\cl_{2}-\cl_{3})^{\pty{\nu_2}\pty{\nu_3}} \sigma_{2,\nu}\\
			&=	\tau_{2,s_1s_2(\nu)}\tau_{1,s_2(\nu)}\tau_{2,\nu}.
		\end{align*}
		The equation $*$ is clear if at least one index is even
		and is verified by multiplying out the $\cl_k$'s if every index is odd. \qedhere
	\end{itemize}
\end{proof}

\subsection{Faithful polynomial representation}
\label{sec:faithful-super-polynomial-rep}
We define a representation of the quiver Hecke Clifford superalgebra $\NHC_n(C)$
as an analogue to the construction in \autocite[§2.3]{KL:quantum-group-I}.
\begin{definition}
	Let $\PolC_n(C)$ be the $k$-linear supercategory defined as follows.
	\begin{itemize}
		\item Its objects are the free $k$-super vector spaces $\PolC_\nu ≔ k[y_{1,\nu},\ldots, y_{n,\nu}, \cl_{1,\nu},\dotsc, \cl_{n,\nu}]$,
		indexed by sequences $\nu\in I^n$.
		The symbols $y_{k,\nu}$ and $\cl_{k,\nu}$ have the same parities
		and satisfy the same relations
		as in \cref{def:super-klr-algebra:HC(C):diagramatically}.
		\item Its morphisms are super vector space homomorphisms.
	\end{itemize}
\end{definition}
Given a sequence $\nu\in I^n$,
$\PolC_n(C)$ and $\HC_n(C)$ both have full subcategories $\PolC(\nu)$ and $\NH_n(\nu)$
with objects $\{\PolC_{\mu}\mid \mu∈S_n\nu\}$
and string diagrams between these, respectively.

\begin{proposition}
	\label{prop:super-polynomials-are-faithful-rep}
	$\HC(\nu)$ acts faithfully on $\PolC(\nu)$,
		where $y_{k,\nu}$ and $\cl_{k,\nu}$ act by multiplication and $\sigma_{k,\nu}$ acts by
			\begin{equation}
			\label{eq:Definition-of-Faithful-Action}
				\sigma_{k,\nu} =
				\begin{cases}
					s_k 			& \text{if $\nu_k\disc \nu_{k+1}$ or $\nu_k\leftarrow\nu_{k+1}$,}\\
					\delC_k 		& \text{if $\nu_k = \nu_{k+1}$,}\\
					\begin{multlined}
						\Big(
							(\cl_{k,\nu}y_{k,\nu})^{d_{\nu_k,\nu_{k+1}}} 
							+(\cl_{k+1,\nu}y_{k+1,\nu})^{d_{\nu_{k+1},\nu_k}}
						\Big)
						s_k
					\end{multlined}
									& \text{if $\nu_k\rightarrow\nu_{k+1}$.}
			\end{cases}
		\end{equation}
\end{proposition}
\begin{proof}
	For the non-super case,
	faithfulness of the polynomial representation of the quiver Hecke algebra
	has been proven in \autocite[§2.3]{KL:quantum-group-I};
	Our super-algebraic setup necessitates only minor alterations,
	which we make explicit in the following.

	\proofsection{Spanning set}
	As a vector space, $\HC_n(\nu)$
	is isomorphic to a direct sum $⨁_{\mu, \mu' ∈ S_n\ldot\nu}({}_{\mu'}\!\HC_{\mu})$
	of subspaces ${}_{\mu'}\!\HC_{\mu}$,
	each of which contains diagrams connecting the bottom sequence $\mu$
	with the top sequence $\mu'$.
	Let ${}_{\mu'}S_\mu$ be the subgroup ${}_{\mu'}S_\mu = \{w ∈ S_n \mid \mu_{w(k)} = \mu'_k\}$ of$S_n$.

	Given a diagram in ${}_{\mu'}\!\HC_{\mu}$,
	assume there are two strings intersecting more than once.
	According to the relation \cref{eq:qHC:double-crossing},
	it can be replaced by a linear combination of diagrams
	with fewer crossings, possibly at the cost of introducing more diamonds.
	Diamonds can be slid across crossings
	by \crefrange{eq:qHC:sliding-black-diamonds}{eq:qHC:sliding-white-diamonds},
	possibly at the cost of introducing	summands with less crossings,
	and past other diamonds, possibly introducing signs.

	The vector space ${}_{\mu'}\HC_{\mu}$ therefore is spanned by diagrams
	with any two strands intersecting at most once,
	all black diamonds below all crossings,
	and all white diamonds below the black ones
	such that there is at most one white diamond per string.
	The white ones may be arranged descending to the right.
	A typical diagram in this spanning set looks as follows:
	
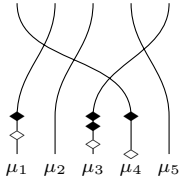
\begin{equation*}
		\tikz[diag]\draw
			(0,-2) node [below]{$\mu_1$} --
				node [rot, at end, deco]{$\blacklozenge$}
				node [rot, midway, deco]{$\whitelozenge$}
				(0,-1)
				to[out=90,in=-90] (1,2)
			(1,-2) node [below]{$\mu_2$}  -- (1,-1) to[out=90,in=-90]  (2,2)
			(2,-2) node [below]{$\mu_3$}  --
				node [rot, at end, deco]{$\blacklozenge$}
				node [rot, near end, deco]{$\blacklozenge$}
				node [rot, to start, deco]{$\whitelozenge$}
				(2,-1)
				to[out=90,in=-90](4,2)
			(3,-2) node [below]{$\mu_4$} --
				node [rot, at end, deco]{$\blacklozenge$}
				node [rot, near start, deco]{$\whitelozenge$}
				(3,-1)
				 to[out=90,in=-90] (0,2)
			(4,-2)  node [below]{$\mu_5$} -- (4,-1)
				to[out=90, in=-90] (3,2);
	\end{equation*}
	This spanning set can be written as
	\newcommand{\BC}{\operatorname{\mathsf{B}\mathfrak{C}}}
	\begin{equation}
			{}_{\mu'}\BC_\mu = \{
				\sigma_{\pi,\mu}^{}
				y_{\mu}^{\bm{\alpha}}
				\cl_{\mu}^{\bm{\beta}}
			\}
			=\{
				\sigma_{\pi,\mu'}^{}
				y_{1,\mu}^{\alpha_1}\dotsm y_{n,\mu}^{\alpha_n}\,
				\cl_{1,\mu}^{\beta_1}\dotsm \cl_{n,\mu}^{\beta_n}
			\},
	\end{equation}
	where $\mu$ and $\mu'$ are permutations of $\nu$,
	$\pi$ is a reduced expression for an element of ${}_{\mu'}S_\mu$,
	$\bm{\alpha}∈\mathbf{N}^n$, $\bm{\beta}∈\{0,1\}^n$ are multiindices,
	and $\cl_\mu^{\bm{\beta}}$ is an ordered monomial.
	Choose a complete order $\leq$ on $I$ such that $i<j$ whenever there is an edge $i\to j$.
	This order induces a lexicographic order on $I^n$.
	We show by induction on $\mu$ \wrt\ this order on $I^n$
	that ${}_{\mu'}\BC_\mu$ is a $k$-basis on which $\HC(\nu)$ acts faithfully.

	\proofsection{Base of induction}
	Let $n_i$ be the number of entries $i$ in $\nu$.
	Let
	\[
		\mu=(\underbrace{i_1, \dotsc, i_1}_{n_{1}}, \underbrace{i_2,\dotsc,i_2}_{n_2}, \dotsc)
	\]
	such that $i_1<i_2<\dotsb$.
	The tuple $\mu$ is the lowest element in the orbit $S_n\nu$ \wrt\ the order $\leq$.
	We may write a permutation $w ∈ {}_{\nu'}S_\nu$ as $w=uv$
	such that $u ∈ S_{n_1} \times S_{n_2} \times \dotsb$ permutes the $i_1$, $i_2$ etc.\ independently,
	and $v$ is of minimal length,
	\ie, it does not interchange identically labelled strings and interchanges distinct labels at most once.
	The spanning set can thus be written as
	\[
		{}_{\nu'}\BC_\nu = \{ \sigma_{u,\mu}^{} \sigma_{v,\mu}^{} y_\mu^{\bm{\alpha}}  \cl_\mu^{\bm{\beta}}
		\}.
	\]
	We make sure that these elements act linearly independently on $\PolC(\nu)$.
	\begin{itemize}
		\item The terms $y_\mu^{\bm{\alpha}}\cl_\mu^{\bm{\beta}}$ take $y_\mu^{\bm{\alpha}'}\cl_\mu^{\bm{\beta}'}$
		to $\pm y_\mu^{\bm{\alpha}+\bm{\alpha}'} \cl_\mu^{\bm{\beta}+\bm{\beta}'}$.
		\item Since $v$ permutes strings with distinct labels,
			$\sigma_{v,\mu}$ acts by the permutation $v$.
		\item Since the permutation $u = u_1\times\cdots\times u_m ∈ S_{n_1} \times\dotsb\times S_{n_m}$
			permutes only identically labelled strings,
			$\sigma_{u,\nu}$ acts by
			$\delC_{u_1}\times\cdots\times\delC_{u_n}\in\NCC_{n_1}\times\NCC_{n_2}\times\dotsb\times\NCC_{n_m}$.
			We shall prove in \cref{lem:Clifford-NilHecke-acts-Faithfully}
			that the action of $\NCC_n$ on polynomials is faithful.
	\end{itemize}

	\proofsection{Induction step}
	It suffices to show that ${}_{\mu'}\BC_\mu$ is linearly independent.
	We have that ${}_{s_k(\mu')}\BC_\mu$ is linearly independent
	if $\mu'_k$, $\mu'_{k+1}$ are distinct and connected by an edge;
	otherwise, ${}_{s_k(\mu')}\BC_\mu$ maps bijectively to ${}_{\mu'}\BC_\mu$ by $\sigma_{k,\mu'}$.
	The multiplication map $(\sigma_k\cdot): {}_{s_k(\mu')}\BC_\nu \into {}_{\mu'}\BC_\mu$
	is seen to be injective by the same argument as in \autocite{KL:quantum-group-I}:

	Let ${}_{\mu'}\BC_\mu$ be endowed with a partial order $\preceq$
	such that if we assign to an element $\sigma_{w,\mu}  y_\mu^{\bm{\alpha}} \cl_\mu^{\bm{\beta}}$
	of the spanning set the tuple $(\ell(w), \bm{\alpha}, \bm{\beta})$,
	then the order $\preceq$ coincides with the lexicographic ordering $\leq$ on these tuples.
	Define a map
	\begin{equation}
		\begin{split}
			\varsigma\colon  {}_{s_k(\mu')}\BC_\mu &\to {}_{\mu'}\BC_\mu\\
			D
			& \mapsto
			\begin{cases}
				s_k D &
				\text{if the $k$-th and $(k+1)$-st strand do not intersect,}\\
				t_{\mu_k,\mu_{k+1}} D^* (\cl_{k,\mu}y_{k,\mu})^{d_{\mu_k,\mu_{k+1}}}  &
				\text{otherwise,}
			\end{cases}
		\end{split}
	\end{equation}
	where $D^*$ is obtained from $D$ by removing the crossing.
	The map $\varsigma$ is injective,
	and multiplication by $\sigma_{k,\mu'}$ satisfies
	\[
		\sigma_{k,\mu'} \cdot D ∈ \{±\varsigma(D)\} + \sum_{\mathclap{D'\prec\varsigma(D)}} \mathbf ZD'
		\subseteq {}_{\mu'}\BC_\mu.
	\]
	By the induction hypothesis on ${}_{\mu'}\BC_\mu$,
	we obtain that the multiplication map $(\sigma_{k,\mu'}⋅)$ must be injective.
\end{proof}
\begin{definition}
	We endow $\PolC(\nu)$ with a polynomial grading
	by setting $\deg(y_{k,\mu})=1$ and $\deg(\cl_{k,\mu})=0$.
	Additionally setting $\deg(\sigma_{k,\mu})=-1$,
	we also endow $\HC(\nu)$ with a grading.
\end{definition}
\begin{corollary}
	With this grading, $\PolC(\nu)$ is a faithful graded $\HC(\nu)$-module.
\end{corollary}

\section{Clifford symmetric Polynomials}
\label{sec:clifford-symmetric-polynomials}
\newcommand{\fCl}{\hat{\Cl}}
Consider a $\sZ$-graded index set $I=\{1,\dotsc, n\} = I_0\sqcup I_1$ as above.
From now on,
we consider the quotient superalgebra $\Cl_{I}≔\Cl_n/(\cl_i)_{i\in I_0}$ by the two-sided ideal $(\cl_i)_{i\in I_0}$.
$\PolC_{I}$ and $\delC_i$ are defined analogously to \cref{def:Polynomial-Clifford-Algebra}.

\begin{definition}
	By \cref{lem:super-klr-algebra:clifford-demazure-relations},
	$\SPolC_{I}≔⋂_{k=1}^{n-1}\ker \delC_k = ⋂_{k=1}^{n-1}\im \delC_i$
	is a $\Cl_{I}$-subalgebra of $\PolC_n(I)$.
	We call $\SPolC_{I}$ the algebra of \term{$\delC$-symmetric polynomials}%
	\footnote{%
		One may be tempted to call these supersymmetric polynomials;
		this term has already been coined for another notion though
		\autocite{Stembridge:supersymmetric-polynomials}.
	}.
\end{definition}

In this section,
we want to investigate some of their properties.
To this end, we shall introduce the notion of elementary $\delC$-symmetric polynomials
and show in \cref{prop:Clifford-Polynomials-free} that they generate $\SPolC_{I}$ as  $\Cl_{I}$-superalgebra.

%

\subsection{Interlude: counting graded ranks}
\label{sec:graded-ranks}
\begin{definition}
	Given a graded ring $R$ and a free $\mathbf{Z}$-graded $R$-module $M=⨁_{i∈\mathbf{Z}} M_i$,
	its \term{graded rank} or \term{Poincaré series}
	is the formal power series $\rk_{q,R}(M) ≔ \sum_{i∈\mathbf{Z}} \rk_R (M_i) q^i$,
	seen as an element of the localisation $\mathbf{Z}[[q]]_{(1-q)}$.
	We write $\rk_{q,\mathbf{Z}}$ for the graded rank
	of graded free abelian groups.
	With respect to the field $k$, we denote the graded dimension by $\dim_{q,k}$.
\end{definition}

\begin{definition}
	Let $(W,S)$ be a Coxeter system with a fixed generating set $S$,
	which determines a length function $\ell$ on $W$.
	Its \emph{$q$-order} is $\ord_q(W) ≔ \sum_{w∈W} q^{\ell(w)}$.
\end{definition}

We compute some examples that we shall use later on.
Note that we are working with classical, non-super polynomial rings in this section.
Since $1-q$ is a unit in $\mathbf{Z}[[q]]_{(1-q)}$,
the \term{q-integers} $(n)_q ≔ 1 + \dotsb + q^{n-1}$ can be expressed as $(n)_q = \frac{1-q^n}{1-q}$;
see \autocite{LeStum:q-Integers} for a general introduction.

\begin{lemma}[{\autocite[cf.][§1]{Winkel:Poincare}}]
	\label{lem:poincare-polynomials-coxeter-groups}
	The Coxeter group of type $\mathrm{A}_n$,  $\mathrm{BC}_n$ and  $\mathrm{D}_n$
	have the following $q$-orders:
	\begin{thmlist}
		\item $\ord_q(\mathrm{A}_n) = (n)_q!$ with the \term{$q$-factorial} $(n)_q! ≔ (n)_q (n-1)_q\dotsm (1)_q$,
		\item $\ord_q(\mathrm{BC}_n) = (2n)_q!!$ with the \term{$q$-double factorial} $(2n)_q!! ≔ (2n)_q (n-2)_q \dotsm (2)_q$,
		\item $\ord_q(\mathrm{D}_n) = (2n-2)_q!!(n)_q$.
	\end{thmlist}
	We have included $\ord_q(\mathrm{BC}_n)$	and $\ord_q(\mathrm{D}_n)$ for the sake of completeness
	but will not need them in the following.
\end{lemma}
\begin{proof}
	\newcommand{\inv}{\operatorname{inv}}
	\begin{prooflist}
	\item
	\label{proof:orders-of-Coxeter-groups:Sn}
	The length $\ell(w)$ of a permutation $w∈S_n$
	equals the number of \term{inversions} \autocite[prop.\ 1.5.2]{Brenti:Coxeter-Groups},
	\ie, the cardinality of $\inv w ≔ \{(i,j) \mid 1\leq i<j\leq n, w(i)>w(j)\}$.
	Assume that $\sum_{w∈S_n} p(w) = (n)_q!$.
	Consider the permutations
	\[
		\begin{split}
			\pi_0\colon  (1,\dotsc, n+1) &\mapsto (1,\dotsc, n , n+1),\\
			\pi_1\colon  (1,\dotsc, n+1) &\mapsto (1,\dotsc , n+1,  n),\\
			\shortvdotswithin{\mapsto}
			\pi_{n}\colon  (1,\dotsc, n+1) &\mapsto (n+1, 1,\dotsc, n)
		\end{split}
	\]
	contained in $S_{n+1}$, where $\pi_0=e$ is the trivial permutation.
	Moving the rightmost entry in $\pi_0$ to the left will subsequently create new inversions
	such that $\pi_k$ has the $k$ inversions
	\[
		\inv\pi_k = \bigl\{(n-k+1,n-k+2), \dotsc, (n-k+1, n+1)\bigr\},
	\]
	and $\sum_{k} q^{\ell(\pi_k)} = (1+q+\dotsb q^{n}) = (n+1)_q$.
	Since any $w∈S_n$ interchanges only the first $n$ slots,
	the inversions of $w$ and $\pi_k$ do not interfere,
	which gives $\inv (\pi_kw) = w^{-1}(\inv \pi_k) \sqcup \inv(w)$.
	Therefore, $\sum_{k} q^{\ell(\pi_k w)}=(n+1)_q q^{\ell(w)}$.
	By the induction hypothesis,
	letting $w$ traverse all of $S_n$ proves the claim.

	\item
	The Coxeter group $\mathrm{BC}_n$,
	called the \term{hyperoctahedral group} or \term{signed permutation group},
	is the group of permutations $\pi$ of the set $\{\pm 1,\dotsc, \pm n\}$
	such that $\pi(-k)=-\pi(k)$.
	It has as generators the simple transpositions $s_k\colon  \pm k\mapstofrom \pm k+1$ for $1\leq k\leq n-1$
	and the additional generator $s_0\colon  1\mapstofrom -1$.
	A signed permutation $w$ has
	\[
		\inv(w) ≔
		\bigl\{ (i,j) \bigm| 1\leq i<j \leq n, w(i)>w(j)\bigr\}
		\cup
		\bigl\{ (-i,j) \bigm| 1\leq i\leq j \leq n, w(-i)>w(j)\bigr\}.
	\]
	We count the number of inversions as in \cref{proof:orders-of-Coxeter-groups:Sn}.
	Assume that $\ord_q \mathrm{BC}_n = (n)_q!!$,
	and consider the permutations
	\begin{align*}
			\pi_0\colon  (1,\dotsc, n+1) &\mapsto (1,\dotsc, n , n+1),\\*
			\shortvdotswithin{\mapsto}
			\pi_n\colon  (1,\dotsc, n+1) &\mapsto (n+1, 1,\dotsc, n),\\
			\pi_{n+1}\colon  (1,\dotsc, n+1) &\mapsto (-n-1, 1,\dotsc, n),\\*
			\shortvdotswithin{\mapsto}
			\pi_{2n+1}\colon  (1,\dotsc, n+1) &\mapsto (1,\dotsc, w(n), -n-1).
	\end{align*}
	These have
	\[
		\inv \pi_m =
		\begin{cases}
		\bigl\{(n-m+1, n-m+2),\dotsc, (n-m+1, n+1)\bigr\} & \text{if $0\leq m\leq n$,}\\[1em]
		\begin{multlined}[b][\widthof{$\bigl\{(n-m+1, n-m+2),\dotsc, (n-m+1, n+1)\bigr\}$}]
		\bigl\{
				\underbrace{(1,l),\dotsc, (l-1,l)}_{\text{$l-1$ many}},
				\underbrace{(-1,l),\dotsc, (l+1,l)}_{\text{$l-1$ many}}\\[-1.4em]
				\underbrace{(-l,l),\dotsc, (l, n+1)}_{\text{$n-l+2$ many}}
		\bigr\}
		\end{multlined}
		& \text{if $1<l\leq n+1$ for $l≔m-n$.}
		\end{cases}
	\]
	Hence, 	$|\inv \pi_m| = m$
	and $\sum_m q^{\ell(\pi_m)} = (2n+2)_q$.
	The rest of the argument is as in \cref{proof:orders-of-Coxeter-groups:Sn}.

	\item
	The Coxeter group $\mathrm{D}_n$, called \term{demihypercube group},
	is the subgroup of $\mathrm{BC}_n$ generated by $s_1,\dotsc, s_n$
	and the additional generator $\tilde{s}_0 ≔ s_0s_1s_0\colon (1,2,3,\dotsc)\mapstofrom (-2,-1,3,\dotsc)$.
	It is the subgroup of signed permutations
	that flip an even number of signs.
	A permutation $w∈ \mathrm{D}_n$ has inversions
	\[
		\inv(w) ≔
		\bigl\{ (i,j) \bigm| 1\leq i<j \leq n, w(i)>w(j) \bigr\}
		\cup
		\bigl\{ (-i,j) \mid 1\leq i \lneq j \leq n, w(-i)>w(j)\bigr\}.
	\]
	Assume by induction that $\ord_q \mathrm{D}_n = (2n-2)_q!!(n)_q$.
	Consider the permutations
	\begin{align*}
			\pi_0\colon  (1,\dotsc, n+1) &\mapsto (w(1),\dotsc, w(n) , n+1),\\*
			\shortvdotswithin{\mapsto}
			\pi_{n-1}\colon  (1,\dotsc, n+1) &\mapsto (1, n+1,\dotsc, w(n)),\\
			\pi_{n}\colon  (1,\dotsc, n+1) &\mapsto (n+1, 1,2,\dotsc, w(n)),\\
			\pi'_{n}\colon  (1,\dotsc, n+1) &\mapsto (-n-1, -1,2,\dotsc, w(n)),\\*
			\shortvdotswithin{\mapsto}
			\pi_{2n}\colon  (1,\dotsc, n+1) &\mapsto (1,\dotsc, n, -n-1).
	\end{align*}
	Their set of inversions
	\[
		\inv\pi^{(\prime)}_m =
		\begin{cases}
			\bigl\{(n-m+1, n-m+2),\dotsc, (n-m+1, n+1)\bigr\} & \text{if $0\leq m\leq n$,}\\[1em]
			\begin{multlined}[b][\widthof{$\Bigl\{(n-m+1, n-m+2),\dotsc, (n-m+1, n+1)\Bigr\}$}]
				\bigl\{
						\underbrace{(1,l),\dotsc, (l-1,l)}_{\text{$l-1$ many}},
						\underbrace{(-1,l),\dotsc, (l+1,l)}_{\text{$l-1$ many}}\\[-1.4em]
						\underbrace{(-l,l+1),\dotsc, (l, n+1)}_{\text{$n-l+1$ many}}
				\bigr\}
			\end{multlined}
			& \text{if $1<l\leq n+1$ for $l≔m-n$}
		\end{cases}
	\]
	has cardinality $|\inv\pi^{(\prime)}_m|=m$
	and therefore
	\begin{align*}
		\sum_m q^{\ell(\pi^{(\prime)}_m)}
		&= 1+\dotsb+q^n+q^n+\dotsb+q^{2n}= (1+q^n)(n+1)_q.\\
	\shortintertext{%
		As before, we see that
	}
		\ord_q \mathrm{D}_{n+1}
		&= (2n-2)_q!! \underbrace{(1+q^n)(n)_q}_{(2n)_q} (n+1)_q
		= (2n)_q!! (n+1)_q.\qedhere
	\end{align*}
	\end{prooflist}
\end{proof}
\begin{lemma}
	\label{rmk:interesting-graded-ranks}
	The following $k$-algebras have the respective graded dimensions:
	\begin{alignat*}{4}
		\rk_{q,\mathbf{Z}} \Pol_n &= \rk_{q,\mathbf{Z}} \mathbf Z[y_1,\dotsc, y_n] &&= \frac{1}{(1-q)^n},\qquad &
		\rk_{q,\mathbf{Z}} \SPol_n &= \rk_{q,\mathbf{Z}} \Pol_n^{S_n} &&= \frac{1}{(n)_q! (1-q)^n},\\
		\rk_{q,\mathbf{Z}} \NC_n &= \rk_{q,\mathbf{Z}} \mathbf Z[\partial_1,\dotsc,\partial_n] &&= \frac{1}{(n)_q!}, &
		\rk_{q,\mathbf{Z}} \NH_n &= \rk_{q,\mathbf{Z}} \Pol_n \otimes_\mathbf Z\NC_n &&= \frac{1}{(n)_q!(1-q)^n},
	\end{alignat*}
	where we set $\deg\partial_i=-1$.
\end{lemma}
\begin{proof}
\begin{description}
	\item[$\Pol_n$] Every indeterminate $y_i$ independently generates a free abelian group
	$⟨1,y_i,y_i^2,\dotsc⟩_{\mathbf{Z}}$ of graded rank $1+q+q^2+\dots = \frac{1}{1-q}$.
	Since $\Pol_n \isom \mathbf{Z}[y]^{⊗n}$ as abelian groups,
	$\Pol_n$ has graded rank $\rk_{q,\mathbf{Z}}\Pol_n = (1-q)^{-n}$.
	\item[$\SPol_n$]
	Recall that $\SPol_n\isom \mathbf{Z}\bigl[\varepsilon^{(n)}_1,\dotsc, \varepsilon^{(n)}_n\bigr]$,
	where $\varepsilon^{(n)}_m$ denotes
	the elementary symmetric polynomials in $n$ indeterminates of degree $m$.
	Each polynomial $\varepsilon^{(n)}_m$ generates a subgroup of graded rank $\frac{1}{1-q^m}$;
	$\SPol_n$ thus has graded rank
	$\rk_{q,\mathbf{Z}} \SPol_n = \prod_{m=1}^{n} \frac{1}{1-q^m}$.
	Since $(1-q^m) = (m)_q (1-q)$ (telescoping sum),
	induction shows that
	$\rk_{q,\mathbf{Z}} \SPol_n = \frac{1}{(n)_q! (1-q)^n}$.
	\item[$\NC_n$]
	By the relations of the NilCoxeter-algebra $\NC_n$,
	there is a bijection $\mathbf Z S_n\to\NC_n, s_n\mapsto\partial_n$
	such that we can write $\partial_w≔\partial_{i_1}\dotsm\partial_{i_n}$
	for any reduced expression $w=s_{i_1}\dotsm s_{i_n}$ in $S_n$.
	In particular, $\deg\partial_w = -\ell(w)$.
	By \cref{lem:poincare-polynomials-coxeter-groups} ,
	this shows that $\rk_{q,\mathbf{Z}} \NC_n = (n)_q!$\,.
	\item[$\NH_n$] This follows from the fact that $\NH_n \isom \Pol_n ⊗_\mathbf{Z} \NC_n$ as abelian groups.
	\qedhere
\end{description}
\end{proof}
\begin{remark}
	Similar statements hold for the invariants under Coxeter groups of type $\mathrm{BC}_n$ and $\mathrm{D}_n$:
	\begin{thmlist}
		\item The hyperoctahedral group $\mathrm{BC}_n$
		acts on the polynomial ring $\Pol_n$ by permuting the indeterminates
		and flipping their signs.
		It thus has invariants
		$\Pol_n^{\mathrm{BC}_n} = \mathbf{Z}[y_1^2, \dotsc, y_n^2]^{S_n}$,
		which yields
		\begin{equation*}
			\rk_{q, \mathbf{Z}}\Pol_n^{\mathrm{B}_n} = \frac{1}{(1-q^2)(1-q^4)\dotsm(1-q^{2n})}
			=\frac{1}{(2n)_q!!(1-q)^n}.
		\end{equation*}
		\item The demihypercube group $\mathrm{D}_n$
		acts on $\Pol_n$ by permuting the indeterminates
		and flipping two signs at once.
		The invariants thus are
		\begin{align*}
			\Pol_n^{\mathrm{D}_n} &= \mathbf{Z}\bigl[\varepsilon_1(y_1^2,\dotsc,y_n^2), \dotsc,\varepsilon_{n-1}(y_1^2,\dotsc,y_n^2), \varepsilon_n(y_1, \dotsc, y_n)\bigr]\\
		\shortintertext{with graded rank}
			\rk_{q, \mathbf{Z}}(\Pol_n) &=
			\frac{1}{(1-q^2)\dotsm(1-q^{2n-2})(1-q^n)}
			= \frac{1}{(2n-2)_q!! (n)_q (1-q)^n}.
		\end{align*}
	\end{thmlist}
	We notice that $\rk_{q,\mathbf{Z}} \Pol_n/\rk_{q,\mathbf{Z}} \Pol_n^W = \ord_q W$
	for $W∈\{S_n, \mathrm{BC}_n, \mathrm{D}_n\}$.
	In fact, graded ranks satisfy $\rk_{q,\mathbf{Z}} \bigl(\Pol_n/ \Pol_n^W\bigr) = \ord_q W$
	for general finite reflection groups \autocite{Chevalley:Invariants-of-finite-groups}.
	Even stronger, $\Pol_n$ is a free $\SPol_n$-module
	of graded rank $\ord_q W$ for other finite reflection groups $W$
	\autocite[thm.\ 6.2]{Demazure:Invariants}.
	We shall show a respective statement in the Clifford setup for $W=S_n$ in \cref{prop:Clifford-Polynomials-free}.
\end{remark}
\begin{corollary}
	\label{cor:graded-rank-Clifford-polynomials}
	We endow $\PolC_{I}$ with a (polynomial) grading
	such that the odd generators $\cl_i$ are of degree zero.
	\Cref{rmk:interesting-graded-ranks} then shows that
	\begin{align*}
		\rk_{q,\Cl_{I}} \PolC_{I} &= \frac{1}{(1-q)^n},&
		\rk_{q,\Cl_{I}} \NHC_{I} &= \frac{(n)_q!}{(1-q)^n}
	\end{align*}
	both as left and right $\Cl_{I}$-modules.
\end{corollary}

\subsection{\texorpdfstring{Elementary $\delC$-symmetric polynomials}{Elementary 𝔡-symmetric polynomials}}
\label{sec:elementary-d-symmetric-polynomials}
Our aim is to construct analogues for the elementary symmetric polynomials for $\SPolC_{I}$
that coincide with the ordinary elementary symmetric polynomials
if all indices are even.
For $i$, $j$ with $\lvert i-j\rvert =1$, we define
\begin{equation*}
	\gamma_{i,j} ≔
	\begin{cases}
		\cl_i & \text{if $i$ and $j = i+1$ both have odd parity,}\\
		-\cl_i & \text{if $i$ and $j = i-1$ both have odd parity,}\\
		1 & \text{otherwise}.
	\end{cases}
\end{equation*}
\begin{lemma}
	\label{lem:Elementary-Clifford-Polynomials:two-indeterminates}
	The operator $\delC_i$ on $\PolC_I$ has the kernel
	\begin{equation}
		\label{eqn:Clifford-Symmetric:Demazure-Kernel}
		\ker\delC_i = ⟨\gamma_{i,i+1} y_i + \gamma_{i+1, i} y_{i+1},\ y_i y_{i+1},\ y_j \mid j\neq i,i+1 ⟩.
	\end{equation}
\end{lemma}
	\begin{proof}
	In the purely even case, the kernel of the Demazure operator $\partial_i$
	is the subalgebra
	\begin{equation*}
		\ker\partial_i=⟨y_i + y_{i+1}, y_iy_{i+1}, y_j \mid j\neq i,i+1⟩⊆\Pol_n.
	\end{equation*}
	In the super-setting,
	we have already seen that $\ker\delC_i⊆\PolC_{I}$ is a $\Cl_{I}$-subalgebra
	in \cref{lem:super-klr-algebra:clifford-demazure-relations}.
	Assume \wlofg\ that $i=1$ and $I=\{1, 2\}$.
	The only indeterminates are thus $y_1,y_2$ and $\cl_1$,$\cl_2$.
	Let $\Lambda≔⟨y_1y_{2},\: \gamma_{1,2}y_1 + \gamma_{2,1}y_{2}⟩\subseteq\PolC_{\{1, 2\}}$ as $\Cl_{\{1, 2\}}$-superalgebra.
	We have to show that $\Lambda = \ker\delC_{\{1, 2\}}$ as subalgebras of $\PolC_{\{1, 2\}}$.
	\begin{itemize}
	\item[($⊆$)]
		We have already shown in the proof of \cref{lemma:clifford-demazure well-defined} that $y_1y_{2} = y_2y_1 ∈\ker\delC_1$.
		From
			\begin{align*}
			&\phantom{{}={}} \delC_1\big(\delC_1(y_1)y_1 - \delC_1(y_{2})y_{2}\big)\\
			&= \delC_1 \bigl((-1-\cl_1\cl_2) y_1 - (1-\cl_1\cl_2) y_2\bigr) \\
			&= (-1+\cl_1\cl_2) \delC_1(y_1) - (-1+\cl_1\cl_2) \delC_1(y_2)\\
			&= (-1+\cl_1\cl_2) (-1-\cl_1\cl_2) - (-1+\cl_1\cl_2) (1 - \cl_1\cl_2)\\
			&= 0,
		\end{align*}
		we get that $\ker \delC_1$ also contains the polynomial%
		\footnote{%
			It seems more natural to put all Clifford-generators on the right,
			because $\delC_1$ is $\Cl_{I}$-right linear.
			We shall stick to putting them on the left though
			in order to preserve compatibility with the notation used in
			\cite{KKT:Quiver-Hecke-Superalgebras}.%
		}
		\begin{align*}
			&\eqsp \big(-\tfrac 12 (\cl_1 - \cl_{2})\big)^{\pty{1}\pty{2}}
				\left( \delC_1(y_1) y_1 - \delC_1(y_{2}) y_{2} \right)\\
			&= \big(-\tfrac 12 (\cl_1 - \cl_{2})\big)^{\pty{1}\pty{2}}
				\big( (-1-\cl_1 \cl_{2})y_1 - (1-\cl_1 \cl_{2})y_1 \big)\\
			&= \begin{cases}
				\cl_1 y_1 - \cl_{2} y_{2} & \text{if both $1$, $2$ have odd parity,}\\
				y_1 + y_{2} & \text{otherwise}
			\end{cases}\\
			&= \gamma_{1,2}y_1 + \gamma_{2,1}y_{2}.
		\end{align*}
	\item[($⊇$)]
		The superalgebra $\PolC_{\{1, 2\}}$ contains an element $\alpha_1 ≔ y_1-y_{2}$,
		which satisfies $\delC_1(\alpha_1)=-2$.
		Recall that $\PolC_{\{1, 2\}}$ is a free left and right $\Cl_{\{1, 2\}}$-module
		of graded rank $\rk_{q,\Cl_{\{1, 2\}}} \PolC_{\{1, 2\}} = \frac{1}{(1-q)^2}$.
		Its subalgebra $\Lambda$
		is also $\Cl_{\{1, 2\}}$-free of graded rank $\rk_{q,\Cl_{\{1, 2\}}}\Lambda=\frac{1}{1-q}\cdot\frac{1}{1-q^2}$.
		The $\Cl_{\{1, 2\}}$-submodule $\Lambda\alpha_1 ⊆ \PolC_{\{1, 2\}}$ does not intersect $\ker\delC_1$
		because for any $\lambda∈\Lambda$ non-zero, we have $\delC_1(\lambda\alpha_1)=-2s_1(\lambda) \neq 0$.
		It has graded rank $\rk_{q,\Cl_{\{1, 2\}}} \Lambda\alpha_1 = q\rk_{q,\Cl_{\{1, 2\}}}\Lambda$.
		Since
		\[
			\rk_{q,\Cl_{\{1, 2\}}} \Lambda + \rk_{q,\Cl_{\{1, 2\}}} \Lambda\alpha_1
			= \frac{1+q}{(1+q)(1-q^2)} = \frac{1}{(1-q)^2} = \rk_{q,\fCl_2} \PolC_{\{1, 2\}},
		\]
		we obtain that $\PolC_{\{1, 2\}}=\Lambda⊕\Lambda\alpha_s$ as $\Cl_{\{1, 2\}}$-module
		and in particular that the inclusion $\Lambda⊆\ker\delC_1$ in fact is an equality.
		\qedhere
	\end{itemize}
\end{proof}
Let $\gamma_{n,n+1} ≔ 1$ if $n+1$ exceeds the number of indeterminates.
Since
\begin{align}
	 y_i y_{i+1}
	 &= \gamma_{i+1,i+2} \, \gamma_{i,i+1}(\gamma_{i, i+1} \, \gamma_{i+1,i+2}\gamma_{i+1,i+2}\, y_i y_{i+1})\\
	 &= (-1)^{\pty{i}\pty{i+1}+\pty{i+1}\pty{i+2}}
		 (y_i y_{i+1} \, \gamma_{i+1,i+2} \, \gamma_{i+1,i+2}) \, \gamma_{i+1,i+2} \, \gamma_{i,i+1}
\end{align}
we may replace the second generator by $\gamma_{i,i+1} y_1\,\gamma_{i+1,i+2} y_{i+1}$
without changing the $\Cl_{I}$-span.
\begin{definition}
	\label{def:Elementary-Clifford-Polynomials:two-indeterminates}
	We denote the generators of $\ker\delC_i$  by
	\begin{equation*}
		\epsC^{(i-1,i+1)}_1≔\gamma_{i,i+1} y_i + \gamma_{i+1, i} y_{i+1};\qquad
		\epsC^{(i-1,i+1)}_2≔\gamma_{i,i+1} y_i\,\gamma_{i+1,i+2} y_{i+1}.
	\end{equation*}
	If $i=1$, we simply write $\epsC^{(2)}_m ≔ \epsC^{(i-1,i+1)}_m$.
	We can thus write $\ker\delC_1 = ⟨\epsC^{(2)}_1, \epsC^{(2)}_2⟩⊆\PolC_2$ as $\Cl_{\{1, 2\}}$-superalgebra.
\end{definition}
Our goal is to show that there are polynomials $\epsC^{(n)}_m$, $1\leq m\leq n$
of polynomial degree $m$
that generate $\SPolC_{I}$ of $\Cl_{I}$-superalgebras.
These are meant to serve as a Clifford-replacement
for the ordinary elementary symmetric polynomials.

Our first task is to find a recursive formula for polynomials lying in $\SPolC_{I}$.
Proving that they are indeed generators requires some more work.
We shall prove this in \cref{prop:Clifford-Polynomials-free}.

Before coming to the $n$-fold intersection $\SPolC_{I}=\bigcap_{k=1}^{n-1}\ker\delC_k$,
we consider the intersection $\ker\delC_1 \cap \ker\delC_2 ⊆ \PolC_{\{1,2,3\}}$
of just the first two kernels.
We can multiply the first generator of $\ker\delC_2$
with the unit $\gamma_{2,1}\gamma_{2,3}$ from the left
without changing its span.
Thus
\makeatletter
\preto\multline{\@fleqnfalse}
\makeatother

\begin{multline}
	\ker\delC_1 \cap \ker\delC_{2}
	=
	\Bigl\langle
		\overunderbraces{&\br{2}{\epsC^{(0,2)}_1}}%
			{&\gamma_{1,2} y_1 + &\gamma_{2,1} y_{2}& + \gamma_{2,1}\gamma_{2,3}\,\gamma_{3,2}\ y_3}%
			{&& \br{2}{\gamma_{2,1}\gamma_{2,3}\, \epsC^{(1,3)}_1}},
	\\[.5em]
	%
		\overunderbraces{&\br{2}{\gamma_{1,2}\,y_1\, \gamma_{2,1}\gamma_{2,3}\,\epsC^{(1,3)}_1}}%
			{
				& \gamma_{1,2}\,y_1 \ \gamma_{2,1}\, y_2
				+ & \gamma_{1,2}\,y_1 \ \gamma_{2,1}\gamma_{2,3}\,\gamma_{3,2}\, y_3 &
				+ \gamma_{2,1}\,y_2 \ \gamma_{2,1}\gamma_{2,3}\,\gamma_{3,2}\, y_3
			}
			{&&\br{2}{\epsC^{(0,2)}_1 \gamma_{2,1}\gamma_{2,3}\,\gamma_{3,2}\, y_3}},
	\\[.5em]
	%
	\overunderbraces{&\br{1}{\epsC^{(0,2)}_2}}%
		{&\gamma_{1,2}\,y_1\ \gamma_{2,3}\, y_2\ & \gamma_{3,4}\, y_3}%
		{}
	\Bigr\rangle
\end{multline}
as a $\Cl_{I}$-superalgebra.
This depiction serves as a template for the following:
\begin{definition}
	\label{lem:Elementary-Clifford-Polynomials:Recursively}
	The \term{elementary $\delC$-symmetric polynomials $\epsC^{(n)}_m$ of degree $m$ in $n$ indeterminates} are the polynomials in $\PolC_{I}$ defined recursively as follows:
	\begin{align}
		\label{eqn:Elementary-Clifford-Polynomials:Recursively 1}
		\epsC^{(1)}_{\mathstrut 1} &=  \gamma_{1,2} y_1\\
			& \shortvdotswithin{=}
			\epsC^{(n)}_1 &= \epsC^{(n-1)}_1
				+ (\gamma_{2,1}\gamma_{2,3}\ \gamma_{3,2}\gamma_{3,4}\dotsm \gamma_{n-1,n-2}\gamma_{n-1,n})
				\gamma_{n,n-1}\ y_n,\\
			\epsC^{(n)}_{\mathstrut 2} &=
				 \epsC^{(n-1)}_1
				 + \epsC^{(n)}_1 (\gamma_{3,2}\gamma_{3,4}\dotsm \gamma_{n-1,n-2}\gamma_{n-1,n})
				\gamma_{n,n-1}\ y_n,\\
			& \shortvdotswithin{=}
			\epsC^{(n)}_m &=
				 \epsC^{(n-1)}_{m}
				 + \epsC^{(n-1)}_{m-1} (\gamma_{m+1,m}\gamma_{m+1,m+2}\dotsm \gamma_{n-1,n-2}\gamma_{n-1,n})
				\gamma_{n,n-1}\ y_n,\\
			& \shortvdotswithin{=}
			\epsC^{(n)}_{n} &= \epsC^{(n-1)}_n\,\gamma_{n+1,n+2}\ y_n.
			\label{eqn:Elementary-Clifford-Polynomials:Recursively n}
	\end{align}
	We define the subalgebra $\Lambda_I≔⟨\epsC^{(n)}_1, \ldots, \epsC^{(n)}_n⟩ \subseteq \PolC_I$.
\end{definition}
\begin{lemma}
	$\Lambda_I\subseteq\SPolC_{I}$.
\end{lemma}
\newcommand{\epsO}{\mathit{o\varepsilon}}
\begin{table}[p]
	\caption{%
		Elementary $\delC$-symmetric polynomials $\epsC^{(n)}_m$
		as defined in \cref{lem:Elementary-Clifford-Polynomials:Recursively},
		spelt out explicitly for $m\leq n\leq 4$.
	}
	\label{tab:elementary-d-symmetric-polynomials}
		\footnotesize
		\centering
		$\begin{array}{@{}l@{}llll@{}}
			\toprule
			m\backslash n & \hfil 1 & \hfil 2 & \hfil 3 & \hfil 4\\
			\midrule
			1
			& \gamma_{1,2} y_1
			& \begin{array}{@{}>{{}}l@{}}
				\phantom{{}+{}}\gamma_{1,2}\, y_1 \cr
				+ \gamma_{2,1}\, y_2
			\end{array}
			& \begin{array}{@{}>{{}}l@{}}
				\phantom{{}+{}} \gamma_{1,2}\, y_1\cr
				+ \gamma_{2,1}\, y_2\cr
				+ \gamma_{2,1}\gamma_{2,3}\gamma_{3,2}\ y_3
			\end{array}
			& \begin{array}{@{}>{{}}l@{}}
				\phantom{{}+{}}\gamma_{1,2}\, y_1\cr
				+ \gamma_{2,1}\, y_2\cr
				+ \gamma_{2,1}\gamma_{2,3}\gamma_{3,2}\ y_3\cr
				+ \gamma_{2,1}\gamma_{2,3}\, \gamma_{3,2}\gamma_{3,4}\gamma_{4,3}\ y_4
			\end{array}
			\\\addlinespace
			2
			&& \phantom{{}+{}} \gamma_{1,2},y_1 \ \gamma_{2,3}\, y_2
			& \begin{array}{@{}>{{}}l@{}}
				\phantom{{}+{}}\gamma_{1,2}\, y_1\ \gamma_{2,3}\, y_2\cr
				+ \gamma_{1,2}\,y_1\ \gamma_{3,2}\, y_3\cr
				+ \gamma_{2,1}\,y_2\ \gamma_{3,2}\, y_3
			\end{array}
			& \begin{array}{@{}>{{}}l@{}}
				\phantom{{}+{}}\gamma_{1,2}\, y_1\ \gamma_{2,3}\, y_2\cr
				+ \gamma_{1,2}\,y_1\ \gamma_{3,2}\, y_3\cr
				+ \gamma_{2,1}\,y_2\ \gamma_{3,2}\, y_3\cr
				+ \gamma_{1,2}\,y_1\ \gamma_{3,2}\gamma_{3,4}\gamma_{4,3}\, y_4\cr
				+ \gamma_{2,1}\,y_2 \ \gamma_{3,2}\gamma_{3,4}\gamma_{4,3}\ y_4\cr
				+ \gamma_{2,1}\gamma_{2,3}\gamma_{3,2}\,y_3 \ \gamma_{3,2}\gamma_{3,4}\,\gamma_{4,3}\ y_4
			\end{array}
			\\\addlinespace
			3 &&& \phantom{{}+{}}\, \gamma_{1,2}\,y_1\,\gamma_{2,3}\,y_3\,\gamma_{3,4}\,y_3
			& \begin{array}{@{}>{{}}l@{}}
				\phantom{{}+{}}
				  \gamma_{1,2}\,y_1\ \gamma_{2,3}\, y_2\ \gamma_{3,4}\,y_3\cr
				+ \gamma_{1,2}\,y_1\ \gamma_{2,3}\, y_2\ \gamma_{4,3}\,y_4\cr
				+ \gamma_{1,2}\,y_1\ \gamma_{3,2}\, y_3\ \gamma_{4,3}\,y_4\cr
				+ \gamma_{2,1}\,y_2\ \gamma_{3,2}\, y_3\ \gamma_{4,3}\,y_4
			\end{array}
			\\\addlinespace
			4 &&&& \phantom{{}+{}}\gamma_{1,2}\,y_1\ \gamma_{2,3}\,y_2\ \gamma_{3,4}\,y_3\ \gamma_{4,5}\,y_4
			\\
			\bottomrule
		\end{array}$
\end{table}
\begin{table}[p]
	\caption{%
		Ordinary elementary symmetric polynomials $\varepsilon^{(n)}_m$ for $m\leq n\leq 4$.
		They are a specialisation of the ones listed in \cref{tab:elementary-d-symmetric-polynomials}
		at $\gamma_{i,i\pm 1}=1$ for all $i$, obtained when all $i$ have \emph{even} parity.
	}
	\label{tab:elementary-symmetric-polynomials}
	\centering
	\footnotesize
	$\begin{array}{@{}l@{}llll@{}}
			\toprule
			m\backslash n
			& \mathmakebox[\widthof{$\gamma_{1,2} y_1$}]{1}
			& \mathmakebox[\widthof{$\phantom{{}+{}} \gamma_{1,2},y_1 \ \gamma_{2,3}\, y_2$}]{2}
			& \mathmakebox[\widthof{$\phantom{{}+{}}\, \gamma_{1,2}\,y_1\,\gamma_{2,3}\,y_3\,\gamma_{3,4}\,y_3$}]{3}
			& \mathmakebox[\widthof{$\gamma_{2,1}\gamma_{2,3}\gamma_{3,2}\,y_3 \ \gamma_{3,2}\gamma_{3,4}\,\gamma_{4,3}\ y_4$}]{4}
			\\
			\midrule
			1
			& y_1
			& y_1 + y_2
			& y_1 + y_2 + y_3
			& \phantom{{}+{}} y_1 + y_2 + y_3 + y_4
			\\\addlinespace
			2
			&& y_1y_2
			& y_1y_2 + y_1y_3 + y_2y_3
			& \begin{array}{@{}>{{}}r<{{}}@{}l}
				& y_1y_2 + y_1y_3\\  + & y_1y_4 + y_2y_3\\ + & y_2y_4 + y_3y_4
			\end{array}
			\\\addlinespace
			3
			&&& y_1 y_2y_3
			& \begin{array}{@{}>{{}}r<{{}}@{}l}
				& y_1y_2y_3 + y_1y_2y_4\\ + & y_1y_3y_4 + y_2y_3y_4
			\end{array}
			\\\addlinespace
			4 &&&& \phantom{{}+{}} y_1y_2y_3y_4
			\\
			\bottomrule
	\end{array}$
\end{table}
\begin{table}[p]
	\caption{%
		Odd elementary symmetric polynomials $\epsO^{(n)}_m$ for $m\leq n\leq 4$.
		They are a specialisation of the ones listed in \cref{tab:elementary-d-symmetric-polynomials}
		at $\gamma_{i,i\pm 1}=\pm 1$ for all $i$, obtained when all $i$ have \emph{odd} parity.
	}
	\label{tab:odd-elementary-symmetric-polynomials}
	\centering
	\footnotesize
	$\begin{array}{@{}l@{}llll@{}}
			\toprule
			m\backslash n
			& \mathmakebox[\widthof{$\gamma_{1,2} y_1$}]{1}
			& \mathmakebox[\widthof{$\phantom{{}+{}} \gamma_{1,2},y_1 \ \gamma_{2,3}\, y_2$}]{2}
			& \mathmakebox[\widthof{$\phantom{{}+{}}\, \gamma_{1,2}\,y_1\,\gamma_{2,3}\,y_3\,\gamma_{3,4}\,y_3$}]{3}
			& \mathmakebox[\widthof{$\gamma_{2,1}\gamma_{2,3}\gamma_{3,2}\,y_3 \ \gamma_{3,2}\gamma_{3,4}\,\gamma_{4,3}\ y_4$}]{4}
			\\
			\midrule
			1
			& \phantom{{}+{}} x_1
			& \phantom{{}+{}} x_1 - x_2
			& \phantom{{}+{}} x_1-x_2+x_3
			& \phantom{{}+{}} x_1-x_2+x_3-x_4
			\\\addlinespace
			2
			&& \phantom{{}+{}} x_1 x_2
			& \begin{array}{@{}>{{}}l@{}}
				\phantom{{}+{}}x_1 x_2
				- x_1 x_3\cr
				+ x_2 x_3
			\end{array}
			& \begin{array}{@{}>{{}}c@{}}
				\phantom{{}+{}}x_1 x_2
				- x_1 x_3
				+ x_2 x_3\cr
				+ x_1 x_4
				- x_2 x_4
				- x_3 x_4
			\end{array}
			\\\addlinespace
			3 &&& \phantom{{}+{}} x_1 x_3 x_3
			& \begin{array}{@{}>{{}}l@{}}
				\phantom{{}+{}} x_1 x_3 x_3
				- x_1 x_2 x_3\cr
				+ x_1 x_3 x_4
				-  x_2  x_3 x_3
			\end{array}
			\\\addlinespace
			4 &&&& \phantom{{}+{}}x_1 x_3 x_3 x_4\\
			\bottomrule
	\end{array}$
\end{table}

\begin{example}
	\label{example:Elementary-Clifford-Polynomials}
	Before proving the lemma let us make the definition of the $\epsC$'s explicit
	for small numbers of indeterminates.
	\begin{thmlist}%
		\item For $n=1,\dotsc,4$ variables,
		the elementary $\delC$-symmetric polynomials are listed explicitly
		in \cref{tab:elementary-d-symmetric-polynomials}.

	\item
		If all indices are even, then all $\gamma_{i,i\pm1}=1$.
		In this case, the $\delC$-elementary symmetric polynomials
		from \cref{tab:elementary-d-symmetric-polynomials}
		specialise to those from \cref{tab:elementary-symmetric-polynomials} at $\gamma_{i,i\pm1}=1$,
		which we recognise at the ordinary elementary symmetric polynomials $\varepsilon^{(n)}_m$.
		In this case, the induction formula from \cref{lem:Elementary-Clifford-Polynomials:Recursively}
		indeed gives
		\begin{equation}
		\begin{split}
		\varepsilon^{(1)}_1 &=  x_1\\
		\varepsilon^{(n)}_m &=
		\varepsilon^{(n-1)}_{m}
		+ \varepsilon^{(n-1)}_{m-1} x_n
		\end{split}
		\end{equation}
		by setting each of the $\gamma$'s in \crefrange{eqn:Elementary-Clifford-Polynomials:Recursively 1}{eqn:Elementary-Clifford-Polynomials:Recursively n} to $1$.
		This is a well-known recursive formula for the ordinary elementary symmetric polynomials
		\begin{equation}
		\varepsilon^{(n)}_m =\sum_{\mathclap{1\leq k_{1}<\dotsb <k_{m}\leq n}} x_{k_{1}}\dotsm x_{k_m}.
		\end{equation}

	\item\label{example:Elementary-Clifford-Polynomials:odd}
		If all indices are odd, then $\gamma_{i,i\pm 1}=\pm \cl_i$.
		We set $x_i≔\gamma_{i,i+1}y_i = \cl_i y_i$
		(cf.\ \cref{lemma:quiver-Hecke-super-embedding}).
		The polynomials from \cref{tab:elementary-d-symmetric-polynomials}
		then specialise at $\gamma_{i,i\pm 1}=\pm \cl_i$
		to the ones listed in \cref{tab:elementary-symmetric-polynomials}.
		We refer to these as \term{odd elementary symmetric} polynomials
		and denote them by $\epsO^{(n)}_m$.
		The recursion formula reduces to
		\begin{equation}
			\begin{split}
			\epsO^{(1)}_1 &=  x_1\\
			\epsO^{(n)}_m &=
			\epsO^{(n-1)}_{m}
			+ (-1)^{n-m}\epsO^{(n-1)}_{m-1}\, x_n\\
			\epsO^{(n)}_n &=
			\epsO^{(n-1)}_{m-1}\, x_n
			\end{split}
		\end{equation}
		Consider also the odd elementary symmetric polynomials
		\begin{equation}
			\prescript{}{\mathrm{EKL}}{o\varepsilon}^{(n)}_m ≔ \sum_{\mathclap{1\leq k_{1}<\dotsb <k_{m}\leq n}} (-1)^{k_m-1} x_{k_{1}}\dotsm (-1)^{k_1-1} x_{k_m}
		\end{equation}
		as defined in \autocite[(2.21–2.23)]{EKL:odd-NilHecke},
		which satisfy the recursion formula
		\begin{equation*}
			\prescript{}{\mathrm{EKL}}{o\varepsilon}^{(n)}_m =
				\prescript{}{\mathrm{EKL}}{o\varepsilon}^{(n-1)}_m
				+ (-1)^{n-1}\prescript{}{\mathrm{EKL}}{o\varepsilon}^{(n-1)}_m x_n.
		\end{equation*}
		Our odd elementary symmetric polynomials
		indeed coincide with those from \autocite{EKL:odd-NilHecke}
		up to an overall sign $(-1)^{\frac{m(m-1)}{2}}$:
		\begin{align}
		\label{eqn:Odd-Elementary-Polynomial:EKL-Sign}
		\epsO^{(n)}_m &=
		\epsO^{(n-1)}_{m}
		+ (-1)^{n-m} \epsO^{(n-1)}_{m-1}\, x_n
		\notag\\
		&= 	(-1)^{\frac{m(m-1)}{2}} \cdot \prescript{}{\mathrm{EKL}}{o\varepsilon}^{(n-1)}_m
		+ (-1)^{\big(\frac{(m-1)(m-2)}{2}\big) + (n-m)}
		\cdot \prescript{}{\mathrm{EKL}}{o\varepsilon}^{(n-1)}_m\, x_n
		\notag\\
		&= (-1)^{\frac{m(m-1)}{2}}  \cdot
		\prescript{}{\mathrm{EKL}}{o\varepsilon}^{(n-1)}_m
		+ (-1)^{n-1}x_n \cdot \prescript{}{\mathrm{EKL}}{o\varepsilon}{(n-1)}_m
		\notag\\
		&=  (-1)^{\frac{m(m-1)}{2}}
		\cdot \prescript{}{\mathrm{EKL}}{o\varepsilon}^{(n)}_m.
		\end{align}
		According to our sign convention,
		the first monomial $x_1\dotsm x_m$ of the odd polynomials
		always has $+1$ as coefficient.
		Odd symmetric functions have already been treated extensively in
		\autocites{EKL:odd-NilHecke}{EK:odd-NilHecke-Hopf}{EQ:odd-NilHecke-dg}.
	\end{thmlist}
\end{example}
\begin{proof}[Proof of \cref{lem:Elementary-Clifford-Polynomials:Recursively}]
	The statement is proven by induction on $n$ and $m$.
	It is convenient to set $\epsC^{(n)}_0 = 1$ for all $n\geq 0$.
	The claim then trivially holds for $n=1$ and $m=0,1$.

	\proofsection{Induction step on $n$ for $m=1$}
	Assume that $\epsC^{(k)}_1 ∈ \delC_{1}\cap\dotsb\cap\delC_{n-1}$ for all $k\leq n$.
	We want to show that $\ker\delC_{1}\cap\dotsb\cap\delC_{n}$ contains the polynomial
	\begin{align*}
		\epsC^{(n+1)}_1
		&= \epsC^{(n)}_1 + (\gamma_{2,1}\gamma_{2,3} \dotsm \gamma_{n,n-1}\gamma_{n-1,n}) \gamma_{n+1,n}\, y_{n+1}.\\
	\intertext{%
		The first summand $\epsC^{(n)}_1$ is contained in $\ker\delC_1 \cap \dotsb \cap\ker\delC_{n-1}$
		by the induction hypothesis.
		The second summand contains no $y_{k}$ for $k\leq n$
		and thus also lies in $\ker\delC_1 \cap \dotsb \cap\ker\delC_{n-1}$
		(cf.\ the definition of $\delC$ in \cref{eq:Definition-NilHecke-Clifford-algebra}).
		In order to show that $\epsC^{(n+1)}_1 ∈ \ker \delC_{n}$,
		we expand the first summand $\epsC^{(n)}_1$ and obtain
	}
		\epsC^{(n+1)}_1 &=
			\epsC^{(n-1)}_1
			+ (\gamma_{2,1} \dotsm \gamma_{n-1,n-2}\gamma_{n-1,n}) \gamma_{n,n-1}\, y_{n}+{}
			(\gamma_{2,1} \dotsm \gamma_{n,n-1}\gamma_{n,n+1}) \gamma_{n+1,n}\, y_{n+1}.
		\\
	\shortintertext{Regrouping the terms gives}
		&= \epsC^{(n-1)}_1  +
		(
			\gamma_{2,1} \dotsm \gamma_{n,n-1}\gamma_{n-1,n}\ \gamma_{n,n-1}
		)
		\underbrace{
			( y_{n} + \gamma_{n,n+1}\gamma_{n+1,n} \ y_{n+1})
		}_{
			 \gamma_{n,n+1}\epsC^{(n-1,n+1)}_{1}
		}.
	\end{align*}
	The first summand $\epsC^{(n-1)}_1 $ is clearly contained in $\ker \delC_{n}$
	because it contains no $y_{n}$, $y_{n+1}$.
	The first expression in parentheses lies in $\Cl_{I}^*$.
	The second summand lies in the $\Cl_{I}$-superalgebra $\bigl\langle\epsC^{(n-1,n+1)}_{1}\bigr\rangle ⊆ \ker \delC_{n}$
	(see \cref{lem:Elementary-Clifford-Polynomials:two-indeterminates,def:Elementary-Clifford-Polynomials:two-indeterminates}).
	Therefore, ${\epsC^{(n+1)}_1 ∈ \ker\delC_{1}\cap\dotsb\cap\delC_{n}}$.

	\proofsection{Induction step for $m>1$}
	Assume $\epsC^{(k)}_{m} ∈ \ker\delC_1 \cap\dotsb\cap \ker\delC_{n-1}$ for all $k\leq n$ and $m\leq k$.
	We want to show that $\ker\delC_1 \cap\dotsb\cap \ker\delC_n$ contains
	\begin{equation}
		\label{eq:proof:Elementary-Clifford-Polynomials:resolve-e-n+1}
		\epsC^{(n+1)}_m = \epsC^{(n)}_m + \epsC^{(n)}_{m-1}
		(\gamma_{m+1,m}\gamma_{m+1,m+2} \dotsm \gamma_{n,n-1}\gamma_{n-1,n})\gamma_{n+1,n}\,y_{n+1}.
	\end{equation}
	\settowidth{\algnRef}{$\epsC^{(n+1)}_m\eqsp$}
	\begin{description}
	\item[$\epsC^{(n+1)}_m ∈ \ker\delC_{k}$ for $k\leq n-1$]
		By the induction hypothesis, we have $\epsC^{(n)}_m ∈ \ker\delC_{k}$.
		The second summand satisfies
		\begin{equation*}
			\begin{split}
			\mathmakebox[\algnRef][l]{
			\delC_k \bigl(
				\epsC^{(n)}_{m-1}
				\underbrace{\gamma_{m+1,m}\gamma_{m+1,m+2} \dotsm \gamma_{n,n-1}\gamma_{n-1,n}}_{*}
				\gamma_{n+1,n}y_{n+1}\bigr)
			} & \\
			&= s_i \bigl(
					\epsC^{(n)}_{m-1}\cdot
					\gamma_{m+1,m}\gamma_{m+1,m+2} \dotsm \gamma_{n,n-1}\gamma_{n-1,n}\gamma_{n+1,n}
			\bigr)\,
			\delC_k (y_{n+1})\\
			&= 0
			\end{split}
		\end{equation*}
		by the induction hypothesis on $\epsC^{(n)}_{m-1}$,
		where the expression $*$ lies in $\Cl_{I}^*$.

	\item[$\epsC^{(n+1)}_m ∈ \ker\delC_{n}$]
		For $\delC_{n}$, we expand $\epsC^{(n)}_m$ and $\epsC^{(n)}_{m-1}$
		in \cref{eq:proof:Elementary-Clifford-Polynomials:resolve-e-n+1}
		by the recursion formula and obtain
		\begin{align}
			\notag
			\epsC^{(n+1)}_m
			&=
			\setlength{\multlinegap}{0pt}
			\begin{multlined}[t][\linewidth-2\mathindent-\widthof{$\epsC^{(n+1)}_m = {}$}]
				\Bigl[
					\overbrace{
						\epsC^{(n-1)}_m
						+ \epsC^{(n-1)}_{m-1} \bigl(\gamma_{m+1,m}\dotsm \gamma_{n-1,n-2}\gamma_{n-1,n} \bigr) \gamma_{n,n-1}\,y_n
					}^{
						\epsC^{(n)}_m
					}
				\Bigr] + {}\\[1ex]
				{}+
				\Big[
				\underbrace{
					\epsC^{(n-1)}_{m-1}
					+ \epsC^{(n-1)}_{m-2}(\gamma_{m,m-1} \dotsm \gamma_{n-1,n})\gamma_{n,n-1}\, y_{n})
				}_{\epsC^{(n)}_{m-1}}
				\Big]
				\cdot(\gamma_{m+1,m} \dotsm \gamma_{n,n-1}\gamma_{n,n+1})\gamma_{n+1,n}\, y_{n+1}.
			\end{multlined}
		\intertext{We regroup both terms containing $\epsC^{(n-1)}_{m-1}$ and get}
			\label{eq:recursion of e-m-n+1 final}
			\epsC^{(n+1)}_m &=
			\begin{multlined}[t][\linewidth-2\mathindent-2cm]
				\epsC^{(n-1)}_m +
				\epsC^{(n-1)}_{m-2} \,
				\overbrace{
					(\gamma_{m,m-1} \dotsm \gamma_{n-1,n})\gamma_{n,n-1}\, y_n \,
					(\gamma_{m+1,m} \dotsm \gamma_{n,n+1})\gamma_{n+1,n-1}\, y_{n+1} \,
				}^{
					\pm (\gamma_{m,m-1} \dotsm \gamma_{n,n-1})(\gamma_{m+1,m}\dotsm \gamma_{n+1,n})
					\gamma_{n+1,n+2}\gamma_{n,n+1}\,\epsC^{(n-1,n+1)}_2
				}
				 + {}\\[2ex]
				{} +
				\epsC^{(n-1)}_{m-1}
				\underbrace{\bigl(\dotsm \gamma_{n-1,n-2}\gamma_{n-1,n}\, \gamma_{n,n-1} \bigr)}%
				_{*}
					\underbrace{\bigl(\gamma_{n,n-1} y_n +\gamma_{n+1,n}y_{n+1}	\bigr)}_{\epsC^{(n-1,n+1)}_1},
			\end{multlined}
		\end{align}
		where the expression $*$ lies in $\Cl_{I}^*$.
		The operator $\delC_n$ maps the terms $\epsC^{(n-1)}_m$, $\epsC^{(n-1)}_{m-2}$ and $\epsC^{(n-1)}_{m-1}$ to zero
		because they do not contain $y_n$ and $y_{n+1}$.
		The second and the last summand of \eqref{eq:recursion of e-m-n+1 final} are mapped to zero by $\delC_n$
		because they are contained in the $\Cl_{I}$-bimodule
		$\epsC^{(n-1)}_{m-1}\cdot⟨\epsC^{(n-1,n+1)}_2⟩ ⊆ \ker\delC_n$.
	\end{description}
	Therefore, $\epsC^{(n+1)}_m ∈ \ker\delC_n$.
	We have already seen that $\epsC^{(n+1)}_m ∈ \ker\delC_1\cap\dotsb\cap\ker\delC_{n-1}$.
	This proves the assertion.
\end{proof}
\begin{remark}
	\label{rmk:graded-rank-d-symmetric-polynomials}
	The $\Cl_{I}$-superalgebra $\Lambda_I$
	has a basis $\bigl\{\bigl(\epsC^{(n)}_1\bigr)^{\alpha_i}, \ldots, \bigl(\epsC^{(n)}_n\bigr)^{\alpha_n}\bigr\}$,
	$\alpha_i ∈ \mathbf{N}$,
	as left and as right $\Cl_{I}$-module
	(cf.\ the proof of \cref{lem:PolCn-is-free}).
	With the same argumentation as in \cref{rmk:interesting-graded-ranks},
	we see that its graded $\Cl_{I}$-rank is $\frac{1}{(n)_q! (1-q)^n}$
	both as left and right $\Cl_{I}$-module.
\end{remark}

\subsection{\texorpdfstring{Schubert Polynomials; freeness of $\PolC_n$}{Schubert Polynomials; freeness of Polℭn}}
\newcommand{\sC}{\mathfrak{s}}
\label{sec:Clifford-complete-Schubert-polynomials}
\begin{definition}
	The \term{$\delC$-Schubert polynomial} $\sC_w ∈ \PolC_{I}$
	associated to a reduced word $w∈S_n$ is
	\begin{equation}
		\sC_w ≔ \delC_{w^{-1}w_0} (y_1^{n-1}\dotsm y_{n-1}^1).
	\end{equation}
	This makes sense, since according to \cref{lem:super-klr-algebra:clifford-demazure-relations},
	the $\delC_i$ satisfy the braid relations from \eqref{eq:NilCoxeter--Clifford-relations}.
	This is the definition of the ordinary Schubert polynomials
	\autocite[§2.3]{Manivel}
	with $\delC$ instead of $∂$.
	Hence, in the purely even case, the $\delC$-Schubert polynomials coincide with the ordinary ones.
\end{definition}
We get the following lemma, which generalises from the statement \cite[173]{Fulton:Young-Tableaux} about ordinary Schubert polynomials, and from the statement \cite[(2.42–2.43)]{EKL:odd-NilHecke} about odd ones:
\begin{lemma}
	The $\delC$-Schubert polynomials have the following properties:
	\begin{thmlist}
	\item\label{lem:divided-difference-on-schubert-polynomials}
		If $wv^{-1}$ is a reduced expression, then $\delC_v\sC_w = \sC_{wv^{-1}}$.
	\item\label{lem:relations-schubert-polynomials}
		For any tuple $\bm{\alpha}\in\mathbf{N}^n$,
		we have $y^{\bm{\alpha}}\delC_v\sC_w$ if $v=w^{-1}$,
		and $y^{\bm{\alpha}}\delC_v\sC_w = 0$ if $v \neq w^{-1}$ and $\ell(v) \geq \ell(w)$.
		In particular, $\sC_e$ is non-zero.
	\end{thmlist}
\end{lemma}
\begin{proof}
	The first statement is clear from the definition of $\sC_w$.
	For the second one, we argue as follows:
	\begin{itemize}
	\item If $\ell(v)>\ell(w)$:
		Every $\delC_i$ reduces the polynomial degree by one.
		Recall that the longest element $w_0$ of the symmetric group has
		$\ell(w_0)=\frac{n(n-1)}{2}=\deg(y^{n-1}_1\dotsm y^1_{n-1})$.
		Thus, $\sC_w$ has polynomial degree $\ell(w)$.
		Since $\delC_i$ acts on constants by zero, the assertion follows.
	\item If $\ell(v)=\ell(w)$ and $v\neq w^{-1}$:
		In this case, $v(w^{-1}w_0)$ is not a reduced expression,
		which implies that $\delC_v\delC_{w^{-1}w_0}=0$.
	\item If $v=w^{-1}$:
		Take the reduced expression $w_0 = s_{n-1} \dotsm s_2 s_1 s_{n-2}\dotsm s_3 s_2 s_1 s_2 s_1$.
		We have
		\begin{align*}
			&\eqsp \delC_{w_0} (y^{n-1}_1\dotsm y^1_{n-1})\\
			&= \delC_{w_0s_1} \delC_1( y_1 \, y_2 y_1 \, y_3 y_2 y_1 \, y_4 y_3 y_2 y_1\dotsm y_{n-1} y_{n-2}\dotsm y_1)\\
			&= 
					\delC_{w_0s_1} \bigl[\delC_1(y_1) \, (y_1 \, y_2 y_1 y_3 y_2 y_1 \dotsm) +
					 y_2 \delC_{y_1 \, y_2 y_1 \, y_3 y_2 y_1 \dotsm}
					\bigr].
		\intertext{%
		The second summand vanishes because in the argument of $\delC_1$,
		the indeterminates $y_1$, $y_2$ only occur in products $y_1y_2$,
		which lie in $\ker\delC_1$ (see \cref{lem:Elementary-Clifford-Polynomials:two-indeterminates}).
		We continue with the next operators.
		It turns out that in every step, we can apply $\delC_i$ to precisely one factor $y_i$,
		since the remaining $y_i$, $y_{i+1}$'s occur pairwise. We obtain
		}
			&= \delC_{w_0s_1s_2} s_2(\delC_1 y_1) \bigl[ (\delC_2y_2) \,(y_1 \, y_3 y_2 y_1 \dotsm) + \overbrace{\delC_2(y_1 \, y_3 y_2 y_1 \dotsm)}^0\bigr]\\
			&= \delC_{w_0s_1s_2s_1} s_1s_2(\delC_1 y_1) \, s_1(\delC_2y_2) \, \bigl[\delC_1(y_1)\, y_3 y_2 y_1 \dotsm\bigr]\\
			\shortvdotswithin{=}
			&= (w_0s_1)(\delC_1 y_1) (w_0s_1s_2)(\underbrace{\delC_2 y_2}_{∈\Cl_{I}^*}) (w_0s_1s_2s_1)(\delC_1 y_1)\dotsm s_1 (\delC_2 y_2) \cdot \delC_1(y_1),
		\end{align*}
		where $\delC_1 y_1, \delC_2 y_2 \in \Cl_{I}^*$.
		Since the permutation action of $S_n$ preserves $\Cl_{I}^*$,
		we obtain that $\sC_e ∈\Cl_{I}^*$.
		\qedhere
	\end{itemize}
\end{proof}
\Needspace{3\baselineskip}
\begin{lemma}
	\label{lem:Clifford-NilHecke-acts-Faithfully}
	The action of $\NHC_{I}$ on $\PolC_{I}$ is faithful.
\end{lemma}
\begin{proof}
	We recall briefly the proof of \autocite[prop.\ 2.11]{EKL:odd-NilHecke}
	which can be applied nearly literally.
	By the defining relations,
	it is clear that $\{y^{\bm{\alpha}}\delC_v\}_{\bm{\alpha}, v∈S_n}$
	is a generating set of  $\NHC_{I}$ as $\Cl_{I}$-superalgebra.
	We show by induction on $\ell(v)$ that
	$y^{\bm{\alpha}}\delC_v \sC_w$ are linearly independent elements of the $\Cl_{I}$-bimodule $\PolC_{I}$.
	This proves that the generating set $\{y^{\bm{\alpha}}\delC_v\}$ acts linearly independently in $\PolC_{I}$.

	\proofsection{Induction base}
	The only element of length $1$ is $e$.
	According to \cref{lem:relations-schubert-polynomials}, we obtain that
	$y^{\bm{\alpha}}\delC_v \sC_e \neq 0$ only if $v=e$.
	In this case, $y^{\bm{\alpha}}\delC_v ∈ \Cl_{I} y^{\bm{\alpha}}$.
	Thus, $y^{\bm{\alpha}}\delC_e$ cannot be a linear combination of $\{y^{\bm{\alpha'}}\delC_v\}_{v>e}$.

	\proofsection{Induction step}
	Assume that $\{y^{\bm{\alpha'}}\delC_{v'}\}_{v'<v}$ is linearly independent
	for a fixed $v∈S_n$.
	Assume $y^{\bm{\alpha}}\delC_{v}$ were a linear combination of $\{y^{\bm{\alpha'}}\delC_{v'}\}_{\ell(v')<\ell(v)}$.
	But $y^{\bm{\alpha}}\delC_{v} \sC_{v^{-1}} ∈ \Cl_{I}^* y^{\bm{\alpha}}$
	by the first case in \cref{lem:relations-schubert-polynomials},
	whereas a linear combination of $\{y^{\bm{\alpha'}}\delC_{v'}\}$ maps $\sC_{v^{-1}}$ either to zero
	or to a polynomial of degree strictly larger than $0$.
	This is a contradiction;
	there can thus be no linear relations between $y^{\bm{\alpha}}\delC_{v}$
	and elements $y^{\bm{\alpha'}}\delC_{v'}$ for $\ell(v')<\ell(v)$.
	Now, assume $y^{\bm{\alpha}}\delC_{v}$ were a non-trivial linear combination
	of $\{y^{\bm{\alpha'}}\delC_{v'}\}_{\ell(v')\geq\ell(v)}$.
	According to the second two cases in \cref{lem:relations-schubert-polynomials}, this is also a contradiction.
\end{proof}

Using Schubert polynomials, we may prove the following generalisation from \cite[prop.~2.13, cor.~2.14, lem.~2.16]{EKL:odd-NilHecke}.
Recall that $\PolC_{I}$ and $\Lambda_{I}$ of $\Pol_{I}$ and $\SPol_{I}$
have graded ranks $\rk_{q,\Cl_{I}} \Lambda_{I} = (n)_q! (1-q)^{-n}$
and $\rk_{q,\Cl_{I}} \PolC_{I} = (1-q)^{-n}$, respectively;
see \cref{rmk:graded-rank-d-symmetric-polynomials,cor:graded-rank-Clifford-polynomials}.
\begin{theorem}
	\label{prop:Clifford-Polynomials-free}
	\leavevmode
	\begin{thmlist}
		\item\label{lem:Clifford-Polynomials-free-over-symmetric}
			$\PolC_{I}$ is a free $\Lambda_{I}$-module of graded rank $(n)_q!$.
		\item\label{lem:Clifford-Symmetric-polynomials-are-kernel-of-Demazure}
			 The inclusion $\Lambda_{I} ⊆ \SPolC_{I}$
			from \cref{lem:Elementary-Clifford-Polynomials:Recursively} is in fact an equality.
		\item The faithful action of $\NHC_{I}$ on $\PolC_{I}$
		gives an isomorphism $\NHC_{I}\xrightarrow{\isom}\End_{\SPolC_{I}}(\PolC_{I})$.
	\end{thmlist}
\end{theorem}
\begin{proof}
	Since Schubert polynomials behave completely analogously to the classical case
	(apart from $\Cl_{I}$),
	the classical proof applies.
	Nevertheless, we briefly explain the argument.
	For details, consider \autocite[§3.2]{Lauda:Categorification-Quantum-sl2} for the even
	and \autocite[prop.\ 2.15]{EKL:odd-NilHecke} for the odd case.
	\begin{prooflist}
		\item\label{proof:Clifford-Polynomials-free-over-symmetric}
		Define the $\Cl_{I}$-bimodule
		\(\mathfrak{H}_{I} ≔ ⟨y_1^{a_1} \dotsm y_{n-1}^{a_{n-1}} \mid a_i \leq n-i⟩_{\Cl_{I}}\).
		This module has a basis given by the Schubert polynomials $\sC_w$;
		see the proof of \autocite[prop.\ 2.12]{EKL:odd-NilHecke}.
		For every index $i$, one may chose an exponent $0\leq a_i \leq n-i$,
		which yields a generator $y_i^{a_i}$ of (polynomial) degree $a_i$.
		The submodule generated by $y_i^\bullet$ thus has graded rank $(n-i)_q$;
		therefore, $\rk_{q,\Cl_{I}}(\mathfrak{H}_n) = (n)_q!$.
		The same proof as in the even (or odd) case shows that the multiplication
		\[
		\SPolC_{I} ⊗_{\Cl_{I}} \mathfrak{H}_{I} \to \PolC_{I}
		\]
		is injective.
		As both sides have the same graded $\Cl_{I}$-rank $(1-q)^{-n}$
		(see \cref{rmk:interesting-graded-ranks,rmk:graded-rank-d-symmetric-polynomials}),
		it is an isomorphism of $\Cl_{I}$-modules.
		This exhibits $\PolC_{I}$ as a free $\SPolC_{I}$-module of graded rank $(n)_q!$.

	\item In \ref{proof:Clifford-Polynomials-free-over-symmetric}, we have shown
		that $\PolC_{I}$ has a $\Lambda_{I}$-basis by Schubert polynomials $\sC_w$.
		Given a polynomial $p∈\PolC_{I}$,
		we take a linear combination
		$\sum_{w∈S_n} q_w \sC_w$ for $q_w ∈ \Lambda_{I}$
		with $q_w$ non-zero for some $w>e$.
		Assume the image
		\[
			\delC_i(p)=\delC_i\Bigl(\sum_w \delC_i(q_w)\sC_w\Bigr)
			= \sum_w s_i(q_w) \sC_{s_iw}
		\]
		were zero.
		By assumption, there was some non-zero $q_w$ for $w>e$, so
		there is still a non-zero term $s_i(q_w) \sC_{s_iw}$ in $\delC_i(p)$.
		However, since $\delC_i \sC_v = \sC_{s_iv}$
		if $s_iv$ is a reduced expression by \cref{lem:divided-difference-on-schubert-polynomials},
		$\delC_i$ maps no other Schubert polynomial to $\sC_{s_iv}$.
		Therefore, $\delC_i(p)$ cannot be zero.
		There are thus no $\Lambda_{I}$-linear combinations of Schubert polynomials in $\PolC_{I}$
		not already contained in $\ker \delC_i$.
		This proves the statement.

		\item We have already seen in \cref{lem:Clifford-NilHecke-acts-Faithfully}
		that $\NHC_n$ acts faithfully on the free $\Cl_{I}$-module $\PolC_{I}$.
		This action is compatible with the $\SPolC_{I}$-module structure
		from \ref{proof:Clifford-Polynomials-free-over-symmetric}:
		as $\NHC_{I} = \NCC_{I} ⊗_{\Cl_{I}} \PolC_{I}$,
		we see that $\PolC_{I}$ acts on itself,
		$\NCC_{I}$ acts trivially on $\SPolC_{I}$ by the very definition of the latter
		and that $\NCC_{I}$ indeed acts on $\mathfrak{H}_n$ because it acts by decreasing exponents.
		A graded rank computation
		\begin{equation*}
			\rk_{q,\Cl_{I}} \End_{\SPolC_{I}}(\PolC_{I})
			= \rk_{q,\Cl_{I}}(\SPolC_{I}) \cdot \rk_{q,\SPolC_{I}}(\PolC_{I})^2 = \frac{(n)_q!^2}{(n)_q! (1-q)^n}
			= \rk_{q,\Cl_{I}} \NHC_{I}
		\end{equation*}
		shows the asserted isomorphism.\qedhere
	\end{prooflist}
\end{proof}

\subsection{Complete symmetric polynomials}
\label{sec:Clifford-complete-symmetric-polynomials}
\newcommand{\hC}{h}
We give another collection of generators for $\SPolC_{I}$,
namely a Clifford-analogue for the complete homogeneous symmetric polynomials.
\begin{definition}
	\label{def:clifford-complete-symmetric-polynomial}
	Let $M_{(n)}$ be the $n\times n$-matrix
	\begin{align}
		M_{(n)}&≔\begin{psmallmatrix}
			\epsC^{(n)}_1 & \tilde{\kappa}_{1,2}\\
			- \epsC^{(n)}_2 && \tilde{\kappa}_{1,3}\\
			\vdots &&& \ddots\\
			(-1)^{n-2} \epsC^{(n)}_{n-1} &&&& \tilde{\kappa}_{1,n}\\
			(-1)^{n-1} \epsC^{(n)}_{n} &&&& 0
		\end{psmallmatrix}\\
	\shortintertext{where for $k<l-1$ we set}
		\label{eq:definition of kappas}
		\kappa_{k,l}
		&≔\gamma_{k+1,k} \, \gamma_{k+1,k+2}  \, \gamma_{k+2,k+1}  \, \gamma_{k+2,k+3}
			\dotsm\gamma_{l-2,l-3} \, \gamma_{l-2,l-1}\,\gamma_{l-1,l-2} \, \gamma_{l-1,l}\,\gamma_{l,l-1}\\\notag
		\kappa_{l-1,l}&≔\gamma_{l,l-1},\\\notag
		 \kappa_{l,l}&≔\gamma_{l,l+1}.\\
	\intertext{Furthermore, let}
			\notag
		\tilde{\kappa}_{k,l}
		&≔\kappa_{k,l}\kappa_{l,l} = \gamma_{k+1,k} \, \gamma_{k+1,k+2}\dotsm \gamma_{l-1,l-2} \, \gamma_{l-1,l} \, \gamma_{l,l-1} \, \gamma_{l,l+1}.
		\intertext{Note that we can “split” the $\kappa$'s by}
		\notag
		\kappa_{k,l} &\within{≔}{=}
		\underbrace{\gamma_{k+1,k} \, \gamma_{k+1,k+2}\dotsm\gamma_{m-1,m-2} \, \gamma_{m-1,m}\,\gamma_{m,m-1}}_{\kappa_{k,m}}
		\gamma_{m,m+1}
		\underbrace{\gamma_{m+1,m} \, \gamma_{m+1,m+2}\dotsm\gamma_{l-1,l-2} \, \gamma_{l-1,l}\,\gamma_{l,l-1}}_{\kappa_{m,l}}\\
		\label{eq:splitting-kappa-factors}
		&\within{≔}{=}  \tilde{\kappa}_{k,m}\kappa_{m,l}.
	\end{align}
	The \term{complete symmetric polynomial} $\hC^{(n)}_m$
	of (polynomial) degree $m$ on $n$ indeterminates
	is the top left entry of the matrix power $M_{(n)}^m$.
	In particular, $\hC^{(n)}_1=\epsC^{(n)}_1$ and $\hC^{(n)}_{0}=\epsC^{(n)}_0=1$.
\end{definition}
\begin{remark}
	The top row of $M_{(n)}^m$ has entries
	\[
		\bigl(M_{(n)}^m\bigr)_{1,\bullet}
		=
		\bigl(\hC^{(n)}_m,\: \hC^{(n)}_{m-1}\tilde{\kappa}_{1,2},\: \hC^{(n)}_{m-2}\tilde{\kappa}_{1,2}\tilde{\kappa}_{1,3},\dotsc,
		\tilde{\kappa}_{1,2}\dotsm\tilde{\kappa}_{1,m+1}, 0,\dotsc, 0\bigr).
	\]
\end{remark}
\begin{example}
	\begin{thmlist}
		\item If all indices are even and thus $\epsC^{(n)}_m = \varepsilon^{(n)}_m$ and $\kappa_{k,l}=1$ for all $k,l$,
			these polynomials coincide with the ordinary complete symmetric polynomials
			\[
				h^{(n)}_m = \sum_{\mathclap{1\leq k_{1}\leq\dotsb \leq k_{m}\leq n}} x_{k_{1}}\dotsm x_{k_m}.
			\]
		\item If all indices are odd,
		the resulting complete homogeneous symmetric polynomials coincide
		with those of \autocite{EKL:odd-NilHecke}
		up to a renormalisation involving the $\kappa$'s.
	\end{thmlist}
\end{example}
\begin{proposition}
	\label{prop:Complete-Clifford-Polynomials:Vanishing-Relation}
	We have
	\begin{equation*}
	\sum_{l=0}^{m} (-1)^l\,\hC^{(n)}_{m-l}\,\tilde{\kappa}_{1,2}\dotsm\tilde{\kappa}_{1,l}\epsC^{(n)}_l
	=
	\sum_{l=0}^{m} (-1)^l\,\tilde{\kappa}_{1,2}\dotsm\tilde{\kappa}_{1,l}\epsC^{(n)}_l \,\hC^{(n)}_{n-l}
	= \begin{cases}
		1 & \text{if $m = 0$,}\\ 0 & \text{otherwise.}
	\end{cases}
	\end{equation*}
\end{proposition}
\begin{proof}
	The complete homogeneous symmetric polynomial $\hC^{(n)}_{m}$
	is the top left entry of the matrix power $M_{(n)}^m=M_{(n)}^{m-1}\cdot M_{(n)}$
	and therefore equals
	\begin{equation*}
		\hC^{(n)}_{m}=
		\left(
			\hC^{(n)}_{m-1},\: \hC^{(n)}_{m-2}\tilde{\kappa}_{1,2},\dotsc,
			\tilde{\kappa}_{1,2}\dotsm\tilde{\kappa}_{1,m+1}, 0,\dotsc, 0
		\right)
		\cdot
		\left(
			\epsC^{(n)}_1,
			-\epsC^{(n)}_2,
			\dotsc,
			(-1)^{n}\epsC^{(n)}_n
		\right)^T.
	\end{equation*}
	We obtain the recursive formula
	\[
		\hC^{(n)}_{m} = \sum_{l=1}^{m} (-1)^{l-1} \hC^{(n)}_{m-l} \tilde{\kappa}_{1,2}\dotsm\tilde{\kappa}_{1,l} \epsC^{(n)}_l
	\]
	for $m\geq 1$.
	Moving $\hC^{\smash{(n)}}_{m}$ to the right side proves the statement.
	For the second equality, we denote the remaining entries of $M_{(n)}^m$ by
	\begin{equation*}
		\begin{psmallmatrix}
			\hC^{(n)}_m & \hC^{(n)}_{m-1}\tilde{\kappa}_{1,2} & \hC^{(n)}_{m-2}\tilde{\kappa}_{1,2}\tilde{\kappa}_{1,3} & \cdots\\
			\hC^{(n)}_{(1),m+1} & \hC^{(n)}_{(1),m}\tilde{\kappa}_{1,2} & \ddots\\
			\hC^{(n)}_{(2),m+2} & \ddots
		\end{psmallmatrix}
		≔M_{(n)}^m,
	\end{equation*}
	such that every $\hC^{(n)}_{(k),l}$ is a polynomial of degree $l$.
	The polynomial $\hC^{(n)}_{m}$ is the top left entry of
	${M_{(n)}^m=M_{(n)}\cdot M_{(n)}^{m-1}}$
	and is thus given by
	\begin{align*}
	\hC^{(n)}_{m} &= \epsC_1 \hC^{(n)}_{m-1} + \tilde{\kappa}_{1,2}\hC^{(n)}_{(1),m}\\
	\hC^{(n)}_{(1),m} &= -\epsC_2 \hC^{(n)}_{m-2} + \tilde{\kappa}_{1,3}\hC^{(n)}_{(2),m}\\
	\shortvdotswithin{=}
	\hC^{(n)}_{(k-1),m} &= (-1)^k\epsC_k \hC^{(n)}_{m-k} + \tilde{\kappa}_{1,k+1}\hC^{(n)}_{(k),m}.
	\end{align*}
	We obtain a recursion formula
	\begin{align*}
		\hC^{(n)}_{m} &= \epsC^{(n)}_1 \hC^{(n)}_{m-1} + \tilde{\kappa}_{1,2}
		\left(
			-\epsC^{(n)}_2 \hC^{(n)}_{m-2} + \tilde{\kappa}_{1,3}\bigl(\cdots(
				\pm\epsC^{(n)}_{m-1} \hC^{(n)}_{1} \mp \tilde{\kappa}_{1,m}\epsC^{(n)}_{m}
			)\cdots\bigr)
		\right)\\
		&= \sum_{l=1}^m (-1)^{l-1} \tilde{\kappa}_{1,2}\dotsm \tilde{\kappa}_{1,l} \epsC^{(n)}_{l} \, \hC^{(n)}_{m-l}.
	\end{align*}
	Again, moving $\hC^{(n)}_{m}$ to the right side proves the claim.
\end{proof}

\subsection{Interlude: cohomology of Grassmann and partial flag varieties}
We recall some classical (non-super) theory on the cohomology rings of Grassmann
and partial flag varieties.
Our aim is to build super-analogues for these rings.
\begin{definition}
	Fix natural numbers $k\leq n$.
	The set of all $k$-dimensional subspaces
	$F⊆k^n$ of the $n$-dimensional $k$-vector space
	forms an algebraic variety, called \term{Grassmann variety} and denoted by $\Gr(k,n)$.
	This is a special case of the following construction:
\end{definition}
\begin{definition}
	Fix natural numbers $\ell\leq n$.
	A \term{partial flag} $F_\bullet$ of length $\ell$ is a chain $F_\bullet$ of $k$-vector subspaces
	$0 = F_0 ⊆ F_1 ⊆ \dotsb ⊆ F_\ell = k^n$.
	The set of all partial flags of length $\ell$ is the
	\term{(partial) flag variety} $\Fl(\ell; n)$.
	If $\ell=n$, then  $\Fl(n)≔\Fl(n; n)$ is called the \term{full flag variety}.
	To a flag $F_\bullet$ of length $\ell$, one associates its \term{dimension vector}
	$\bm{k}=(\dim_k F_0, \dotsc, \dim_k F_\ell)$.
	The connected components of $\Fl(\ell; n)$
	are the varieties $\Fl(\bm{k})$ of partial flags with a common dimension vector $\bm{k}$.
	In particular, $\Gr(k,n)=\Fl(\bm{k})$ with $\bm{k}=(0, k, n)$.
\end{definition}
From now on, we assume that $k = \mathbf{R}$ or $k = \mathbf{C}$.
In this case, the flag varieties are smooth manifolds.
In the following, cohomology $H^*$ always means singular cohomology with coefficients in the field $k$.
\begin{proposition}[{\autocites[prop.\ 3.6.15, rmk.\ 3.6.16]{Manivel}[161]{Fulton:Young-Tableaux}}]
	\label{prop:Flag-variety-cohomology}
	The full flag variety $\Fl(n)$
	has cohomology ring $H^*(\Fl(n))\isom\Pol_n/(\SPol_n)_+$.
	For the dimension vector $\bm{k}$ with $k_0=0$, $k_\ell=n$,
	the partial flag variety $\Fl(\bm{k})$ has the cohomology ring
	\begin{equation}
		H_{\bm{k}}≔H^*(\Fl(\bm{k}))\isom \frac{\SPol_{k_1-k_0}⊗_k\dotsb⊗_k\SPol_{k_\ell-k_{\ell-1}}}{(\SPol_n)_+}.
	\end{equation}
	We denote the nominator by $\SPol_{\bm{k}}$.
\end{proposition}
\begin{corollary}
	\label{cor:Grassmannian-cohomology}
	The Grassmann variety $\Gr(k,n)$ has cohomology ring
	\begin{equation}
		H_{(k,n)}≔H^*(\Gr(k,n))\isom \frac{k\bigl[\varepsilon^{(n)}_1, \dotsc, \varepsilon^{(n)}_m\bigr]}%
		{\bigl(h^{(n)}_{n-k+1}, \dotsc, h^{(n)}_{n}\bigr)},
	\end{equation}
	where the $\varepsilon$'s and $h$'s are the elementary and the complete symmetric polynomials, respectively.
\end{corollary}

Recall the notion of graded ranks from \cref{sec:graded-ranks}.
\begin{lemma}
	\label{lem:graded-dimension-Grassmann-cohomology}
	The cohomology ring $H(k,n)≔H^*(\Gr(k,n))$ of the Grassmannian
	has graded dimension
	$
		\dim_{q,k} H(k,n) = \binom{n}{k}_q
	$
	with the \term{$q$-binomial coefficient}
	 $\binom{n}{k}_q ≔ \frac{(n)_q!}{(k)_q! (n-k)_q!}$.
\end{lemma}
\begin{proof}
	Without quotienting out the ideal, the polynomial ring
	$\SPol_k≔k[\varepsilon_1, \dotsc, \varepsilon_k]$  in the symmetric polynomials
	has graded dimension
	$\dim_{q,k} \SPol_k = (1-q)^{-1}\dotsm(1-q^k)^{-1}$.
	To count the graded dimension of the ideal $\mathfrak{a}=(h_{n-k+1}, \dotsc, h_n)$ to be quotiented out,
	we note that each generator $h_m$ contributes a principal ideal with graded dimension
	$\dim_{q,k} (h_m)=q^m \dim_{q,k}\SPol_k$.
	We then proceed by the inclusion\-/exclusion principle to obtain
	\begin{align*}
		&\eqsp\dim_{q,k} (h_{n-k+1}, \dotsc, h_n)\\
		&=
		\Bigl(
			\bigl(q^{n-k+1} + \dotsb q^n\bigr) - \bigl(q^{2(n-k)+3} + \dotsb + q^{2n-1}\bigr) + \dotsb \pm q^{(n-k+1)+\dotsb + n}
		\Bigr) \dim_{q,k}\SPol_k \\
		&= \bigl(1- (1-q^{n-k+1})\dotsm(1-q^n)\bigr)\dim_{q,k}\SPol_k.
	\end{align*}
	The quotient thus has graded dimension
	\begin{align*}
		\dim_{q,k} (\SPol_k/\mathfrak a) &= \frac{(1-q^{n-k+1})\dotsm(1-q^n)}{ (1-q)\dotsm(1-q^k)}
		= \frac{(1-q)^{n-(n-k)} (n)_q! / (n-k)_q!}{(1-q)^k (k)_q!}
		= \tbinom{n}{k}_q.\qedhere
	\end{align*}
\end{proof}
For the cohomology ring of partial flag varieties, this yields:
\begin{corollary}
	\label{cor:graded-dimension-flag-cohomology}
	Let $\bm{k}≔(0=k_0\leq \dotsc\leq k_r=n)$ be a dimension vector.
	The cohomology ring $H_{\bm{k}}$ of the partial flag variety $\Fl(\bm{k})$
	has as graded dimension the \term{$q$-multinomial coefficient}
	\begin{equation*}
		\dim_{q,k} H(\bm{k}) = \dbinom{n}{k_1, k_2-k_1, \dotsc, k_r-k_{r-1}}_q
		≔ \frac{(n)_q!}{(k_1)_q! (k_2-k_1)_q! \dotsm (k_r-k_{r-1})_q!}.
	\end{equation*}
\end{corollary}
\begin{proof}
	The cohomology ring is the quotient
	\[
		\frac{
			\SPol_{k_1}⊗\dotsb ⊗\SPol_{k_l-k_{l-1}}
		}{
			\left(
				1-\bigl(1+\varepsilon_1^{(k_1)}+\dotsb+\varepsilon^{(k_1)_{k_1}}\bigr)
				\dotsm
				\bigl(1+\varepsilon_1^{(k_l-k_{l-1})}+\dotsb+\varepsilon_{(k_l-k_{l-1})}^{k_l-k_{l-1}}\bigr)
			\right)
		}
	\]
	of $\SPol_{\bm{k}} ≔ \SPol_{k_1}⊗\dotsb ⊗\SPol_{k_l-k_{l-1}}$.
	Denote the ideal quotiented out by $\mathfrak{a}$ and proceed
	as in the proof of \cref{lem:graded-dimension-Grassmann-cohomology}.
	We see that the nominator has graded dimension
	\begin{align*}
		\dim_{q,k} \Lambda
		 &= \prod_{a=1}^{l}\frac{1}{1-q}\dotsm\frac{1}{1-q^{k_a-k_{a-1}}}
		= \frac{1}{(1-q)^n}\prod_{a=1}^{l}\frac{1}{(k_a-k_{a-1})_q!},
	\end{align*}
	and the ideal $\mathfrak{a}$ quotiented out has graded  dimension
	\[
		\dim_{q,k}\mathfrak a = \bigl((1-q)^n(n)_q!-1\bigr)\dim_{q,k} \Lambda.
	\]
	This shows the assertion.
\end{proof}

\subsection{Cyclotomic quotients}
\label{sec:cyclotomic-quotients}
It is known that in the non-super case,
the cohomology ring $H_{(k,n)}$ of Grassmann varieties (see \cref{cor:Grassmannian-cohomology})
is Morita equivalent to the \term{cyclotomic quotient} of the NilHecke algebra
\cite[§5]{Lauda:Introduction-to-Diagrammatic-Algebra}.
For the odd setting,
the respective Morita equivalence has been established in \cite[§5]{EKL:odd-NilHecke}.
We introduce analogues of both rings in the Clifford setting
and prove that they are Morita equivalent.
\begin{definition}
	The \term{$m$-th cyclotomic quotient} of the Hecke Clifford superalgebra $\NHC_n$
	is the quotient $\NHC_n^m ≔ \NHC_n/\bigl((\kappa_{1,n}y_n)^m\bigr)$ by the two-sided ideal $\bigl((\kappa_{1,n}y_n)^m\bigr)$,
	with the coefficients $\kappa_{k,l}$ from \cref{eq:definition of kappas}.
	\label{def:Clifford-Grassmann-Cohomology}%
	We call the quotient $\HGrC_{(m,n)} ≔ \SPolC_{I}/\bigl(\hC^{(n)}_{k}\bigr)_{k>n-m}$
	by the two-sided ideal $\bigl(\hC^{(n)}_{k}\bigr)_{k>n-m}$
	the \term{Clifford Grassmann ring}.
\end{definition}
\begin{remark}
	In contrast to the standard definition of the cyclotomic quotient,
	we mod out some power of $y_n$ instead of $y_1$;
	the latter would necessitate taking the transpose of the matrix $M_{(n)}$
	in \cref{def:clifford-complete-symmetric-polynomial}.
\end{remark}
\begin{theorem}
	\label{prop:grassmann-cyclotomic-morita-eq}
	The rings $\NHC_n^m$ and $\HGrC_{(m,n)}$ are Morita-equivalent.
\end{theorem}
\begin{proof}
	The proof in \autocite[§5]{Lauda:Introduction-to-Diagrammatic-Algebra}
	does not rely on commutativity of the rings involved and thus applies immediately in our setting.
	We review the argument briefly.
	The reader is strongly encouraged to verify the following calculation
	for $n=3$.
	Let us start with
	\begin{align}\notag
		\epsC^{(n)}_n
		&= \gamma_{1,2}y_1 \, \gamma_{2,3}y_2 \dotsm \gamma_{n,n+1}y_n\\\notag
		&= \bigl(\gamma_{1,2}y_1 \, \gamma_{2,3}y_2 \dotsm \gamma_{n-1,n}y_{n-1}\bigr)\gamma_{n,n+1}y_n\\\notag
		&= \bigl(\epsC^{(n)}_{n-1} - \tilde{\epsC}^{(n)}_{n-1}\bigr)
			\smash{\underbrace{\gamma_{n,n+1}}_{\kappa_{n,n}}}y_n\\
	\shortintertext{where we set}\notag
		\tilde{\epsC}^{(n)}_{n-1}
		&= \gamma_{1,2}y_1\dotsm\gamma_{n-2,n-1}y_{n-2}\,\gamma_{n,n-1}y_n
		+ \dotsb
		+ \gamma_{2,1}y_2\,\gamma_{3,2}y_3\dotsm \gamma_{n,n-1}y_n\\\notag
		&= \bigl(\gamma_{1,2}y_1\dotsm \gamma_{n-2,n-1}y_{n-2}
			+\dotsb
			+\gamma_{2,1}y_2\dotsb \gamma_{n-1,n-2}y_{n-1}
			\bigr) \gamma_{n,n-1}y_n\\\notag
		&= \bigl(\epsC^{(n)}_{n-2} - \tilde{\epsC}^{(n)}_{n-2}\bigr)
		\smash{\underbrace{\gamma_{n,n-1}}_{\kappa_{n-1,n}}} y_n\\
	\shortintertext{where we set}\notag
		\tilde{\epsC}^{(n)}_{n-2}
		&= \bigl(\gamma_{1,2}y_1\dotsm \gamma_{n-3,n-2}y_{n-3}
		+ \dotsb
		+ \gamma_{3,2}y_3\dotsm \gamma_{n-1,n-2}y_{n-1}\bigr) \gamma_{n-1,n-2}\gamma_{n-1,n}\gamma_{n,n-1} y_n\\\notag
		&= \bigl(\epsC^{(n)}_{n-3}-\tilde{\epsC}^{(n)}_{n-3}\bigr)
			\smash{\underbrace{\gamma_{n-1,n-2}\gamma_{n-1,n}\gamma_{n,n-1}}_{\kappa_{n-2,n}}} y_n\\\notag
	&\vdotswithin{=}\\\notag
		\tilde{\epsC}^{(n)}_{2}
		&= \bigl(\epsC^{(n)}_{1} - \tilde{\epsC}^{(n)}_{1}\bigr)
			\smash{\underbrace{\gamma_{3,2}\gamma_{3,4}\dotsm\gamma_{n-1,n-2}\gamma_{n-1,n}\gamma_{n,n-1}}_{\kappa_{2,n}}} y_n\\
	\shortintertext{where}\notag
		\tilde{\epsC}^{(n)}_{1} &=
			\underbrace{\gamma_{2,1}\gamma_{2,3}\dotsm\gamma_{n-1,n-2}\gamma_{n-1,n}\gamma_{n,n-1}}_{\kappa_{1,n}} y_n.\\
	\intertext{%
		To state the above calculation differently, we have that
	}\notag
		0 &=
		\epsC^{(n)}_n - \Bigl(\epsC^{(n)}_{n-1} - \bigl(\cdots\bigl(\epsC^{(n)}_1 - \kappa_{1,n}y_n\bigr)\cdots\bigr)\kappa_{n-1,n}y_n\Bigr) \kappa_{n,n} y_n\\\notag
		&= \sum_{k=0}^{n} (-1)^k \epsC^{(n)}_{n-k} \kappa_{n-k+1,n}y_n\dotsm \kappa_{n,n}y_n\\
		&= \label{eq:relation-of-bi-basis}
			(-1)^n\sum_{k=0}^{n} (-1)^{k} \epsC^{(n)}_k \smashoperator{\prod_{l=k+1}^{n}}\kappa_{l,n}y_n.
	\end{align}
	Set $b_l ≔ \prod_{l=k+1}^{n}\kappa_{l,n}y_n$ for $1 \leq k\leq n$;
	in particular $b_{n} ≔ 1$.
	Recall the $\Cl_{I}$-bimodule $\mathfrak{H}_{I}$
	from the proof of \cref{lem:Clifford-Polynomials-free-over-symmetric}.
	We define a similar free $\Cl_{I}$-bimodule of graded rank $(n)_q!$.
	For multiindices $\bm{\alpha}$, let
	\begin{align}
		B_{\bm{\alpha}} &\coloneqq ⟨y^{\bm{\alpha}} b_1, \dotsc, y^{\bm{\alpha}} y_n^{n-1} \cdot b_n⟩_{\Cl_{I}},\\
		\label{eq:cyclotomic-quotients:H-tilde}
		\tilde{\mathfrak{H}}_{I} &≔ ⟨y_1^{\alpha_1} \dotsm y_{n-1}^{\alpha_{n-1}} \mid \alpha_i <i⟩_{\Cl_{I}}
		= ⨁_{\bm{\alpha}} B_{\bm{\alpha}}
	\end{align}
	It is analogous to the proof of \cref{lem:Clifford-Polynomials-free-over-symmetric}
	that multiplication gives an isomorphism
	$\tilde{\mathfrak{H}}_{I} ⊗_{\Cl_{I}} \SPolC_{I} \xto{\isom} \PolC_{I}$.
	Recall from \cref{eq:splitting-kappa-factors} how to split the $\kappa$'s.
	Multiplication from the left by $\kappa_{1,n}y_n$ acts on this basis of $B_{\bm{\alpha}}$ by
	\begin{alignat*}{3}
		(\kappa_{1,n}y_n\cdot)\colon B_{\bm{\alpha}} &\to B_{\bm{\alpha}},
		\\*
		1 = b_{n}
			&  \mapsto \kappa_{1,n}y_n
			&& \overset{\mathclap{\raisebox{3pt}[0pt][0pt]{\scriptsize{\cref{eq:splitting-kappa-factors}}}}}{=} \kappa_{1,n}\kappa_{n,n}\kappa_{n,n} y_n
			&&=\smash{\overbrace{\kappa_{1,n}\kappa_{n,n}}^{\tilde{\kappa}_{1,n}}} b_{n-1}
		\\*
		b_{n-1}
			&  \mapsto \kappa_{1,n}y_n b_{n-1}
			&&= \tilde\kappa_{1,n-1} \kappa_{n-1,n}y_n b_{n-1}
			&&= \tilde\kappa_{1,n-1} b_{n-2}
		\\*
			& \vdotswithin{\mapsto}
		\\
		b_2 & \mapsto \kappa_{1,n}y_n b_2 &&= \tilde\kappa_{1,2}\kappa_{2,n} y_n b_2
		&&=
			\tilde\kappa_{1,2}
			b_1 \\*
		b_1 &\mapsto \kappa_{1,n}y_nb_1 &&\overset{\mathclap{\raisebox{3pt}[0pt][0pt]{\scriptsize{\cref{eq:relation-of-bi-basis}}}}}{=}
			\textstyle\sum_{k=1}^{n} (-1)^k \epsC^{(n)}_k
			\underbrace{
				\textstyle \prod_{l=k+1}^{n}\kappa_{l,n}y_n
			}_{b_k}.
			\span\span 
	\end{alignat*}
	This shows that multiplication with $\kappa_{1,n}y_n$ from the left
	acts on the basis from \cref{eq:cyclotomic-quotients:H-tilde}
	by the matrix
	\begin{equation*}
		\begin{psmallmatrix}
			\epsC^{(n)}_1 & \tilde{\kappa}_{1,2}\\
			- \epsC^{(n)}_2 && \tilde{\kappa}_{1,3}\\
			\vdots &&& \ddots\\
			(-1)^{n-2} \epsC^{(n)}_{n-1} &&&& \tilde{\kappa}_1,n\\
			(-1)^{n-1} \epsC^{(n)}_{n} &&&& 0
		\end{psmallmatrix} = M_{(n)}.
	\end{equation*}
	Quotienting out the two-sided ideal $(\kappa_{1,n}y_n)^m$ of $\NHC_n$
	is the same as requiring that $M_{(n)}^m=0$.
	\begin{claim}
		The ideal $(h_{n-m+1},\dotsc, h_{n})$ of $\PolC_{I}$
		is also generated by the first column of $M_{(n)}^{m+1}$.
	\end{claim}
	By definition, $h_{k+1}$ is the top left entry of $M_{(n)}\cdot M_{(n)}^{k}$ for any $k$.
	Recall that we denoted the entries of the first column $\bigl(M_{(n)}^{k+1}\bigr)_{\bullet,1}$ of $M_{(n)}^{k+1}$
	by
	\[
		\bigl(M_{(n)}^{k+1}\bigr)_{\bullet,1} \eqqcolon \bigl(h_{k+1}, h_{(1), k+2}, h_{(2), k+3}, \dotsc, h_{(n-k-1), n}\bigr)^T.
	\]
	We thus have
	\begin{align*}
		h_{n-m+2} &= \epsC^{(n)}_1 h_{n-m+1} + \tilde{\kappa}_{1,2} h_{(1),n-m+2}\\
		&\equiv \tilde{\kappa}_{1,2} h_{(1),n-m+2} \pmod{h_{n-m+1}}\\
		h_{n-m+3} &=
		\epsC^{(n)}_1 h_{n-m+2} + \tilde{\kappa}_{1,2} h_{(1),n-m+3}\\
		&= \mathrlap{\epsC^{(n)}_2 \left(
			\epsC^{(n)}_1 h_{n-m+1} + \tilde{\kappa}_{1,2} h_{(1),n-m+2}
		\right)+
		\tilde{\kappa}_{1,2}\left(
			-\epsC^{(n)}_2 h_{n-m+1} + \tilde{\kappa}_{1,3} h_{(2),n-m+2}
		\right)}\\
		&\equiv \tilde{\kappa}_{1,2} \tilde{\kappa}_{1,3}  h_{(2), n-m+3} \pmod{h_{n-m+1}, h_{n-m+2}}\\
		\shortvdotswithin{=}
		h_{n} &\equiv \tilde{\kappa}_{1,2} \dotsm \tilde{\kappa}_{1,m}  h_{(m-1), n} \pmod{h_{n-m+1}, \dotsc, h_{n-1}}.
	\end{align*}
	This proves the claim, since all $\tilde{\kappa}$'s are units.
	Taking the product  $M_{(n)}^{m} = M_{(n)}^{m-n+1} \cdot  M_{(n)}^{n-1}$
	shows that the last column of $M_{(n)}^{m}$ has entries
	\[
		\bigl(M_{(n)}^{m}\bigr)_{\bullet,n} =
		\bigl(h_{n-m+1}, h_{(1), n-m+2}, \dotsc, h_{(m-1,n)}\bigr)^T \tilde{\kappa}_{1,2}\dotsm\tilde{\kappa}_{1,n}.
	\]
	Thus, the entries of $\bigl(M_{(n)}^{m}\bigr)_{\bullet,n}$
	also generate the ideal $(h_{n-m+1},\dotsc, h_{n})$ of $\PolC_{I}$.
	\begin{claim}
		All entries of $M_{(n)}^{m}$ are $\SPolC_{I}$-linear combinations
	of entries of the last column $\bigl(M_{(n)}^{m}\bigr)_{\bullet,n}$.
	\end{claim}
	Since $M_{(n)}^{k+1} = M_{(n)}^{1} M_{(n)}^{k}$ for any $k$,
	the entries of $\bigl(M_{(n)}^{k+1}\bigr)_{\bullet,1}$
	are $\SPolC_{I}$-linear combinations of entries of $\bigl(M_{(n)}^{k}\bigr)_{\bullet,1}$.
	Since $M_{(n)}^{k+1} = M_{(n)}^{k} M_{(n)}^{1}$,
	we have $\bigl(M_{(n)}^{k+1}\bigr)_{\bullet,2}=\tilde\kappa_{1,2} \bigl(M_{(n)}^{k}\bigr)_{\bullet,1}$.
	The entries of the first column $\bigl(M_{(n)}^{k+1}\bigr)_{\bullet,1}$ thus are linear combinations
	of entries of the second column $\bigl(M_{(n)}^{k+1}\bigr)_{\bullet,2}$;
	the claim follows iteratively.

	Altogether, we have shown that requiring $M_{(n)}^{m}=0$ is the same
	as quotienting out the ideal $(h_{n-m+1},\dotsc, h_{n})$ of $\PolC_{I}$.
	Applying this to each summand $B_{\bm{\alpha}}$ of \cref{eq:cyclotomic-quotients:H-tilde}
	yields
	\[
		\NHC_n^m \isom \operatorname{Mat}_{n!}(\SPolC_{I})/\bigl(M_{(n)}^{m}\bigr) \isom \operatorname{Mat}_{n!}(\HC_{(m,n)}),
	\]
	where $\bigl(M_{(n)}^{m}\bigr)$ is the two-sided ideal generated by matrices $M_{(n)}^{m}$.
	This establishes the asserted Morita equivalence.
\end{proof}

\subsection{Clifford superalgebras associated to partial flag varieties}
\label{sec:partial-flag-varieties}
Let $\epsC^{(k,n)}_m$ be the polynomial of degree $m$
in the indeterminates $y_k,\dotsc,y_n$
that is obtained from $\epsC^{(n-k)}_m$ by replacing every index $i$ by $i+k$.
Recall from \cref{lem:Elementary-Clifford-Polynomials:Recursively}
the recursion formula for the elementary $\delC$-symmetric polynomials.
We can easily derive the following corollary
by regrouping the terms of \crefrange{eqn:Elementary-Clifford-Polynomials:Recursively 1}{eqn:Elementary-Clifford-Polynomials:Recursively n}:
\begin{corollary}
	\label{cor:Elementary-Clifford-Polynomials:Recursively-Left}
	The elementary $\delC$-symmetric polynomials satisfy the recursion formula
	\begin{equation}
		\label{eqn:Elementary-Clifford-Polynomials:Recursively-Left}
		\epsC^{(n)}_m = \gamma_{1,2}\,y_1\, \epsC^{(1,n)}_{m-1} + \gamma_{2,1}\gamma_{2,3}\,\epsC^{(1,n)}_m
	\end{equation}
	for the elementary $\delC$-symmetric polynomials,
	where $\gamma_{i,i\pm 1}$ is as defined in \cref{lem:Elementary-Clifford-Polynomials:two-indeterminates}.
\end{corollary}
\begin{table}
	\caption{%
		Elementary $\delC$-symmetric polynomials $\epsC^{(k,n)}_m$ for $n=3$,
		as defined in \cref{def:Elementary-Clifford-Polynomials:two-indeterminates}.
		The leftmost column is the same as in \cref{tab:elementary-d-symmetric-polynomials}.
		The braces illustrates how one can construct the polynomials recursively
		as stated in \cref{cor:Elementary-Clifford-Polynomials:Recursively-Left},
		starting with the rightmost polynomial $\epsC^{(2,3)}_3 = \gamma_{3,4}y_3$.
	}
	\label{tab:elementary-d-symmetric-polynomials:Recursively-Left}
	\centering
	\small
	\addtolength{\arraycolsep}{1em}
	$\begin{array}{@{}l@{}lll@{}}
		\toprule
		m\backslash k+1 & \hfil 1 & \hfil 2 & \hfil 3\\
		\midrule
		1
		&
		\gamma_{1,2}\, y_1
		+ \underbrace{
	  \gamma_{2,1}\, y_2
			+ \gamma_{2,1}\gamma_{2,3}\gamma_{3,2}\ y_3
		}_{\gamma_{2,1}\gamma_{2,3}\epsC^{(1,3)}_1}
		&
		\gamma_{2,3}\, y_2
		+ \underbrace{\gamma_{3,2}\,y_3}_{\mathclap{\gamma_{3,2}\gamma_{3,4}\,\epsC^{(1,3)}_1}}
		& \gamma_{3,4}\,y_3
		\\
		2
		&
		\underbrace{
	  \gamma_{1,2}\,y_1\gamma_{2,3}\, y_2
			+ \gamma_{1,2}\,y_1\gamma_{3,2}\, y_3
		}_{\gamma_{1,2}\,y_1\epsC{(1,3)}_1}
		+ \gamma_{2,1}\,y_2\ \gamma_{3,2}\, y_3
		& \gamma_{2,3}\, y_2\,\gamma_{3,2}\, y_3\\
		3
		& \gamma_{1,2}\,y_1\,\underbrace{\gamma_{2,3}\,y_3\,\gamma_{3,4}\,y_3}_{\epsC^{(1,3)}_2}\\
		\bottomrule
	\end{array}$
\end{table}
\begin{example}
	The “regrouping” of terms is best made explicit
	by considering the first $\delC$-symmetric polynomials
	listed explicitly in \cref{tab:elementary-d-symmetric-polynomials}.
	We can start with the polynomial $\epsC^{(2,3)}_1≔\gamma_{3,4} y_3$
	as defined in \cref{def:Elementary-Clifford-Polynomials:two-indeterminates}
	and construct the polynomials $\epsC^{(k,3)}_m$
	for $0\leq k<3$ and $m\leq n-k$
	as described in the corollary.
	The resulting polynomials are listed in \cref{tab:elementary-d-symmetric-polynomials:Recursively-Left}
	with the grouping indicated by braces.
\end{example}
The recursive relations in
\crefrange{eqn:Elementary-Clifford-Polynomials:Recursively 1}{eqn:Elementary-Clifford-Polynomials:Recursively n} and \cref{eqn:Elementary-Clifford-Polynomials:Recursively-Left}
from \cref{lem:Elementary-Clifford-Polynomials:Recursively,cor:Elementary-Clifford-Polynomials:Recursively-Left}
are subsumed in the following definition:
\begin{definition}
	\label{definition:Clifford-Symmetric-Polynomials:Lambda-Coefficients}
	For $1<k<n$, let $\lambda^{(0,k,n)}_{m,l}∈\Cl_{I}$ be the coefficients defined by
	\begin{equation}
		\label{eqn:Clifford-Flag-Varieties:Expansion-Coefficients}
		\epsC^{(n)}_m = \sum_{\mathclap{0\leq l\leq m}}
		\epsC^{(k)}_l \lambda^{(0,k,n)}_{m,l}\, \epsC^{(k,n)}_{m-l}.
	\end{equation}
\end{definition}
\begin{example}
	\label{example:Clifford-Symmetric-Polynomials:Lambda-Coefficients}
	By \crefrange{eqn:Elementary-Clifford-Polynomials:Recursively 1}{eqn:Elementary-Clifford-Polynomials:Recursively n} and \cref{eqn:Elementary-Clifford-Polynomials:Recursively-Left}
	we know the coefficients $\lambda^{(0,k,n)}_{m,l}$ for two special cases:
	\begin{thmlist}
	\item $\lambda^{(0,n-1,n)}_{m,l}$ is the coefficient of the expansion
		$\epsC^{(n)}_m = \sum_l \epsC^{(n-1)}_l\lambda_l \gamma_{n,n+1} y_n$.
		The recursion formula from \cref{lem:Elementary-Clifford-Polynomials:Recursively} gives
		\settowidth{\algnRef}{$\gamma_{m+1,m}\gamma_{m+1,m+2}\dotsm \gamma_{n-1,n-2}\gamma_{n-1}{n}$}
		\begin{align*}
			\lambda^{(0,n-1,n)}_{m,l} &=
			\begin{cases}
				1 & \text{if $l=m$}\\
				\gamma_{m+1,m}\gamma_{m+1,m+2}\dotsm \gamma_{n-1,n-2}\gamma_{n-1}{n} & \text{if $l=m-1$}\\
				0 & \text{otherwise.}
			\end{cases}
		\end{align*}
	\item $\lambda^{(0,1,n)}_{m,l}$ is the coefficient of the expansion
		$\epsC^{(n)}_m = \sum \gamma_{1,2} y_1 \lambda_l \epsC^{(1, n-1)}_l$.
		By \cref{cor:Elementary-Clifford-Polynomials:Recursively-Left} these are given by
		\begin{align*}
			\mathmakebox[\widthof{$\lambda^{(0,n-1,n)}_{m,l}$}][r]{\lambda^{(0,1,n)}_{m,l} }
			&=
			\begin{cases}
				\mathmakebox[\algnRef][l]{\gamma_{2,1}\gamma_{2,3}} & \text{if $l=0$}\\
				\mathmakebox[\algnRef][l]{1} & \text{if $l=1$}\\
				\mathmakebox[\algnRef][l]{0} & \text{otherwise.}
			\end{cases}
		\end{align*}
	\end{thmlist}
	In general, however, it seems to be difficult to provide an explicit formula for $\lambda^{(0,k,n)}_{m,l}$
	for arbitrary $k$.
\end{example}
\begin{definition}
	Recall from \cref{lem:Elementary-Clifford-Polynomials:Recursively}
	that we defined $\SPolC_{n}≔⟨\epsC^{(n)}_1,\dotsc, \epsC^{(n)}_n⟩$
	as the $\Cl_{I}$-subalgebra of $ \PolC_{I}$
	generated by the elementary $\delC$-symmetric polynomials.
	We have shown $\SPolC_{n} = ⋂_{i=1}^{n-1}\ker \delC_i$
	in \cref{lem:Clifford-Symmetric-polynomials-are-kernel-of-Demazure}.
	We additionally define the $\Cl_{I}$-superalgebras
	$\SPolC_{(k, n)} ≔ ⟨\epsC^{(k,n)}_m\mid 1\leq m\leq m-k⟩ ⊆ \bigcap_{i=k+1}^{n-1}\ker \delC_i$
	and more generally
	\begin{align}
		\notag
		\SPolC_{( k_0, \dotsc, k_l)}
		&≔
			\Cl_{I}\bigl\langle\epsC^{(k_0,k_1)}_{m_0},\epsC^{(k_1,k_2)}_{m_1},
			\dotsc, \epsC^{(k_{l-1},k_l)}_{m_{l-1}} \bigm| \forall j\colon 1\leq m_j\leq k_{j+1}-k_j \bigr\rangle\\*
		&\within{≔}{=}
		\label{eq:LPolC and kernels}
		\smashoperator{\bigcap_{\substack{k_0 < i < k_l \\ i \neq k_0, \dotsc, k_l}}} \ker\delC_{i}
	\end{align}
	In particular, the $\Cl_{I}$-superalgebra of $\delC$-symmetric polynomials
	defined in \cref{lem:Elementary-Clifford-Polynomials:Recursively}
	is $\SPolC_{n} = \SPolC_{(0,n)}$.
\end{definition}
\begin{remark}
	\label{rmk:General-Clifford-Symmetric-polynomials-are-kernel-of-Demazure}
	Note that \cref{lem:Clifford-Symmetric-polynomials-are-kernel-of-Demazure}
	applies to each $\bigl\langle\epsC^{(k_j,k_{j+1})}_{m_j} \mid 1\leq m_j\leq k_{j+1}-k_j\bigr\rangle$.
	The obvious inclusion in \cref{eq:LPolC and kernels} is thus indeed an equality.
\end{remark}

\subsubsection{One step flag varieties}
\label{sec:Clifford-one-step-flag-varieties}%
For $0 < j < k_l$, let $( k_0, \dotsc, \widehat{k_{j}}, \dotsc, k_l)$ denote the tuple $( k_0, \dotsc, k_l)$ with the entry $k_j$ omitted.
Consider
\begin{equation}
	\SPolC_{( k_0, \dotsc, \widehat{k_{j}}, \dotsc, k_l)} = \smashoperator{\bigcap_{\substack{k_0 < i < k_l \\ i \neq k_0, \dotsc, \widehat{k_{j}}, \dotsc, k_l}}} \ker\delC_{i};
\end{equation}
that is, in contrast to \cref{eq:LPolC and kernels}, the intersection with $\ker\delC_{k_j}$ is not omitted.
Clearly there is an inclusion
\begin{equation}
	\SPolC_{( k_0, \dotsc, \widehat{k_{j}}, \dotsc, k_l)}
	\subseteq \SPolC_{( k_1, \dotsc, k_l)}\label{eqn:Clifford-Symmetric-Rings:Inclusion}
\end{equation}
of superalgebras.
\begin{lemma}
	Let $\bm{k}=(0,k, k+1,\dotsc, n)$ for some $k$,
	so that the there are inclusions
	\begin{equation}
		\label{eqn:Clifford-Symmetric-Rings:One-step-Inclusion}
		\SPolC_{(0, \dotsc, k, \dotsc, n)}
		\subseteq\SPolC_{(0,\dotsc, k,k+1,\dotsc, n)}
		\supseteq\SPolC_{(0,\dotsc,  k+1,\dotsc,  n)}.
	\end{equation}
	by \cref{eqn:Clifford-Symmetric-Rings:Inclusion}.
	The two inclusions can be written explicitly in terms of $\delC$-symmetric polynomials,
	namely
	\begin{align}
		\SPolC_{(0, k, n)} &\longinto \SPolC_{(0,k,k+1,n)}\\
		\epsC^{(0,k)}_{m}  &\longmapsto \epsC^{(0,k)}_m \notag\\
		\epsC^{(k,n)}_{m}  &\longmapsto \gamma_{k+1, k+2}\,y_{k+1} \, \epsC^{(k+1,n)}_{m-1}
			+ \gamma_{k+1,k}\gamma_{k+1,k+2}\, \epsC^{(k+1,n)}_m \notag
		\\[.5\baselineskip]
		\SPolC_{(0, k+1, n)} &\longinto \SPolC_{(0,k,k+1,n)}\\
		\epsC^{(0,k+1)}_{m}  &\longmapsto
			\epsC^{(0,k)}_m + \epsC^{(0,k)}_{m-1}\cdot \gamma_{m+1,m}\dotsm\gamma_{k+1,k}\,y_{k+1} \notag\\
		\epsC^{(k+1,n)}_{m}  &\longmapsto \epsC^{(k+1,n)}_{m}. \notag
	\end{align}
\end{lemma}
\begin{proof}
	A polynomial from $\SPolC_{(0, k, n)}$
	is contained in all $\ker\delC_i$ for $i∈\{1,\dotsc, k-1, k+1,k+2, \dotsc,n\}$
	by \cref{lem:Elementary-Clifford-Polynomials:Recursively}.
	It is a fortiori contained in all $\ker\delC_i$ for $i∈\{1,\dotsc, k-1, k+2\dotsc,n\}$;
	\ie, in $\ker\delC_1\cap\dotsb\ker\delC_{k-1}\cap\ker\delC_{k+2}\cap\dotsb\cap\ker\delC_n$.

	We know from
	\cref{lem:Clifford-Symmetric-polynomials-are-kernel-of-Demazure,rmk:General-Clifford-Symmetric-polynomials-are-kernel-of-Demazure}
	that this intersection equals $\SPolC_{(0, k, k+1, n)}$,
	which has generators $\epsC^{(0, k, k+1, n)}_m$.
	The coefficients for writing elementary $\delC$-symmetric polynomials from $\SPolC_{(0, k, n)}$
	in terms of those polynomials from $\SPolC_{(0,k,k+1,n)}$
	are given in \cref{example:Clifford-Symmetric-Polynomials:Lambda-Coefficients}.
	This proves the statement.
\end{proof}
\begin{remark}
	\label{rmk:Clifford-Symmetric-Polynomials:Bimodule-structure}
	These inclusions turn $\SPolC_{(0,k,k+1,n)}$ into a $\SPolC_{(0, k, n)}$-$\SPolC_{(0, k+1, n)}$-bimodule.
	In the even case, this bimodule structure corresponds to the one
	described in \autocite[(5.17)]{KL:quantum-group-III}.
	In contrast to the notation in \autocite{KL:quantum-group-III},
	we let $\SPolC_{(0, k, n)}$ act from the left
	and $\SPolC_{(0, k+1, n)}$ from the right.
\end{remark}

\subsubsection{Clifford flag rings}
\label{sec:Clifford-Cohomology-Flag-Variety}
Recall the Clifford Grassmann ring $\HGrC_{(m,n)}$
from \cref{def:Clifford-Grassmann-Cohomology}
and the ordinary cohomology ring of partial flag varieties from \cref{prop:Flag-variety-cohomology}.
The following definition generalises both:
\begin{definition}
	\label{def:clifford-flag-cohomology}
	For dimension vector $\bm{k}=(0=k_0\leq k_1\leq\dotsb\leq k_\ell=n)$,
	let the \term{Clifford flag ring}
	be the quotient
	\begin{equation}
		\HGrC_{\bm{k}} ≔ \quot{\SPolC_{(k_1, \ldots, k_l)}}{(\SPolC_{(0,n)})_+}
	\end{equation}
	by the two-sided ideal generated by non-constant $\delC$-symmetric polynomials.
\end{definition}
\begin{remark}
	The graded dimension computations carried out in \cref{lem:graded-dimension-Grassmann-cohomology,cor:graded-dimension-flag-cohomology}
	remain valid for $\HGrC_{\bm{k}}$ if one replaces $\dim_{q,k}$ by $\rk_{q,\Cl_{I}}$.
\end{remark}
\begin{proposition}
	\label{exm:clifford-flag-cohomology-generalises-grassmannians}
	The quotient $\SPolC_{(0,k,n)}/(\SPolC_{(0,n)})_+$ is isomorphic to the
	Clifford Grassmann ring $\HGrC_{(k,n)}$ defined in \cref{def:Clifford-Grassmann-Cohomology}.
\end{proposition}
\begin{proof}
	We may expand the generators $\epsC^{(0,n)}_k$
	as in \cref{eqn:Clifford-Flag-Varieties:Expansion-Coefficients}.
	In the quotient, we therefore obtain relations
	\[
		\sum_{l=0}^{m}
		\epsC^{(0,k)}_l \ \lambda^{(0,k,n)}_{m,l}\, \epsC^{(k,n)}_{m-l}
		= \delta_{m,0}.
	\]
	Comparing coefficients with those of the identity
	\begin{equation*}
		\sum_{l=0}^{m} (-1)^l\,\hC^{(k,n)}_{m-l}\,\tilde{\kappa}_{k,k+1}\dotsm\tilde{\kappa}_{k, k+l}\epsC^{(k,n)}_l = \delta_{m,0}
	\end{equation*}
	from \cref{prop:Complete-Clifford-Polynomials:Vanishing-Relation}
	implies that
	\begin{equation}
		(-1)^l\,\hC^{(k,n)}_{m-l}\,\tilde{\kappa}_{k,k+1}\dotsm\tilde{\kappa}_{k, k+l} = \epsC^{(0,k)}_{m-l}\,\lambda^{(0,k,n)}_{m,l};
	\end{equation}
	we thus indeed have an isomorphism $\SPolC_{(0,k,n)}/(\SPolC_{(0,n)})_+ \isom \HGrC_{(k, n)}$.
\end{proof}
We have arrived at a super-generalisation
for the cohomology rings of Flag varieties.
In the ordinary set-up, it has been proven in \autocite{KL:quantum-group-III}
that these rings admit an action
by the Kac-Moody 2-category from \autocite{Bru:Kac-Moody}.
The latter has a super-analogue constructed in \autocite{Bru:Super-Kac-Moody}.
A natural question to finish this manuscript is if one can construct an action of the Kac-Moody 2-supercategory from \autocite{Bru:Super-Kac-Moody}.


\printbibliography
\end{document}